\documentclass[12pt]{amsart}
\usepackage{amssymb,amsfonts,amsmath,longtable}
\title{Procesi bundles and Symplectic reflection algebras}
\address{Department
of Mathematics, Northeastern University, Boston, MA 02115, USA}
\email{i.loseu@neu.edu}
\newcommand{\kf}{\mathfrak{k}}
\newcommand{\C}{\mathbb{C}}
\newcommand{\Sp}{\operatorname{Sp}}
\newcommand{\Z}{\mathbb{Z}}
\newcommand{\End}{\operatorname{End}}
\newcommand{\Ext}{\operatorname{Ext}}
\newcommand{\Str}{\mathcal{O}}
\newcommand{\Q}{\mathbb{Q}}
\newcommand{\gr}{\operatorname{gr}}
\newcommand{\A}{\mathcal{A}}
\newcommand{\g}{\mathfrak{g}}
\newcommand{\h}{\mathfrak{h}}
\newcommand{\Weyl}{\mathbf{A}}
\newcommand{\Fi}{\mathbb{F}}
\newcommand{\Vect}{\operatorname{Vect}}
\newcommand{\M}{\mathcal{M}}
\newcommand{\quo}{/\!/}
\newcommand{\red}{/\!/\!/}
\newcommand{\param}{\mathfrak{c}}
\newcommand{\Spec}{\operatorname{Spec}}
\newcommand{\GL}{\operatorname{GL}}
\newcommand{\SL}{\operatorname{SL}}
\newcommand{\Dcal}{\mathcal{D}}
\newtheorem{Thm}{Theorem}[section]
\newtheorem{Prop}[Thm]{Proposition}
\newtheorem{Cor}[Thm]{Corollary}
\newtheorem{Lem}[Thm]{Lemma}
\theoremstyle{definition}
\newtheorem{Ex}[Thm]{Example}
\newtheorem{defi}[Thm]{Definition}
\newtheorem{Rem}[Thm]{Remark}

\newtheorem{Exer}[Thm]{Exercise}
\author{Ivan Losev}
\begin{document}
\begin{abstract}
In this survey we describe an interplay between Procesi bundles on symplectic resolutions
of quotient singularities and Symplectic reflection algebras. Procesi bundles were  constructed by Haiman and, in a greater generality, by Bezrukavnikov and Kaledin. Symplectic reflection algebras are deformations
of skew-group algebras defined in complete generality by Etingof and Ginzburg. We construct and classify
Procesi bundles, prove an isomorphism between spherical Symplectic reflection algebras,
give a proof of wreath Macdonald positivity and of localization theorems for cyclotomic Rational
Cherednik algebras.
\end{abstract}
\maketitle
\section{Introduction}
\subsection{Procesi bundles: Hilbert scheme case}
A Procesi bundle is a vector bundle of rank $n!$ on the Hilbert scheme $\operatorname{Hilb}_n(\C^2)$
whose existence was predicted by Procesi and proved by Haiman, \cite{Haiman}. This bundle was used
by Haiman to prove a famous $n!$ conjecture in Combinatorics that, in turn, settles another
famous conjecture: Schur positivity of Macdonald polynomials.

\subsubsection{n! theorem}
Consider the Vandermond determinant $\Delta(\underline{x})$, where we write $\underline{x}$ for $(x_1,\ldots,x_n)$,
it is given by $\Delta(\underline{x})=\det(x_i^{j-1})_{i,j=1}^n$.
Consider the space $\partial \Delta$ spanned by all partial
derivatives of $\Delta$. This space is graded and carries an action of the symmetric group $\mathfrak{S}_n$
(by permuting the variables $x_1,\ldots,x_n$). A deeper fact is that $\dim \partial \Delta=n!$
(and $\partial \Delta\cong \C \mathfrak{S}_n$ as an $\mathfrak{S}_n$-module), in fact, $\partial \Delta$
coincides with the space of the $\mathfrak{S}_n$-harmonic polynomials, i.e., all polynomials annihilated
by all elements of $\C[\underline{\partial}]^{\mathfrak{S}_n}$ without constant term.

One can ask if there is a two-variable generalization of that fact. We have several two-variable versions of
$\Delta$, one for each Young diagram $\lambda$ with $n$ boxes. Namely, let $(a_1,b_1),\ldots, (a_n,b_n)$
be the coordinates of the boxes in $\lambda$, e.g., $\lambda=(3,2)$ gives pairs $(0,0),(1,0),(2,0), (0,1),(1,1)$.

\begin{picture}(60,17)
\put(2,1){\line(0,1){14}}
\put(12,1){\line(0,1){14}}
\put(22,1){\line(0,1){14}}
\put(32,1){\line(0,1){7}}
\put(2,1){\line(1,0){30}}
\put(2,8){\line(1,0){30}}
\put(2,15){\line(1,0){20}}
\put(3,3){\small $(0,0)$}
\put(13,3){\small $(1,0)$}
\put(23,3){\small $(2,0)$}
\put(3,10){\small $(0,1)$}
\put(13,10){\small $(1,1)$}
\end{picture}

Then set $\Delta_\lambda(\underline{x},\underline{y}):=\det(x_i^{a_j}y_i^{b_j})_{i,j=1}^n$ so that, for
$\lambda=(n)$, we get $\Delta_\lambda(\underline{x},\underline{y})=\Delta(\underline{x})$, for
$\lambda=(1^n)$, we get $\Delta_\lambda(\underline{x},\underline{y})=\Delta(\underline{y})$, while,
for $\lambda=(2,1)$, we get $\Delta_\lambda(\underline{x},\underline{y})=x_1y_2+x_2y_3+x_3y_1
-x_2y_1-x_3y_2-x_1y_3$.

\begin{Thm}[Haiman's n! theorem]\label{Thm_n!}
The space $\partial \Delta_\lambda$ spanned by the partial derivatives of $\Delta_\lambda$
is isomorphic to $\C\mathfrak{S}_n$ as an $\mathfrak{S}_n$-module (where $\mathfrak{S}_n$
acts by permuting the pairs $(x_1,y_1),\ldots, (x_n,y_n)$) and, in particular, has dimension
$n!$.
\end{Thm}

This is a beautiful  result with an elementary statement and an  extremely  involved proof, \cite{Haiman}.
\subsubsection{Macdonald positivity}
Before describing some ideas of the proof that are relevant to the present survey, let
us describe an application to {\it Macdonald polynomials}, particularly important and interesting
symmetric polynomials with coefficients in $\Q(q,t)$. It will be more convenient for us to speak about
representations of $\mathfrak{S}_n$ rather than about symmetric polynomials
(they are related via taking the Frobenius character) and to deal with Haiman's
modified Macdonald polynomials.

\begin{defi}\label{defi_Macdonald}
The modified Macdonald polynomial $\tilde{H}_\lambda$ is the Frobenius character of
a bigraded $\mathfrak{S}_n$-module $P_\lambda:=\bigoplus_{i,j\in \Z} P_\lambda[i,j]$ subject to the following
three conditions
\begin{itemize}
\item[(a)] The class of $P_\lambda\otimes \sum_{i=0}^n (-1)^i \bigwedge^i \C^n[1,0]$
is expressed via the irreducibles $V_\mu$ with $\mu\geqslant \lambda$ (in
the $K_0$ of the bigraded $\mathfrak{S}_n$-modules).
\item[(b)] $P_\lambda\otimes \sum_{i=0}^n (-1)^i \bigwedge^i \C^n[0,1]$
is expressed via the irreducibles $V_\mu$ with $\mu\geqslant \lambda^t$.
\item[(c)] The trivial module $V_{(n)}$ occurs in $P_\lambda$ once and in degree $(0,0)$.
\end{itemize}
Here $\mu\geqslant \lambda$ is the usual dominance order on the set of Young diagrams
(meaning that $\sum_{i=1}^k \mu_i\geqslant \sum_{i=1}^k \lambda_i$) and $\lambda^t$
denotes the transpose of $\lambda$.
\end{defi}

It is not clear from this definition that the representations $P_\lambda$ exist
(the statement on the level of $K_0$ is easier but also non-trivial, this was known before Haiman's
work).

\begin{Thm}[Haiman's Macdonald positivity theorem]\label{Thm:Mac_pos}
A bigraded $\mathfrak{S}_n$-module $P_\lambda$ exists (and is unique) for any $\lambda$.
Moreover, $P_\lambda$ coincides with $\partial \Delta_\lambda$, where $\partial \Delta_\lambda$
is is given the structure of a bigraded $\mathfrak{S}_n$-module as the quotient
of $\C[\underline{\partial}_x,\underline{\partial}_y]$ (via $f\mapsto f\Delta_\lambda$).
\end{Thm}

\subsubsection{Hilbert schemes and Procesi bundles}\label{SSS_HS_PB}
The proofs of the two theorems above given in \cite{Haiman} are based on the geometry of the Hilbert schemes
$\operatorname{Hilb}_n(\C^2)$ of points on $\C^2$. A basic reference for Hilbert schemes
of points on smooth surfaces is \cite{Nakajima_book}. As a set, $\operatorname{Hilb}_n(\C^2)$ consists
of the ideals $J\subset \C[x,y]$ of codimension $n$. It turns out that $\operatorname{Hilb}_n(\C^2)$
is a smooth algebraic variety of dimension $2n$. It admits a morphism (called the Hilbert-Chow map)
to the variety $\operatorname{Sym}_n(\C^2)$ of the unordered $n$-tuples of points in $\C^2$:
to an ideal $J$ one assigns its support, where points are counted with multiplicities.
Of course, $\operatorname{Sym}_n(\C^2)$ is nothing else but the quotient space
$(\C^2)^{\oplus n}/\mathfrak{S}_n$, the affine algebraic variety whose algebra of functions
is the invariant algebra $\C[\underline{x},\underline{y}]^{\mathfrak{S}_n}$.
The Hilbert-Chow map is  a resolution of singularities.

Note that the two-dimensional torus $(\C^\times)^2$ acts on $\operatorname{Hilb}_n(\C^2)$ and on
$\operatorname{Sym}_n(\C^2)$, the action is induced from the following action on $\C^2$:
$(t,s).(a,b):=(t^{-1}a,s^{-1}b)$. The fixed points of this action on $\operatorname{Hilb}_n(\C^2)$
correspond to the monomial ideals (=ideals generated by monomials) in $\C[x,y]$, they are in
a natural one-to-one correspondence with Young diagrams (as before we fill a Young diagram
with monomials and take the ideal spanned by all monomials that do not appear in the diagram).
Let $z_\lambda$ denote the fixed point corresponding to a Young seminar $\lambda$.

Following Haiman, consider the isospectral Hilbert scheme $I_n$, the reduced Cartesian product
$\C^{2n}\times_{\operatorname{Sym}_n(\C^2)}\operatorname{Hilb}_n(\C^2)$, let $\eta: I_n\rightarrow
\operatorname{Hilb}_n(\C^2)$ be the natural morphism. It is finite of generic degree $n!$.
The main technical result of Haiman, \cite{Haiman}, is that $I_n$ is Cohen-Macaulay and Gorenstein.
So $\mathcal{P}:=\eta_*\mathcal{O}_{I_n}$ is a rank $n!$ vector bundle on $\operatorname{Hilb}_n(\C^2)$
(the Procesi bundle). By the construction, each fiber of this bundle carries an algebra structure that
is a quotient of $\C[\underline{x},\underline{y}]$. Let us write $\mathcal{P}_\lambda$
for the fiber of $\mathcal{P}$ in $z_\lambda$, this is an  algebra that carries a
natural bi-grading because the bundle
$\mathcal{P}$ is $(\C^\times)^2$-equivariant by the construction. On the other hand, $\partial \Delta_\lambda$
is a quotient of $\C[\underline{\partial}_{x},\underline{\partial}_y]$ by an ideal and so is also
an algebra. The latter algebra is bigraded. Haiman has shown that $\mathcal{P}_\lambda\cong \partial\Delta_\lambda$,
an isomorphism of bigraded algebras. This finishes the proof of  Theorem \ref{Thm_n!}.

Let us proceed to Theorem \ref{Thm:Mac_pos}. The class in (a) of Definition \ref{defi_Macdonald} is
that of the fiber at $z_\lambda$ of the Koszul complex
\begin{equation}\label{eq:Koszul}\mathcal{P}\leftarrow \h_{\underline{x}}\otimes\mathcal{P}\leftarrow \Lambda^2 \h_{\underline{x}}\otimes\mathcal{P} \leftarrow\ldots,\end{equation}
where $\h_{\underline{x}}$ is the span of $x_1,\ldots,x_n$ viewed as endomorphisms of $\mathcal{P}$.
Haiman has shown that $I_n$ is flat over $\operatorname{Spec}(\C[\underline{x}])$
(with morphism $I_n\rightarrow \operatorname{Spec}(\C[\underline{x},\underline{y}])\rightarrow
\operatorname{Spec}(\C[\underline{x}])$). It follows that (\ref{eq:Koszul}) is a resolution of
$\mathcal{P}/\h_{\underline{x}}\mathcal{P}$. Now (a) follows from the claim that, for any Young
diagram $\lambda$, the support of the isotypic $V_\lambda$-component in
$\mathcal{P}/\h_{\underline{x}}\mathcal{P}$ contains only points $z_\mu$ with $\mu\leqslant \lambda$.
This was checked by Haiman. Part (b) is analogous, while (c) follows directly from the construction.

There are several other proofs of Theorem \ref{Thm:Mac_pos} available. Two of them use the geometry
of Hilbert schemes and Procesi bundle, \cite{Gordon_Procesi,BF}. We will discuss (a somewhat modified)
approach from \cite{BF} in detail in Section \ref{S_Macdonald}.

\subsection{Quotient singularities and symplectic resolutions}
\subsubsection{Setting}
Let $V$ be a finite dimensional vector space over $\C$ equipped with a symplectic form $\Omega\in \bigwedge^2 V^*$.
Let $\Gamma$ be a finite subgroup of $\Sp(V)$. The invariant algebra $\C[V]^\Gamma$ is a graded Poisson algebra
(as a subalgebra of $\C[V]$) and the corresponding quotient $V/\Gamma=\operatorname{Spec}(\C[V]^\Gamma)$ is a singular Poisson affine variety
that comes with a $\C^\times$-action induced from the action on $V$ by dilations: $t.v:=t^{-1}v$.

\subsubsection{Symplectic resolutions}
One can ask if there is a resolution of singularities of $V/\Gamma$ that is nicely compatible with  the Poisson
structure (and with the $\C^\times$-action). This compatibility is formalized in the notion of a (conical) symplectic resolution.

\begin{defi}\label{defi_sympl_res}
Let $X_0$ be a singular normal affine Poisson variety such that
the regular locus $X_0^{reg}$ is symplectic. We say that a variety $X$ equipped with a morphism $\rho: X\rightarrow X_0$ is a {\it symplectic resolution} of $X_0$ if $X$ is symplectic (with form $\omega$), $\rho$ is a resolution of singularities and $\rho: \rho^{-1}(X_0^{reg})\rightarrow X_0^{reg}$ is a symplectomorphism.
\end{defi}

\begin{defi}\label{defi_conic_res}
Further, suppose that $X_0$ is equipped with a $\C^\times$-action such that
\begin{itemize}
\item the corresponding grading $\C[X_0]=\bigoplus_{i\in \Z}\C[X_0]_i$ is positive, meaning that $\C[X_0]_i=\{0\}$ for $i<0$ and $\C[X]_0=\C$,
\item and the Poisson bracket on $\C[X_0]$ has degree $-d$ for some fixed
$d\in \Z_{>0}$: $\{\C[X_0]_i,\C[X_0]_j\}\subset \C[X_0]_{i+j-d}$
for all $i,j$.\end{itemize}
We say that a symplectic resolution $X$ is {\it conical} if it is equipped with a $\C^\times$-action
making $\rho$ equivariant.
\end{defi}

The variety $X_0=V/\Gamma$ is normal and carries a natural $\C^\times$-action (by dilations)
as in Definition \ref{defi_conic_res} with $d=2$. Also note that the $\C^\times$-action on
$X$ automatically satisfies $t.\omega=t^{-d}\omega$. Finally, note that, under assumptions of Definition \ref{defi_sympl_res}, we have $\C[X]=\C[X_0]$.

\subsubsection{Symplectic resolutions for quotient singularities}
In the previous subsection, we have already seen an example of $(V,\Gamma)$ such that $V/\Gamma$ admits a conical symplectic resolution: $V=(\C^{2})^{\oplus n}, \Gamma=\mathfrak{S}_n$, in this case one can take $X=\operatorname{Hilb}_n(\C^2)$ together with the Hilbert-Chow morphism, see \cite[Section 1]{Nakajima_book}.

There are other examples as well. Let $\Gamma_1$ be a finite subgroup of $\operatorname{SL}_2(\C)$, such subgroups
are classified (up to conjugacy) by Dynkin diagrams of ADE types. Say, the cyclic subgroup $\Z/(\ell+1)\Z$ (embedded
into $\operatorname{SL}_2(\C)$ via $n\mapsto \operatorname{diag}(\eta^n,\eta^{-n})$ with $\eta:=\exp(2\pi\sqrt{-1}/(\ell+1))$) corresponds to the diagram $A_\ell$. The quotient singularity
$\C^2/\Gamma_1$ admits a distinguished {\it minimal} resolution to be denoted by $\widetilde{\C^2/\Gamma_1}$.
This resolution is conical symplectic, see, e.g., \cite[Section 4.1]{Nakajima_book}.

The examples of $\mathfrak{S}_n$ and $\Gamma_1$ can be ``joined'' together. Consider the group
$\Gamma_n:=\mathfrak{S}_n\ltimes \Gamma_1^n$. It acts on $V_n:=(\C^2)^{\oplus n}$: the symmetric
group permutes the summands, while each copy of $\Gamma_1$ acts on its own summand. The
quotient singularities $V_n/\Gamma_n$ admit symplectic resolutions. For example, one can take
$X:=\operatorname{Hilb}_n(\widetilde{\C^2/\Gamma_1})$. But there are other conical symplectic
resolutions of $V_n/\Gamma_n$, conjecturally, they are all constructed as Nakajima quiver varieties,
we will recall the construction of these varieties in \ref{SSS_quiv_var_constr}.

To finish this section, let us point out that, presently,  two more pairs $(V,\Gamma)$
such that $V/\Gamma$ admits a symplectic resolutions are known, see \cite{Bellamy,BSched}. In this paper,
we are not interested in these cases.

\subsection{Procesi bundles: general case}
\subsubsection{Smash-product algebra}
One nice feature of quotient singularities $V/\Gamma$ is that they always have a nice resolution of singularities
which is, however, noncommutative algebraic rather than algebro-geometric: the smash-product algebra $\C[V]\#\Gamma$.
As a vector space, $\C[V]\#\Gamma$ is the tensor product $\C[V]\otimes \C\Gamma$, and the product on
$\C[V]\#\Gamma$ is given by
$$(f_1\otimes \gamma_1)\cdot (f_2\otimes \gamma_2)=f_1\gamma_1(f_2)\otimes \gamma_1\gamma_2, f_1,f_2\in \C[V], \gamma_1,\gamma_2\in \Gamma,$$
where $\gamma_1(f_2)$ denotes the image of $f_2$ under the action of $\gamma_1$. The definition is arranged
in such a way that a $\C[V]\#\Gamma$-module is the same thing as a $\Gamma$-equivariant $\C[V]$-module.
Note that the algebra $\C[V]\#\Gamma$ is graded, for a homogeneous element $f$ of degree $n$, the degree of
$f\otimes \gamma$ is $n$.

Let us explain what we mean when we say that $\C[V]\#\Gamma$ is a resolution of singularities of $V/\Gamma$.
Note that $\C[V]^\Gamma$ can be recovered from $\C[V]\#\Gamma$ in two different but related ways. First,
we have an embedding $\C[V]^\Gamma\hookrightarrow \C[V]\#\Gamma$ given by $f\mapsto f\otimes 1$. The image lies
in the center (this is easy) and actually coincides with the center (a bit harder). Second, consider the element
$e\in \C\Gamma, e:=|\Gamma|^{-1}\sum_{\gamma\in \Gamma}\gamma$, the averaging idempotent. Consider the subspace
$e(\C[V]\#\Gamma)e\subset \C[V]\#\Gamma$. It is obviously closed under multiplication, and $e$ is a unit
with respect to the multiplication there. So $e(\C[V]\#\Gamma)e$ is an algebra, to be called the spherical
subalgebra of $\C[V]\#\Gamma$. It is isomorphic to $\C[V]^{\Gamma}$, an isomorphism is given by
$f\mapsto ef$.

Thanks to the realization of $\C[V]^{\Gamma}$ as a spherical subalgebra, we can consider the functor
$M\mapsto eM: \C[V]\#\Gamma\operatorname{-mod}\rightarrow \C[V]^{\Gamma}\operatorname{-mod}$
(an analog of the morphism $\rho$). Note that the algebra $\C[V]\#\Gamma$ has finite homological
dimension (because $\C[V]$ does) and so is ``smooth''. The algebra $\C[V]\#\Gamma$ is finite over
$\C[V]^{\Gamma}$ which can be thought as an analog of $\rho$ being proper. Also, after replacing
$\C[V]\#\Gamma, \C[V]^{\Gamma}$ with sheaves $\mathcal{O}_{V^{reg}}\#\Gamma, \mathcal{O}_{V^{reg}/\Gamma}$
on $V^{reg}/\Gamma$, where
\begin{equation}\label{eq:reg_locus} V^{reg}:=\{v\in V| \Gamma_v=\{1\}\},
\end{equation}
the functor $M\mapsto eM$ becomes a category equivalence. This is an analog of $\rho$
being birational.

\subsubsection{Procesi bundle: an axiomatic description}
Now let $X$ be a conical symplectic resolution of $V/\Gamma$. We want to relate $X$ to $\C[V]\#\Gamma$.

\begin{defi}
A Procesi bundle $\mathcal{P}$ on $X$ is a $\C^\times$-equivariant vector bundle on $X$
together with an isomorphism $\End_{\Str_X}(\mathcal{P})\xrightarrow{\sim} \C[V]\#\Gamma$
of graded algebras over $\C[X]=\C[V]^\Gamma$ such that $\Ext^i(\mathcal{P},\mathcal{P})=0$
for $i>0$.
\end{defi}

Note that the isomorphism $\End_{\Str_X}(\mathcal{P})\xrightarrow{\sim} \C[V]\#\Gamma$
gives a fiberwise $\Gamma$-action on $\mathcal{P}$. The invariant sheaf $e\mathcal{P}$
is a vector bundle of rank $1$. We say that $\mathcal{P}$ is {\it normalized}
if $e\mathcal{P}=\mathcal{O}_X$ (as a $\C^\times$-equivariant vector bundle).
We can normalize an arbitrary Procesi bundle by tensoring it with $(e\mathcal{P})^{*}$. Below we only consider
normalized Procesi bundles.

In particular, Haiman's Procesi bundle on $X=\operatorname{Hilb}_n(\C^2)$ fits the definition, this is essentially
a part of \cite[Theorem 5.3.2]{Haiman_CDM}
(and is normalized). The existence of a Procesi bundle on a general $X$ was proved by Bezrukavnikov and Kaledin in
\cite{BK_Procesi}.  We will see that the number of different Procesi bundles on
a symplectic resolution of $\C^{2n}/\Gamma_n$ equals $2|W|$ if $n>1$, where $W$
is the Weyl group of the Dynkin diagram corresponding to $\Gamma_1$. For example,
when $\Gamma_1=\Z/\ell\Z$, we get $W=\mathfrak{S}_\ell$ and so the number of different
Procesi bundles is $2\ell!$.

\subsection{Symplectic reflection algebras}
\subsubsection{Definition}
Symplectic reflection algebras were introduced by Etingof and Ginzburg in \cite{EG}. Those are filtered deformations
of $\C[V]\#\Gamma$.

By a symplectic reflection in $\Gamma$ one means an element $\gamma$ with
$\operatorname{rk}(\gamma-1_V)=2$. Note that the rank has to be even: the image of $\gamma-1_V$
is a symplectic subspace of $V$. By $S$ we denote the set of all symplectic reflections in $\Gamma$,
it is a union of conjugacy classes, $S=\sqcup_{i=1}^r S_i$. Now pick $t\in \C$
and $c=(c_1,\ldots,c_r)\in \C^r$. We define the algebra $H_{t,c}$ as the quotient
of $T(V)\#\Gamma$ by the relations
\begin{equation}\label{eq:SRA_reln} u\otimes v-v\otimes u= t\Omega(u,v)+\sum_{i=1}^r c_i\sum_{s\in S_i} \Omega(\pi_su,\pi_sv)s, u,v\in V.
\end{equation}
Here we write $\pi_s$ for the projection  $V\twoheadrightarrow\operatorname{im}(s-1_V)$ corresponding to the decomposition
$V=\operatorname{im}(s-1_V)\oplus \ker(s-1_V)$.

As Etingof and Ginzburg checked in \cite{EG}, the algebra $H_{t,c}$ satisfies the PBW property:
if we filter $H_{t,c}$ by setting $\deg\Gamma=0, \deg V=1$, then $\gr H_{t,c}=\C[V]\#\Gamma$
(here we identify $V$ with $V^*$ by means of $\Omega$ so that $\C[V]\cong S(V)$). Moreover,
we will see that $H_{t,c}$ satisfies a certain universality property so this deformation
of $\C[V]\#\Gamma$ is forced on us, in a way.

\subsubsection{Connection to Procesi bundles}
It may seem that Symplectic reflection algebras and Procesi bundles are not related. This is not so.
It turns out that the algebra $H_{t,c}$ is the endomorphism algebra of a suitable understood
deformation of a Procesi bundle $\mathcal{P}$. This connection is beneficial for studying both.
On the Procesi side, it allows to classify Procesi bundles, \cite{Procesi}, and prove
the Macdonald positivity in the case of groups $\Gamma_n$ with $\Gamma_1=\Z/\ell \Z$,
\cite{BF}. On the symplectic reflection side, it allows to relate the algebras
$H_{t,c}$ to quantized Nakajima quiver varieties, see \cite{EGGO,quant} and references
therein, which then allows to study the representation theory of $H_{t,c}$ (\cite{BL})
and to prove versions of Beilinson-Bernstein localization theorems, \cite{GL,cher_ab_loc}.
Connections between Procesi bundles and Symplectic reflection algebras is a subject of this survey.

\subsection{Notation and conventions} Let us list some notation used in the paper.

{\it Quantizations and deformations}.
We use the following conventions for quantizations. For a Poisson algebra $A$, we write $\A_\hbar$ for its
formal quantization. When $A$ is graded, we write $\A$ for its filtered quantization. The notation
$\mathcal{D}_\hbar$ is usually used for a formal quantization of a variety, while $\mathcal{D}$
usually denotes a filtered quantization.

When $X$ is a conical symplectic resolution of singularities, we write $\tilde{X}$
for its universal conical deformation (over $H^2_{DR}(X)$) and $\tilde{\mathcal{D}}_\hbar$
stands for the canonical quantization of $\tilde{X}$.

{\it Symplectic reflection groups and algebras}. We write $\Gamma_1$ for a finite subgroup
of $\SL_2(\C)$ and $\Gamma_n$ for the semidirect product $\mathfrak{S}_n\ltimes \Gamma_1^n$.
This semi-direct product acts on $V_n:=\C^{2n}$. In the case when $\Gamma_1=\{1\}$, we usually
write $V_n$ for $T^*\C^{n-1}$, where $\C^{n-1}$ is the reflection representation of $\mathfrak{S}_n$.

For a group $\Gamma$ acting on a space $V$ by linear symplectomorphisms, by $S$ we denote the set
of symplectic reflections in $\Gamma$. By $e$ we denote the averaging idempotent of $\Gamma$.
By ${\bf H}$ we denote the universal symplectic reflection algebra of $(V,\Gamma)$. Its specializations
are denoted by $H_{t,c}$.

{\it Quotients and reductions}. Let $G$ be a group acting on a variety $X$.
If $G$ is finite and $X$ is quasi-projective, then the quotient is denoted by $X/G$ (note
that this quotient may fail to exist when $X$ is not quasi-projective).
If $G$ is reductive and $X$ is affine, then $X\quo G$ stands for the categorical quotient.
A GIT quotient of $X$ under the $G$-action with stability condition $\theta$
is denoted by $X\quo^\theta G$.

When $X$ is Poisson, and the $G$-action is Hamiltonian, we write $X\red_\lambda G$ for $\mu^{-1}(\lambda)\quo G$
and $X\red_\lambda^\theta G$ for $\mu^{-1}(\lambda)\quo^\theta G$.

{\it Miscellaneous notation}.

\setlongtables
\begin{longtable}{p{2.5cm} p{12.5cm}}
$\widehat{\otimes}$&the completed tensor product of complete topological vector spaces/ modules.\\
$(a_1,\ldots,a_k)$& the two-sided ideal in an associative algebra generated by  elements $a_1,\ldots,a_k$.\\
 $A^{\wedge_\chi}$&
the completion of a commutative (or ``almost commutative'') algebra $A$ with respect to the maximal ideal
of a point $\chi\in \Spec(A)$.\\
$\Weyl(V)$& the Weyl algebra of a symplectic vector space $V$.\\
$D(X)$& the algebra of differential operators on a smooth variety $X$.\\
$\Fi_q$& the finite field with $q$ elements.\\
$\gr \A$& the associated graded vector space of a filtered
vector space $\A$.\\
$H^i_{DR}(X)$& the $i$th De Rham cohomology of $X$ with coefficients in $\C$. \\
$\Str_X$& the structure sheaf of a scheme $X$.\\
$R_\hbar(\A)$&$:=\bigoplus_{i\in
\mathbb{Z}}\hbar^i \A_{\leqslant i}$ :the Rees $\C[\hbar]$-module of a filtered
vector space $\A$.\\
$\mathfrak{S}_n$& the symmetric group in $n$ letters.\\
$S(V)$& the symmetric algebra of a vector space $V$.\\
$\operatorname{Sp}(V)$& the symplectic linear group of a symplectic vector space $V$.\\
$\Gamma(\mathcal{S})$& global sections of a sheaf $\mathcal{S}$.
\end{longtable}

{\bf Acknowledgements}: This survey is a greatly expanded version of
lectures I gave at Northwestern in May 2012.
I would like to thank Roman Bezrukavnikov and Iain Gordon for numerous
stimulating discussions. My work was supported by the NSF  under Grant  DMS-1161584.

\section{Quantizations}\label{S_quantization}
In this section we review the quantization formalism. In Section \ref{SS_quant_alg}
we discuss  quantizations of
Poisson algebras.  There are two formalisms here: filtered quantizations and formal quantizations. We introduce both of them, discuss a relation between them and then give examples.

Then, in Section \ref{SS_sheaf_quant}, we proceed to quantizations of non-necessarily
affine Poisson algebraic varieties. Here we quantize the structure sheaf. We explain that to quantize an
affine variety is the same thing as to quantize its algebra of functions. Then we mention a theorem
of Bezrukavnikov and Kaledin classifying quantizations of symplectic varieties under certain cohomology
vanishing conditions.

After that we proceed to modules over quantizations. We define coherent and quasi-coherent sheaves of modules
and outline their basic properties. For a coherent sheaf of modules, we define its support. Then we discuss
global section and localization functors and their derived analogs.

We finish this system by discussing Frobenius constant quantizations in positive characteristic.

\subsection{Algebra level}\label{SS_quant_alg}
Here we will review formalisms of quantizations of Poisson algebras. Let $A$ be a Poisson algebra
(commutative, associative and with a unit).

\subsubsection{Formal quantizations}
First, let us discuss formal quantizations. By a formal quantization of $A$ we mean an associative $\C[[\hbar]]$-algebra
$\A_\hbar$ equipped with an algebra isomorphism $\pi:\A_\hbar/(\hbar)\xrightarrow{\sim}A$ such that
\begin{itemize}
\item[(i)]$\A_\hbar\cong A[[\hbar]]$ as a $\C[[\hbar]]$-module and this isomorphism intertwines $\pi$
and the natural projection $A[[\hbar]]\rightarrow A$.
\item[(ii)] We have $\pi(\frac{1}{\hbar}[a,b])\equiv \{\pi(a),\pi(b)\}$ (note that $\pi([a,b])=[\pi(a),\pi(b)]=0$
and so $\frac{1}{\hbar}[a,b]$ makes sense).
\end{itemize}
Condition (i) can be stated equivalently as follows: $\A_\hbar$ is flat over $\C[[\hbar]]$ and is complete and separated
in the $\hbar$-adic topology.

\subsubsection{Filtered quantizations}\label{SSS_filt_quant_alg}
Second, we will need the formalism of filtered quantizations. Suppose that $A$ is equipped with an algebra grading,
$A=\bigoplus_{i\in \Z}A_i$, that is compatible with $\{\cdot,\cdot\}$ in the following way: $\{A_i,A_j\}\subset A_{i+j-1}$.

First, we consider the case when the grading on $A$ is non-negative: $A_i=\{0\}$ for $i<0$. Then, by a filtered quantization
of $A$ one means a $\Z_{\geqslant 0}$-filtered algebra $\A=\bigcup_{i\geqslant 0}\A_{\leqslant i}$
together with a graded algebra isomorphism $\pi:\gr \A\xrightarrow{\sim}A$ such that, for $a\in \A_{\leqslant i},
b\in \A_{\leqslant j}$, one has $\{\pi(a+\A_{\leqslant i-1}),\pi(b+\A_{\leqslant j-1})\}=
\pi([a,b]+\A_{\leqslant i+j-2})$ (note that $[a,b]\in \A_{\leqslant i+j-1}$ because $\gr\A$
is commutative).

\subsubsection{Relation between the two formalisms}\label{SSS_form_rel}
Let us explain a connection between the two formalisms (that will also motivate the definition of  a filtered
quantization in the case when the grading on $A$ has negative components). Take a filtered
quantization $\A$ of $A$. Form the {\it Rees algebra} $R_\hbar(\A):=\bigoplus_{i\geqslant 0} \A_{\leqslant i}\hbar^i$
that is equipped with a graded algebra structure as a subalgebra in $\A[\hbar]$. We have natural identifications $R_\hbar(\A)/(\hbar)\cong A, R_\hbar(\A)/(\hbar-1)\cong \A$. The $\hbar$-adic
completion $R_\hbar(\A)^{\wedge_\hbar}:=\varprojlim_{n\rightarrow +\infty} R_\hbar(\A)/(\hbar^n)$ satisfies
(i) and (ii) and so is a formal quantization of $A$. Moreover, it comes with a $\C^\times$-action by algebra
automorphisms such that $t.\hbar=t\hbar, t\in \C^\times$: the action is given by $t.\sum_{i=0}^{+\infty} a_i\hbar^i:=
\sum_{i=0}^{+\infty} t^ia_i\hbar^i$. Clearly, the induced action on $A$ coincides with the action coming from the grading.
Conversely, suppose we have a formal quantization $\A_\hbar$ of $A$ equipped with a $\C^\times$-action by algebra
automorphisms such that $t.\hbar=t\hbar$ and the epimorphism $\pi$ is $\C^\times$-equivariant. Assume, further,
that the action is pro-rational meaning that it is rational on all quotients $\A_\hbar/(\hbar^n)$. Consider the subspace
$\A_{\hbar,fin}\subset \A_\hbar$ consisting of all {\it $\C^\times$-finite} elements, i.e., those elements
that are contained in some finite dimensional $\C^\times$-stable subspace. This is a $\C^\times$-stable
$\C[\hbar]$-subalgebra of $\A_{\hbar}$. It is easy to see that $\pi$ induces an isomorphism
$\A_{\hbar,fin}/(\hbar)\cong A$. Then $\A:=\A_{\hbar,fin}/(\hbar-1)$ is a filtered quantization.

\subsubsection{Filtered quantizations, general case}
Let us proceed to the case when the grading on $A$ is not necessarily non-negative. We can still consider
a formal quantization $\A_\hbar$ with a $\C^\times$-action as above, the subalgebra $\A_{\hbar,fin}\subset
\A_\hbar$ and the quotient $\A:=\A_{\hbar,fin}/(\hbar-1)$. It is still a filtered quantization
in the sense explained above (with the difference that now we have a $\Z$-filtration rather than a $\Z_{\geqslant 0}$-filtration) but, moreover, the filtration on $\A$ has a special property:
it is complete and separated meaning that a natural homomorphism $\A\rightarrow \varprojlim_{n\rightarrow -\infty}
\A/\A_{\leqslant n}$ is an isomorphism. By a filtered quantization of $A$ we now mean a $\Z$-filtered
algebra $\A$, where the filtration is complete and separated, together with an isomorphism
$\pi:\gr\A\xrightarrow{\sim}A$ of graded algebras such that $\{\pi(a+\A_{\leqslant i-1}),\pi(b+\A_{\leqslant j-1})\}=\pi([a,b]+\A_{\leqslant i+j-2})$.

Our conclusion is that the following two formalisms are equivalent: filtered quantizations and formal quantizations
with a $\C^\times$-action. To get from a filtered quantization $\A$ to a formal one, one takes $R_\hbar(\A)^{\wedge_\hbar}$. To get from a formal quantization $\A_\hbar$ to  a filtered one,
one takes $\A_{\hbar,fin}/(\hbar-1)$.

\subsubsection{Examples}
Let us proceed to examples. In examples, one usually gets $\Z_{\geqslant 0}$-filtered quantizations,
more general $\Z$-filtered or formal quantizations arise in various constructions (such as (micro)localization
or completion).

\begin{Ex}\label{Ex:un_env}
Let $\g$ be a Lie algebra. Then, by the PBW theorem, the universal enveloping algebra $U(\g)$
is a filtered quantization of $S(\g)$.
\end{Ex}

\begin{Ex}\label{Ex:diff_op}
Let $Y$ be an affine algebraic variety. The algebra $D(Y)$ of linear differential operators on $Y$
(together with the filtration by the order of  differential operators) is a filtered quantization of
$\C[T^*Y]$.
\end{Ex}

\begin{Rem}\label{Rem:degree} Often one needs to deal with a more general compatibility condition
between the grading and the bracket: $\{A_i,A_j\}\subset A_{i+j-d}$ for some fixed $d>0$. In this
case, one can modify the definitions of  formal and filtered quantizations. Namely, in the definition
of a formal quantization one can require that $[\A_\hbar,\A_\hbar]\subset \hbar^d \A_{\hbar}$ and
$\pi(\frac{1}{\hbar^d}[a,b])=\{\pi(a),\pi(b)\}$. The definition of a filtered quantization can be
modified similarly.
\end{Rem}

\begin{Ex}\label{Ex:SRA}
Let $V$ be a symplectic vector space and $\Gamma\in \Sp(V)$ be a finite group. Consider $A=S(V)^\Gamma$
with Poisson bracket $\{\cdot,\cdot\}$ restricted from $S(V)$.
In the notation of Remark \ref{Rem:degree}, $d=2$. As was essentially checked in \cite{EG}, the spherical subalgebra
$eH_{1,c}e$ (with a filtration restricted from $H_{1,c}$) is a quantization of $S(V)^\Gamma$ for any parameter $c$
When $\Gamma=\{1_V\}$, we recover the usual Weyl algebra, $\Weyl(V)$, of $V$.

To check that $eH_{1,c}e$ is a quantization carefully we note that the proof of Theorem 1.6 in {\it loc.cit.}
shows that the bracket on $S(V)^{\Gamma}$ coming from the filtered deformation $eH_{1,c}e$
coincides with $a\{\cdot,\cdot\}$, where $a$ is a nonzero number independent of $c$.
Then we notice that for $c=0$ we get $eH_{1,c}e=\Weyl(V)^\Gamma$ and so $a=1$.
\end{Ex}

In fact, in the previous example we often can also achieve $d=1$. Namely, if $-1_V\in \Gamma$, then all degrees
in $S(V)^\Gamma$ are even and so we can consider the grading $A=\bigoplus_{i\geqslant 0}A_i$ with
$A_i$ consisting of all homogeneous elements with usual degree $2i$. We introduce a filtration on
$e H_{1,c}e$ in a similar way (this filtration is not restricted from $H_{1,c}$). Then we get a filtered
quantization according to our original definition. When $\Gamma=\Gamma_n$, we only have $-1_V\not\in\Gamma$ if
$\Gamma=\Z/\ell\Z$ for odd $\ell$. For $\Gamma=\Z/\ell\Z$ (and any $\ell$), $V$ splits as
$\h\oplus\h^*$, where $\h=\C^n$. We can grade $S(V)$ by setting $\deg \h^*=0, \deg \h=1$
and take the induced grading on $S(V)^\Gamma$ and the induced filtration on $H_{1,c}$.

\subsection{Sheaf level}\label{SS_sheaf_quant}
Above, we were dealing with Poisson algebras or, basically equivalently, with affine Poisson algebraic
varieties. Now we are going to consider general Poisson varieties (or schemes). Recall that
by a Poisson variety one means a variety $X$ such that the structure sheaf $\Str_X$ is equipped
with a Poisson bracket (meaning that all algebras of sections are Poisson and the restriction homomorphisms
respect the Poisson brackets). In this case a quantization of $X$ will be a (formal or filtered)
quantization of $\Str_X$ in the sense explained below in this section.

\subsubsection{Formal quantizations}\label{SSS_form_quant_sheaf}
We start with a formal setting. A  quantization $\mathcal{D}_\hbar$ of $X$ is a sheaf
of $\C[[\hbar]]$-algebras on $X$ together with an isomorphism $\pi:\mathcal{D}_\hbar/(\hbar)\xrightarrow{\sim}
\Str_X$ such that
\begin{itemize}
\item[(a)]  $\mathcal{D}_\hbar$ is flat over $\C[[\hbar]]$ (equivalently, there are no nonzero local sections
annihilated by $\hbar$) and complete and separated in the $\hbar$-adic topology (meaning that
$\mathcal{D}_\hbar\xrightarrow{\sim}\varprojlim_{n\rightarrow+\infty} \mathcal{D}_\hbar/(\hbar^n)$).
\item[(b)] $\pi(\frac{1}{\hbar}[a,b])=\{\pi(a),\pi(b)\}$ for any local sections $a,b$ of $\mathcal{D}_\hbar$.
\end{itemize}

\subsubsection{Motivation: star-products}
The origins of this definition are in the deformation quantization introduced in \cite{BFFLS}.
Let us adopt this definition to our situation. Let $A$ be a Poisson algebra. By a star-product
on $A$ one means a bilinear map $*:A\otimes A\rightarrow A[[\hbar]]$ subject to the following conditions:
\begin{enumerate}
\item The $\C[[\hbar]]$-bilinear extension of $*$ to $A[[\hbar]]$ is associative and $1\in A$ is a unit.
\item $a*b\equiv ab \operatorname{mod} \hbar A[[\hbar]]$, $a*b-b*a\equiv\hbar\{a,b\} \operatorname{mod} \hbar^2 A[[\hbar]]$.
\end{enumerate}
Of course, $A[[\hbar]]$ together with $*$ is a formal quantization of $A$ in the sense of the previous section.
Conversely, any formal quantization $\A_\hbar$ is isomorphic to $(A[[\hbar]],*)$.

Traditionally, one imposes an additional restriction on $*$: the locality axiom that requires that
the coefficients $D_i$ in the $\hbar$-adic expansion of $*$ ($a*b=\sum_{i=0}^{\infty} D_i(a,b)\hbar^i$)
are bidifferential operators. If $*$ is local, then it naturally extends to any localization
$A[a^{-1}]$. So, if $A=\C[X]$ for $X$ affine, then  a local star-product defines a quantization
of $\Str_X$.

Let us provide an example of a local star-product. Consider $A=\C[\underline{x},\underline{y}]$
with standard Poisson bracket: $\{x_i,x_j\}=\{y_i,y_j\}, \{y_i,x_j\}=\delta_{ij}$. Then set
\begin{equation}\label{eq:star} f*g=m\circ \exp(\hbar\sum_{i=1}^n \partial_{y_i}\otimes \partial_{x_i})f\otimes g,\end{equation}
where $\mu:A\otimes A\rightarrow A$ is the usual commutative product. For example, we have $x_i*x_j=x_ix_j,y_i*y_j=y_iy_j,
x_i*y_j=x_iy_j, y_j*x_i=x_iy_j+\hbar \delta_{ij}$. In this case, $A[\hbar]$ is closed with respect to $*$
and is identified with $R_\hbar(D(\C^n))$.

\subsubsection{Algebra vs sheaf setting in the affine case}
It turns out that any formal quantization $\A_\hbar$ of $\C[X]$ for an affine variety $X$ defines a quantization of
$X$. The reason is that we can localize elements of $\C[X]$ in $\A_\hbar$. The construction is as follows.
Pick $f\in \C[X]$ and lift it to $\hat{f}\in \A_\hbar$. The operator $\operatorname{ad}\hat{f}$ is nilpotent
in $\A_\hbar/(\hbar^n)$ for any $n$ and so the set $\{\hat{f}^n\}\subset \A_\hbar/(\hbar^n)$ satisfies the Ore conditions,
hence the localization $\A_\hbar/(\hbar^n)[\hat{f}^{-1}]$ makes sense. It is easy to see that these localizations
do not depend on the choice of the lift $\hat{f}$ and form an inverse system. We set $\A_\hbar[f^{-1}]:=\varprojlim_{n\rightarrow+\infty} \A_\hbar/(\hbar^n)[\hat{f}^{-1}]$.

\begin{Exer}\label{Exer:microloc}
Check that there is a unique sheaf $\mathcal{D}_\hbar$ in the Zariski topology on $X$ such that
$\mathcal{D}_\hbar(X_f)=\A_\hbar[f^{-1}]$ for any $f\in \C[X]$ and that this sheaf is a quantization of $X$.
\end{Exer}

So we see that there is a natural bijection between the quantizations of $X$ and of $\C[X]$
(to get from a quantization of $X$ to that of $\C[X]$ we just take the global sections).
Thanks to this, we can view a quantization of a general variety $X$ as glued from affine pieces.

\subsubsection{Filtered quantizations}\label{SSS_filt_quant_sheaf}
Let us proceed to the filtered setting. Suppose that $X$ is equipped with a $\C^\times$-action
such that the Poisson bracket has degree $-1$. Obviously, for an arbitrary
open $U\subset X$, the algebra $\C[U]$ does not need to be graded. However, it is graded when
$U$ is $\C^\times$-stable. By a {\it conical topology} on $X$ we mean the topology, where
``open'' means Zariski open and $\C^\times$-stable. One can ask whether this topology is sufficiently
rich, for example, whether any point has an open affine neighborhood.

\begin{Thm}[Sumihiro]\label{Thm:Sumihiro}
Suppose $X$ is normal. Then any point in $X$ has an open affine neighborhood in the conical topology.
\end{Thm}

Below we always assume that $X$ is  normal. Note that $\mathcal{O}_X$ is a sheaf of graded algebras in
the conical topology. By a filtered quantization of $X$ we mean a sheaf $\mathcal{D}$ of filtered algebras
(in the conical topology on $X$) equipped with an isomorphism $\pi:\gr\mathcal{D}\xrightarrow{\sim}\Str_X$
of graded algebras such that the filtration on $\mathcal{D}$ is complete and separated and $\pi$
is compatible with the Poisson brackets as in \ref{SSS_filt_quant_alg}.

We still have a one-to-one correspondence between filtered quantizations and formal quantizations with
$\C^\times$-actions. This works just as in \ref{SSS_form_rel}  (note that $\mathcal{D}_{\hbar,fin}$
makes sense as a sheaf in conical topology).

\subsubsection{Quantization in families}
Let $X$ be a smooth scheme over a scheme $\mathcal{S}$. It still makes sense to speak about closed and non-degenerate
forms in $\Omega^2(X/\mathcal{S})$. By a symplectic $\mathcal{S}$-scheme we mean a smooth $\mathcal{S}$-scheme $X$ together with
a closed non-degenerate form $\omega_{\mathcal{S}}\in \Omega^2(X/\mathcal{S})$. Note that from $\omega$ one can recover
an $\mathcal{O}_S$-linear Poisson bracket on $X$.

By a formal quantization $\mathcal{D}_\hbar$ of $X$ we mean a sheaf of $\mathcal{O}_{\mathcal{S}}$-algebras on $X$
satisfying conditions (a),(b) in \ref{SSS_form_quant_sheaf}.

Note that the definition above still makes sense when $\mathcal{S}$ is a formal scheme and $X$ is a formal $\mathcal{S}$-scheme.

\subsubsection{Classification theorem}\label{SSS_quant_classif}
Let us finish this section with a classification theorem due to Bezrukavnikov and Kaledin, \cite{BK_quant_0}
(with a ramification given in \cite{quant}).

\begin{Thm}\label{Thm:BK}
Let $X$ be a smooth symplectic variety. Suppose $H^1(X,\Str_X)=H^2(X,\Str_X)=0$ (this holds when
$X$ is affine, for example). Then the formal quantizations of $X$ are parameterized by $H^2_{DR}(X,\C)[[\hbar]]$.
If $X$ has a $\C^\times$-action compatible with the bracket (where we have $d=1$), then the filtered
quantizations are in one-to-one correspondence with $H^2_{DR}(X,\C)$.
\end{Thm}

Even without the cohomology vanishing assumption, there is a so called {\it period map}
$\mathsf{Per}$ from the set $\mathsf{Quant}(X)$ of formal quantizations of $X$ (considered up to
an isomorphism) to
$H^2_{DR}(X)[[\hbar]]$. When the vanishing condition holds, this map is a bijection.
The classification of filtered quantizations follows from the observation that
once a quantization admits a $\C^\times$-action by automorphisms, its period
lies in $H^2_{DR}(X)\subset H^2_{DR}(X)[[\hbar]]$ (and if the vanishing holds, the converse is also true),
see \cite[2.3]{quant}.

Assume until the end of the section that the vanishing condition holds.

A formal quantization $\mathcal{D}_\hbar$ having a $\C^\times$-action by automorphisms and satisfying  $\operatorname{Per}(\mathcal{D}_\hbar)=0$ has a
nice property: it is {\it even}. When $X$ is affine this means that the quantization can be realized
by a star-product $f*g=\sum_{i=0}^\infty D_i(f,g)\hbar^i$ with $\deg D_i=-i$ and
$D_i(f,g)=(-1)^i D_i(g,f)$. For general $X$, being even means that there is an antiautomorphism $\varrho$ of $\mathcal{D}_\hbar$ that commutes with the $\C^\times$-action, is the identity modulo $\hbar$, and maps $\hbar$ to $-\hbar$.

Let us finish this subsection with the discussion of the universal quantization. The variety
$X$ has a universal symplectic deformation $\widehat{X}$ over the formal disc $\mathcal{S}$
that is the formal neighborhood of $0$ in $H^2_{DR}(X)$ (provided $H^i(X,\mathcal{O}_X)=0$
for $i=1,2$), see \cite{KaVe}. The universality means that any
other formal symplectic deformation of $X$ is obtained from $\widehat{X}$ by pull-back.
Further, there is a canonical quantization $\widehat{\mathcal{D}}_\hbar$ of $\widehat{X}/\mathcal{S}$.
All quantizations of $X$ are obtained by pulling back $\widehat{\mathcal{D}}_\hbar$.
More precisely, we can view $\widehat{\mathcal{D}}_\hbar$ as a sheaf of $\C[[H_{DR}^2(X),\hbar]]$-algebras
on $X$ (via the sheaf-theoretic pull-back) and then we can obtain quantizations of $X$
by base change to $\C[[\hbar]]$.

In the case when $X$, in addition, has a $\C^\times$-action rescaling the symplectic form, we can consider
the universal $\C^\times$-equivariant deformation $\tilde{X}$ over $H^2_{DR}(X)$
as well as its canonical quantization $\tilde{\mathcal{D}}_\hbar$.

\subsection{Modules over quantizations}
Let $X$ be a Poisson variety (or scheme). We are going to define  coherent and  quasi-coherent
modules over filtered and formal quantizations of $X$ (to be denoted by $\mathcal{D}$ and $\mathcal{D}_\hbar$, respectively).
\subsubsection{Coherent modules over formal quantizations}
By definition, a sheaf $\mathcal{M}_\hbar$ of $\Dcal_\hbar$-modules on $X$ is called {\it coherent}
if $\mathcal{M}_\hbar/\hbar \mathcal{M}_\hbar$ is a coherent $\mathcal{O}_X$-module and
$\mathcal{M}_\hbar$ is complete and separated in $\hbar$-adic topology. Note that the condition
of being complete and separated in the $\hbar$-adic topology is local. So being coherent is a local
condition (as in Algebraic geometry).

Let $X$ be affine and let $\A_\hbar:=\Gamma(\Dcal_\hbar)$. Let $\mathcal{N}_\hbar$ be a finitely generated $\A_\hbar$-module. Then it is easy to see that $\mathcal{N}_\hbar$ is complete and separated in the $\hbar$-adic
topology. It follows that $\Dcal_\hbar\otimes_{\A_\hbar}\mathcal{N}_\hbar$ is a coherent $\Dcal_\hbar$-module.
Conversely, for a coherent $\Dcal_\hbar$-module $\mathcal{M}_\hbar$, the global sections $\Gamma(\mathcal{M}_\hbar)$
is a finitely generated $\A_\hbar$-module.

\begin{Lem}\label{Lem:coh_affine}
Let $X$ be affine. Then the functors $\Dcal_\hbar\otimes_{\A_\hbar}\bullet$ and $\Gamma(\bullet)$
are mutually quasi-inverse equivalences between the categories of coherent $\Dcal_\hbar$-modules
and finitely generated $\A_\hbar$-modules.\end{Lem}
\begin{proof}
Note that these functors define compatible equivalences between the categories of coherent $\Dcal_\hbar/(\hbar^n)$-modules
and of finitely generated $\A_\hbar/(\hbar^n)$-modules for any $n$ (which is proved in the same way as the classical
statement for $n=1$). Then we use that all objects we consider are complete and separated in the $\hbar$-adic topology.
\end{proof}

From this lemma we easily see that a subsheaf and a quotient sheaf  of a coherent $\Dcal_\hbar$-module
is coherent itself. So the category $\operatorname{Coh}(\mathcal{D}_\hbar)$ of coherent $\Dcal_\hbar$-modules is an
abelian category.

\subsubsection{Quasi-coherent modules over formal quantizations}
By a quasi-coherent $\Dcal_\hbar$-module we mean a direct limit of coherent $\Dcal_\hbar$-modules.
Lemma \ref{Lem:coh_affine} implies that, when $X$ is affine, the category of quasi-coherent $\Dcal_\hbar$-modules
is equivalent to the category of $\Gamma(\mathcal{D}_\hbar)$-modules.

Analogously to the classical algebro-geometric result, the category $\operatorname{QCoh}(\Dcal_\hbar)$ of quasi-coherent $\Dcal_\hbar$-modules has enough injective objects. Note that the natural functor from $D^b(\operatorname{Coh}(\Dcal_\hbar))$ to the full subcategory in $D^b(\operatorname{QCoh}(\Dcal_\hbar))$
of all complexes with coherent homology is a category equivalence. This is because any quasi-coherent complex
is a union of coherent subcomplexes, as in the usual Algebro-geometric situation.

\subsubsection{Modules over filtered quantizations}
Let us proceed to modules over filtered quantizations. Let $\mathcal{M}$ be a sheaf of $\mathcal{D}$-modules.
We say that $\mathcal{M}$ is {\it coherent} if it can be equipped with a global complete and separated
filtration compatible with that on $\mathcal{D}$ and such that $\operatorname{gr}\mathcal{M}$
is a coherent sheaf on $X$ (such a filtration is usually called {\it good}).
The $\hbar$-adic completion of the Rees sheaf $R_\hbar(\mathcal{M})$
is then a $\C^\times$-equivariant coherent $\mathcal{D}_\hbar$-module. Conversely, if we take
a $\C^\times$-equivariant coherent $\mathcal{D}_\hbar$-module $\mathcal{M}_\hbar$, take the
$\C^\times$-finite part $\mathcal{M}_{\hbar,fin}$, then $\mathcal{M}_{\hbar,fin}/(\hbar-1)$
is a coherent $\mathcal{D}$-modules.

\begin{Lem}\label{Lem:cat_equi}
Consider the full subcategory $\operatorname{Coh}^{\C^\times}(\mathcal{D}_\hbar)_{tor}$
consisting of all modules that are torsion over $\C[[\hbar]]$. Then taking quotient by
$\hbar-1$ gives rise to an equivalence  $\operatorname{Coh}^{\C^\times}(\mathcal{D}_\hbar)/
\operatorname{Coh}^{\C^\times}(\mathcal{D}_\hbar)_{tor}\xrightarrow{\sim}\operatorname{Coh}(\mathcal{D})$.
\end{Lem}
\begin{proof}
Let us produce a quasi-inverse functor. Of course, the $R_\hbar(\mathcal{D})$-module $R_\hbar(\mathcal{M})$
depends on the choice of a good filtration. Let $\operatorname{F}, \operatorname{F}'$ be two good filtrations.
Then one can find positive integers $d_1,d_2$ such that $\operatorname{F}_{i-d_1}\mathcal{M}
\subset \operatorname{F}'_{i}\mathcal{M}\subset \operatorname{F}_{i+d_2}\mathcal{M}$
the inclusion of subsheaves (of vector spaces) in $\mathcal{M}$ (it is enough to check this claim
for local sections over open subsets from an affine cover, where it is easy). It follows that modulo $\hbar$-torsion
the sheaf $R_\hbar(\mathcal{M})$ is independent of the choice of a good filtration. Our quasi-inverse
functor sends $\mathcal{M}$ to the $\hbar$-adic completion of $R_\hbar(\mathcal{M})$. To check that this
is indeed a quasi-inverse functor is standard.
\end{proof}

\subsubsection{Supports}\label{SSS_supports}
For a coherent $\mathcal{D}_\hbar$-module $\mathcal{M}_\hbar$ we have the notion of support.
By definition, $\operatorname{Supp}(\mathcal{M}_\hbar):=\operatorname{Supp}(\mathcal{M}_\hbar/\hbar \mathcal{M}_\hbar)$,
this is a closed subvariety in $X$.

Now let $\mathcal{M}\in \operatorname{Coh}(\mathcal{D})$. Then we can take a good filtration on $\mathcal{M}$
and set $\operatorname{Supp}(\mathcal{M}):=\operatorname{Supp}(\operatorname{gr}\mathcal{M})$.
By the argument in the proof of Lemma \ref{Lem:cat_equi}, the support of $\mathcal{M}$ is well-defined,
i.e., it does not depend on the choice of a good filtration.

\subsubsection{Global sections and localization}
Let $\mathcal{D}$ be a filtered quantization of $X$. We have natural functors $\operatorname{Coh}(\mathcal{D})
\rightarrow \Gamma(\mathcal{D})\operatorname{-mod}$ of taking global sections (to be denoted by $\Gamma$)
as well as a functor in the opposite direction $\operatorname{Loc}: \Gamma(\mathcal{D})\operatorname{-mod}
\rightarrow \operatorname{Coh}(\mathcal{D}), M\mapsto \mathcal{D}\otimes_{\Gamma(\mathcal{D})}M$.

Let us discuss a situation when these functors behave particularly nicely. Namely, let $X$ be a conical
symplectic resolution of singularities of an affine variety $X_0$. Note that, by the Grauert-Riemenschneider
theorem, the higher cohomology of $\mathcal{O}_X$ vanish. This has the following corollary (the proof is
left to the reader).

\begin{Lem}\label{Lem:der_glob_sec}
We have $H^i(\mathcal{D})=0$ for $i>0$. Moreover, $\Gamma(\mathcal{D})$ is a quantization of $X_0$.
\end{Lem}

Thanks to this lemma, it makes sense to consider derived functors $R\Gamma: D(\operatorname{Coh}(\mathcal{D}))
\rightarrow D(\Gamma(\mathcal{D})\operatorname{-mod})$ and $L\operatorname{Loc}:D(\Gamma(\mathcal{D})\operatorname{-mod})
\rightarrow D(\operatorname{Coh}(\mathcal{D}))$. In fact, $R\Gamma$ is given by the ${\rm \check{C}}$ech
complex and so restricts to bounded (to the left and to the right) derived categories. The functor
$L\operatorname{Loc}$ restricts to $D^-$'s. Lemma \ref{Lem:der_glob_sec} implies that
$R\Gamma\circ L\operatorname{Loc}$ is the identity on $D^-(\Gamma(\mathcal{D})\operatorname{-mod})$.
Furthermore, if $\Gamma(\mathcal{D})$ has finite homological dimension, then $L\operatorname{Loc}$
maps $D^b(\Gamma(\mathcal{D})\operatorname{-mod})$  to $D^b(\operatorname{Coh}(\mathcal{D}))$
and is left inverse to $R\Gamma$. It is likely (and is proved in many cases, see, e.g., \cite{MN}) that $R\Gamma$
and $L\operatorname{Loc}$ are mutually quasi-inverse equivalences in this case.

\subsection{Frobenius constant quantizations}
Above, we were dealing with the case when the ground field is $\C$. Everything works the same
for any algebraically closed field of characteristic $0$. In this section we are going to work
over an algebraically closed field $\Fi$ of positive characteristic.

The notions of filtered and formal quantizations still make sense, both for algebras and for
varieties. But in positive characteristic we have an important special class of quantizations: Frobenius
constant ones.

\subsubsection{Basic example}\label{SSS_Frob_basic}
Let us start our discussion with an example of a quantization: the Weyl algebra $\Weyl(V)$,
where $V$ is a symplectic $\Fi$-vector space. A new feature is that this algebra is finite
over its center. Namely, for $v\in V\subset \Weyl(V)$, the element $v^p\in \Weyl(V)$ lies in the center.
We have a semi-linear map $\iota: V\rightarrow \Weyl(V)$ given by $v\mapsto v^p$ on $v\in V$ with central
image that extends to a ring homomorphism $S(V)\rightarrow \Weyl(V)$.
The semi-linearity condition is  $\iota(av)=\operatorname{Fr}(a)\iota(v)$, where
$\operatorname{Fr}:\Fi\rightarrow \Fi$ is the Frobenius automorphism. Let $V^{(1)}$ denote the
$\Fi$-vector space identified with $V$ as an abelian group but with new multiplication by scalars:
$a.v=\operatorname{Fr}^{-1}(a)v$. So $\iota$ becomes an algebra homomorphism
when viewed as a map $S(V^{(1)})\rightarrow \Weyl(V)$, its image is usually called the
$p$-center, in our case it coincides with the whole center. Another important feature of this example
is that $\Weyl(V)$ is an Azumaya algebra over $V^{(1)}$, i.e.,  $\Weyl(V)$ is a vector bundle over $\operatorname{Spec}(S(V^{(1)}))$ and all (geometric) fibers are matrix algebras (of rank $p^{\dim V/2}$).

\subsubsection{Definition}
The notion of a Frobenius constant quantization generalizes the example
in \ref{SSS_Frob_basic}. We will give the definition in the filtered setting and only
for symplectic varieties -- we will only need it in this case. Let $X$ be a smooth symplectic
$\Fi$-variety equipped with an $\Fi^\times$-action rescaling the symplectic form (by the
character $t\mapsto t^d$). Let $X^{(1)}$ be the $\Fi$-variety that is identified with $X$ as a scheme over
$\operatorname{Spec}(\Z)$ but with twisted multiplication by scalars in the structure sheaf just as
in \ref{SSS_Frob_basic}. We have a natural morphism $\operatorname{Fr}:X\rightarrow X^{(1)}$
of $\Fi$-varieties and hence we have a sheaf $\operatorname{Fr}_*(\Str_X)$ on $X^{(1)}$. This is
a coherent sheaf of algebras and a vector bundle of rank $p^{\dim X}$.

\begin{defi}\label{Fron:const}
A Frobenius constant quantization is a filtered sheaf $\mathcal{D}$ of Azumaya algebras on $X^{(1)}$
together with an isomorphism $\gr\mathcal{D}\xrightarrow{\sim} \operatorname{Fr}_*\Str_X$
of graded algebras (in conical topology) that satisfies our usual compatibility condition
on Poisson brackets.
\end{defi}

It is not difficult to show that a Frobenius constant quantization gives rise to a filtered
quantization of $X$. But, as we will see \ref{SSS_qHam_char_p}, not every filtered quantization arises
this way.

\subsubsection{Differential operators}
Let us give another example that should be thought as a global analog
of \ref{SSS_Frob_basic}. Let $Y$
be a smooth $\Fi$-variety. Consider the sheaf $D_Y$ of differential operators on $Y$.
Let $\xi$ be a  vector field on an open subset  $Y'\subset Y$. Define a vector field $\xi^{[p]}$ as follows.
For every open affine subvariety $Y^0\subset Y'$, we can regard $\xi$ as a derivation of $\Fi[Y^0]$.
The map $\xi^p: \Fi[Y^0]\rightarrow \Fi[Y^0]$ is again a derivation. The corresponding vector
field on $Y'$ (that is easily seen to be well-defined) is what we denote by $\xi^{[p]}$.
It is easy to see that $f^p$, for a function $f$ on $Y$,  and $\xi^p-\xi^{[p]}$, for
a vector field $\xi$ (here $\xi^p$ is taken with respect to the product on $D_Y$),
are central. The maps $f\mapsto f^p, \xi\mapsto \xi^p-\xi^{[p]}$ give rise to
a sheaf of algebras homomorphism $\pi_* \Str_{(T^*Y)^{(1)}}\rightarrow \operatorname{Fr}_* D_Y$,
where we write $\pi$ for the projection $(T^*Y)^{(1)}=T^*(Y^{(1)})\rightarrow Y^{(1)}$.
The sheaf $D_Y$ then becomes a Frobenius constant quantization of $T^*Y$.

To finish this section, let us mention that, under some restrictions on $X$, there is
a classification of Frobenius constant quantizations, see \cite{BK_char_p}.

\section{Hamiltonian reductions}
In this section we recall the notions of the classical and quantum Hamiltonian reduction. The classical
Hamiltonian reduction produces a new Poisson variety from an existing Poisson variety with suitable symmetries.
The quantum Hamiltonian reduction does the same on the level of quantizations.

We start by discussing classical Hamiltonian reductions, Section \ref{SS_class_Ham}.
First, we recall Hamiltonian actions and moment maps.
Then we define classical Hamiltonian reductions in the settings of categorical quotients and of GIT quotients.
We then proceed to the construction and basic properties of Nakajima quiver varieties that are our main examples
of Hamiltonian reductions. Next, we explain how quotient singularities $V_n/\Gamma_n$ are realized
as quiver varieties. Finally, we construct symplectic resolutions of singularities for $V_n/\Gamma_n$
and establish, following Namikawa, some isomorphisms between some of these resolutions.

In Section \ref{SS_Quant_Ham_red}
we proceed to quantum Hamiltonian reductions. We define them on the level of algebras and on the level
of sheaves and compare the two levels. After that we state one of the main results of this survey: an isomorphism
between spherical SRA for wreath-product groups and quantum Hamiltonian reductions. We finish this section
by discussing a quantum version of Namikawa's Weyl group action.

Section \ref{SS_quant_Ham_Frob} deals with Hamiltonian reductions for Frobenius constant quantization.
We first recall some basic results on GIT in positive characteristic. Then we discuss
Nakajima quiver varieties in sufficiently large positive characteristic. Finally,
we prove, following Bezrukavnikov, Finkelberg and Ginzburg, that the quantum Hamiltonian
reduction of a Frobenius constant quantization at an integral value of the quantum comoment
map is Frobenius constant.

\subsection{Classical Hamiltonian reduction}\label{SS_class_Ham}
\subsubsection{Hamiltonian group actions}
Let $X$ be a Poisson variety (over an algebraically closed field) and let $G$ be an algebraic group acting on $X$.
The action induces a Lie algebra homomorphism $\g\rightarrow \Vect(X)$,
the image of $\xi\in \g$ under this homomorphism will be denoted by $\xi_X$.
We say that the $G$-action on $X$ is {\it Hamiltonian}, if there is a
$G$-equivariant linear map $\g\rightarrow \C[X], \xi\mapsto H_\xi$,
such that $\{H_\xi,\cdot\}=\xi_X$. Note that this map is automatically a Lie algebra
homomorphism. This map is called the {\it comoment map}, the dual map
$\mu:X\rightarrow \g^*$ is the {\it moment map}.

Let us provide two examples of Hamiltonian actions.

\begin{Ex}
Let $Y$ be a smooth variety, $G$ act on $Y$. Then $X:=T^*Y$ carries a natural
$G$-action. This action is Hamiltonian with $H_\xi=\xi_Y$ (viewed as a function on $X$).
\end{Ex}

\begin{Ex}
Let $V$ be a vector space (with symplectic form $\Omega$) and let $G$ act on
$V$ by linear symplectomorphisms. The action is Hamiltonian with $H_\xi(v)=\frac{1}{2}\Omega(\xi v,v)$.
\end{Ex}

Below we will need a standard property of Hamiltonian actions.

\begin{Lem}\label{Lem:moment_submersion}
Let $x\in X$. Then $\operatorname{im}d_x\mu\subset \g^*$ coincides with the annihilator of $\g_x:=\operatorname{Lie}(G_x)$.
In particular, $\mu$ is a submersion at $x$ if and only if $G_x$ is finite.
\end{Lem}

\subsubsection{Hamiltonian reduction}\label{SSS_Ham_red_basic}
Let $A$ be a Poisson algebra and $\g$ be a Lie algebra equipped with a Lie algebra homomorphism $\g\rightarrow A, \xi\mapsto H_\xi$. Consider the ideal $I:=A\{H_\xi, \xi\in \g\}$. The adjoint action of $\g$ on $A$
preserves this ideal so we can take the invariants $A\red_0\g:= (A/I)^\g$.
This algebra comes with a natural Poisson bracket: $\{a+I,b+I\}:=\{a,b\}+I$
(but $A/I$ has no Poisson bracket!).

This construction has several ramifications. First, let $\lambda:\g\rightarrow \C$ be a character
(i.e., a function vanishing on $[\g,\g]$). Then we can set $A\red_\lambda \g:=(A/A\{H_\xi-\langle \lambda,\xi\rangle\})^\g$. Also we can set $A\red\g:=(A/A\{H_\xi, \xi\in [\g,\g]\})^\g$. The latter
is a Poisson $S(\g/[\g,\g])$-algebra whose specialization at $\lambda\in (\g/[\g,\g])^*$
coincides with $A\red_\lambda \g$ provided that the $\g$-action on $A/A\{H_\xi, \xi\in [\g,\g]\}$
is completely reducible.

Let us proceed to a geometric incarnation of this construction. Suppose the base field
is of characteristic $0$. To ensure a good behavior of quotients
assume that $G$ is a reductive group. Let $X$ be an affine Poisson variety equipped with a
Hamiltonian $G$-action. Then we can take $A:=\C[X]$ together with the comoment map
$\xi\mapsto H_\xi$. We set $A\red_0 G:=(A/I)^G$, this algebra coincides with $\A\red_0 \g$
when $G$ is connected. It is finitely generated by the Hilbert theorem, here we use that $G$
is reductive. The variety (or scheme) $\Spec(A\red_0 G)$ is nothing else but the categorical
quotient $X\red_0 G:=\mu^{-1}(0)\quo G$.

Here is a corollary of Lemma \ref{Lem:moment_submersion}.

\begin{Cor}\label{Lem:reduct_sympl}
Suppose that $X$ is smooth and symplectic and that the $G$-action on $\mu^{-1}(0)$ is free.
Then $X\red_0 G$ is smooth and symplectic of dimension $\dim X-2\dim G$.
\end{Cor}
\begin{proof}
The variety $\mu^{-1}(0)$ is smooth by Lemma \ref{Lem:moment_submersion}.
That the quotient is smooth of required dimension is a straightforward corollary of the Luna slice
theorem, see, e.g., \cite[Section 6.3]{PV}.

The form on $X\red_0 G$ can be recovered as follows. Let $\Omega$ denote the form on $X$,
$\iota:\mu^{-1}(0)\hookrightarrow X$ denote the inclusion map and $\pi:\mu^{-1}(0)\rightarrow
X\red_0 G$ be the projection. Then there is a unique 2-form $\Omega_{red}$ on $X\red_0 G$
such that $\pi^* \Omega_{red}=\iota^*\Omega$ and this is the form we need.
\end{proof}

\subsubsection{GIT Hamiltonian reduction}\label{SSS_GIT_red}
We will be mostly interested in Hamiltonian reductions for linear actions
$G\curvearrowright V$.  The assumptions of Corollary \ref{Lem:reduct_sympl} are not satisfied in this case. However, if one uses GIT quotients instead of the usual categorical quotients, one can often get a smooth
symplectic  variety that will be a resolution of the usual reduction $V\red_0 G$.

Let us recall the construction of a GIT quotient. Let $G$ be a reductive algebraic group
acting on an affine algebraic variety $X$. Fix a character $\theta:G\rightarrow \C^\times$.
We use the additive notation for the multiplication of characters.
Then consider the space $\C[X]^{G,n\theta}$ of $n\theta$-semiinvariants:
$\C[X]^{G,n\theta}:=\{f\in \C[X]| g.f:=\theta(g)^n f\}$ (recall that $g.f(x):=f(g^{-1}x)$).
Consider the graded algebra $\bigoplus_{n\geqslant 0} \C[X]^{G,n\theta}$, where $\deg \C[X]^{G,n\theta}:=n$. Then we set $X\quo^\theta G:=\operatorname{Proj}(\bigoplus_{n\geqslant 0} \C[X]^{G,n\theta})$,
this is a projective variety over $X\quo G$. Note that we no longer have a morphism $X\rightarrow X\quo^\theta G$.
Instead, consider the open subset of {\it $\theta$-semistable} points $X^{\theta-ss}$, a point $x\in X$
is called semistable if there is $f\in \C[X]^{G,n\theta}$ for $n>0$ with $f(x)\neq 0$. We clearly have a
natural morphism $X^{\theta-ss}\rightarrow X\quo^\theta G$ that makes the following diagram commutative

\begin{picture}(70,30)
\put(4,2){$X$}
\put(2,22){$X^{\theta-ss}$}
\put(34,2){$X\quo G$}
\put(33,22){$X\quo^{\theta}G$}
\put(6,21){\vector(0,-1){14}}
\put(38,21){\vector(0,-1){14}}
\put(9,3){\vector(1,0){23}}
\put(12,23){\vector(1,0){20}}
\put(7,13){\tiny $\subseteq$}
\end{picture}

The variety $X\quo^\theta G$ is glued from the varieties of the form $X_f\quo G$, where $f\in \C[X]^{G,n\theta}$
with some $n>0$. The intersection of $X_f\quo G, X_g\quo G$ inside $X\quo^\theta G$ is identified
with $X_{fg}\quo G$, where the inclusions $X_{fg}\quo G\hookrightarrow X_f\quo G,X_g\quo G$
are induced from the inclusions $X_{fg}\hookrightarrow X_f,X_g$ by passing to the quotients.

In the setting of \ref{SSS_Ham_red_basic}, we set $X\red_0^\theta G:=\mu^{-1}(0)^{\theta-ss}\quo G$.
This is a Poisson variety (the bracket comes from gluing together the brackets on the open subvarieties
$X_f\red_0 G$) equipped with a projective morphism $X\red_0^\theta G\rightarrow
X\red_0 G$  of Poisson varieties. If $X$ is smooth and symplectic, and the $G$-action on
$\mu^{-1}(0)^{\theta-ss}$ is free, then $X\red^\theta_0 G$ is smooth and symplectic
of dimension $\dim X-2\dim G$. The symplectic form on $X\red^\theta_0 G$ is recovered similarly
to the case of $X\red_0 G$ considered above.

\subsubsection{Nakajima quiver varieties:construction}\label{SSS_quiv_var_constr}
Now we are going to introduce an important special class of varieties constructed by means of
Hamiltonian reduction: Nakajima quiver varieties, introduced in \cite{Nakajima1}, see also
\cite{Nakajima2}.

By a quiver, we mean an oriented graph. Formally, it can be presented as a quadruple
$Q=(Q_0,Q_1,t,h)$, where $Q_0,Q_1$ are finite sets of vertices and arrows, respectively,
and $t,h:Q_1\rightarrow Q_0$ are maps that to an arrow $a$ assign its tale and head.

Let us proceed to (framed) representations of $Q$. Fix two elements $v,w\in \Z_{\geqslant 0}^{Q_0}$
and set $V_i:=\C^{v_i},W_i:=\C^{w_i}, i\in Q_0$. Consider the space
$$R(=R(Q,v,w)):=\bigoplus_{a\in Q_1}\operatorname{Hom}_{\C}(V_{t(a)},V_{h(a)})\oplus \bigoplus_{i\in Q_0}
\operatorname{Hom}_{\C}(W_i,V_i).$$
An element of $R$ can be thought as a collection of linear maps, one for each arrow, between the corresponding
vector spaces, together with  collections of vectors in each $V_i$. This description suggests a group
of symmetry of $R$: we set $G:=\prod_{i\in Q_0}\GL(V_i)$, this group acts by changing bases in the spaces
$V_i$.

A character of $G$ is of the form $g=(g_i)_{i\in Q_0}\mapsto \prod_{i\in Q_0}\det(g_i)^{\theta_i}$,
where $\theta=(\theta_i)_{i\in Q_0}\in \Z^{Q_0}$. We will identify the character group of $G$
with $\Z^{Q_0}$.

A Nakajima quiver variety $\M_\lambda^\theta(v,w)$ is, by definition, the reduction $T^*R\red^\theta_\lambda G$.
Here $\lambda$ is a character of $\g$, it can be thought as an element of $\C^{Q_0}$ via
$\lambda(x):=\sum_{i\in Q_0}\lambda_i\operatorname{tr}(x_i)$. The moment map
$\mu:T^*R\rightarrow \bigoplus_{i\in Q_0}\operatorname{End}(V_i)=\g(\cong \g^*)$ is explicitly given as follows:
$$(x_a,x_{a^*}, i_k,j_k)_{a\in Q_1, k\in Q_0}\mapsto \sum_{a\in Q_1}(x_ax_{a^*}-x_{a^*}x_a)-\sum_{k\in Q_0}j_k i_k,$$
where $x_a\in \operatorname{Hom}(V_{t(a)},V_{h(a)}), x_{a^*}\in \operatorname{Hom}(V_{h(a)},V_{t(a)}), i_k\in \operatorname{Hom}(V_k,W_k), j_k\in \operatorname{Hom}(W_k,V_k)$.

We also would like to remark that the quiver variety is independent of the choice of an orientation of $Q$.
Indeed, let $Q'$ be a quiver obtained from $Q$ by changing the orientation of a single arrow $a$
and let $R'$ be the corresponding representation space. Then we have an isomorphism $T^*R\cong T^*R'$
that sends $x_a$ to $x'_{a^*}$, $x_{a^*}$ to $-x'_a$ and does not change the other components.
This is a $G$-equivariant symplectomorphism that intertwines the moment maps and hence inducing
a symplectomorphism of the corresponding Nakajima quiver varieties.

When $\lambda=0$, we have a $\C^\times$-action on $\M^\theta_0(v,w)$ that rescales the Poisson structure.
For example, one can take the action induced by the dilation action on $T^*R$, that is, $t.v:=t^{-1}v,
t\in \C^\times, v\in T^*R$ to be called the dilation action as well. Then the Poisson bracket on
$\M^\theta_0(v,w)$ has degree $-2$. We can also have an action such that the Poisson bracket has
degree $-1$ coming from $t.(r,r^*):=(r,t^{-1}r^*), r\in R, r^*\in R$.

\subsubsection{Nakajima quiver varieties: structural results}\label{SSS_quiver_struct}
Let us explain some structural results regarding the quiver varieties and the corresponding moment
maps. We will need algebro-geometric properties of $\mu^{-1}(\lambda)$ and of $\M^0_\lambda(v,w)$
due to Crawley-Boevey and also a criterium for the freeness of the $G$-action on
$\mu^{-1}(\lambda)^{\theta-ss}$  due to Nakajima.

\begin{Thm}[Crawley-Boevey, \cite{CB_norm}]\label{Thm_quiv_norm}
The scheme $\M^0_\lambda(v,w)$ is reduced and normal.
\end{Thm}

We now want to  provide a criterium for $\mu:T^*R\rightarrow \g^*$ to be flat proved in \cite{CB_quiv}.
Define the symmetrized Tits form $\C^{Q_0}\times \C^{Q_0}\rightarrow \C$:
$$(v^1,v^2):=\sum_{a\in Q_1}(v^1_{t(a)}v^2_{h(a)}+v^1_{h(a)}v^2_{t(a)})-2\sum_{i\in Q_0}v^1_iv^2_i$$
and  quadratic maps $p,p_w: \C^{Q_0}\rightarrow \C$ by
$$p(v):=1-\frac{1}{2}(v,v),\quad p_w(v):=w\cdot v-\frac{1}{2}(v,v).$$

\begin{Thm}[Crawley-Boevey, \cite{CB_quiv}]\label{Thm_flat}
The following two conditions are equivalent:
\begin{itemize}
\item[(i)] $\mu$ is flat.
\item[(ii)] $p_w(v)\geqslant p_w(v^0)+\sum_{i=1}^k p(v^i)$ for any decomposition $v=v^0+v^1+\ldots+v^k$ with $v^i\in \Z_{\geqslant 0}^{Q_0}$ for $i=1,\ldots,k$.
\end{itemize}
\end{Thm}

\begin{Thm}[Crawley-Boevey, \cite{CB_quiv}]\label{Thm:irred}
Suppose that, for a proper decomposition $v=v^0+v^1+\ldots+v^k$, we have $p_w(v)>p_w(v^0)+\sum_{i=1}^k p(v^i)$.
Then $\mu^{-1}(0)$ is irreducible and a generic $G$-orbit there is closed and free.
\end{Thm}

Let us proceed to a criterium for the action of $G$ on $\mu^{-1}(\lambda)^{\theta-ss}$ to be free.
We can view $Q$ as a Dynkin diagram and form the corresponding Kac-Moody algebra $\g(Q)$.
Then $\C^{Q_0}$ gets identified with the dual of the Cartan of $\g(Q)$ in such a way that the coordinate
vector $\epsilon_i, i\in Q_0,$ becomes a simple root.
Then, \cite{Nakajima1}, the action of $G$ on $\mu^{-1}(\lambda)^{\theta-ss}$ is free if and only if there
are no roots $v'$ of $\g(Q)$ such that $v'\leqslant v$ (component-wise) and
$v'\cdot \theta=v'\cdot \lambda=0$.

The equations $v'\cdot \theta=0$, where $v'$ is a root satisfying $v'\leqslant v, v'\cdot \lambda=0$ split the character
lattice into the union of cones. It is a classical fact from GIT, that when $\theta,\theta'$
are generic and inside one cone, we have $\mu^{-1}(\lambda)^{\theta-ss}=\mu^{-1}(\lambda)^{\theta'-ss}$.
So $\M_\lambda^\theta(v,w)=\M_\lambda^{\theta'}(v,w)$.


\subsubsection{$\operatorname{Hilb}_n(\C^2)$ and $\C^{2n}/\mathfrak{S}_n$ as quiver varieties}\label{SSS_Hilb_quiver}
Let $Q$ be a quiver with a single vertex and a single loop (a.k.a. the Jordan quiver).
We are going to show that $\operatorname{Hilb}_n(\C^2)$ is identified with $\M^{-1}_0(n,1)$
and $\C^{2n}/\mathfrak{S}_n$ is identified with $\M^0_0(n,1)$ (and the Hilbert-Chow
map from \ref{SSS_HS_PB} becomes the natural morphism $\M^{-1}_0(n,1)\rightarrow \M^0_0(n,1)$
from \ref{SSS_GIT_red}).

An identification  $\M_0^{-1}(n,1)\cong \operatorname{Hilb}_n(\C^2)$ is an easier part.
We have $R=\operatorname{End}(\C^n)\oplus \C^n$. Using the trace pairing, we identify
$R^*$ with $\operatorname{End}(\C^n)\oplus \C^{n*}$ so that
$T^*R=\operatorname{End}(\C^n)^{\oplus 2}\oplus \C^n\oplus \C^{n*}$.
We write $(A,B,i,j)$ for a typical point of $T^*R$. Identifying $\g$
with $\g^*$ again using the trace pairing, we can write the moment map
$\mu:T^*R\rightarrow \g$ as $\mu(A,B,i,j)=[A,B]+ij$.

Using the Hilbert-Mumford theorem from Invariant theory, see, e.g., \cite[5.3]{PV}, one shows that
$(T^*R)^{\theta-ss}=\{(A,B,i,j)| \C\langle A,B\rangle i=\C^n\}$.
Then it is a nice Linear Algebra exercise to show that if $[A,B]+ij=0$
and $\C\langle A,B\rangle i=\C^n$, then $j=0$. So
$\mu^{-1}(0)^{\theta-ss}\quo G=\{(A,B,i)| [A,B]=0, \C[A,B]i=\C^n\}/G$
that recovers the classical description of $\operatorname{Hilb}_n(\C^2)$,
see \cite[Theorem 1.14]{Nakajima_book}.

An identification $\M^0_0(n,1)\cong \C^{2n}/\mathfrak{S}_n$ is more subtle.
An easy part is to construct a morphism $\iota:\C^{2n}/\mathfrak{S}_n\rightarrow
\M^0_0(n,1)$: we send $(\underline{x},\underline{y})\in \C^{2n}$ to
$(\operatorname{diag}(\underline{x}),\operatorname{diag}(\underline{y}),0,0)\in \mu^{-1}(0)$
and this induces a morphism of quotients. Then one checks that $\iota$ is  a closed
embedding. For this, one uses a classical result of Weyl to see that polynomials
of the form $\iota^* F_{m,n}$, where $F_{m,n}(A,B,i,j):=\operatorname{tr}(A^n B^m)$
generate the algebra $\C[\underline{x},\underline{y}]^{\mathfrak{S}_n}$. It remains to
prove that $\iota$ is surjective. This  is  based on an even nicer linear
algebra fact: $A,B\in \operatorname{End}(\C^n)$  with $\operatorname{rk}[A,B]\leqslant 1$
are upper-triangular in some basis.

\begin{Lem}\label{Lem:typeA_Poisson}
The isomorphism $\M^0_0(n,1)\cong \C^{2n}/\mathfrak{S}_n$ intertwines  the Poisson brackets.
\end{Lem}
\begin{proof}
Consider the principal open subsets $$R^{reg}=\{(A,i)| A\text{ has distinct e-values}\}, \C^{n,reg}:=\{(x_1,\ldots,x_n)| x_i\neq x_j, \text{ for }i\neq j\}.$$ Note that under the above embedding $\C^{2n}\hookrightarrow T^*R$, we have
$T^*\C^{n,reg}\hookrightarrow T^*R^{reg}$. Moreover, the pull-back of the symplectic form from
$T^*R^{reg}$ to $T^*\C^{n,reg}$ coincides with the natural symplectic form on the latter.
Using the description of the symplectic form on the reduction, we conclude that
the induced morphism of quotients $T^*\C^{n,reg}/\mathfrak{S}_n\rightarrow T^*R^{reg}\red_0^0 G$
is a symplectomorphism. But $T^*R^{reg}\red_0^G$ embeds as an open subset into $\M^0_0(n,1)$
and the  symplectomorphism above is the restriction of the isomorphism $\C^{2n}/\mathfrak{S}_n\xrightarrow{\sim}
\M^0_0(n,1)$ to $T^*\C^{n,reg}/\mathfrak{S}_n$. The claim of the lemma follows.
\end{proof}

\subsubsection{McKay correspondence}
Let $\Gamma_1$ be a finite subgroup of $\SL_2(\C)$. It turns out that the singular
Poisson variety $V_n/\Gamma_n$ (where recall $V_n=\C^{2n}$) and its symplectic resolutions also can be realized
as Nakajima quiver varieties.

The first step in this isomorphism is the McKay correspondence: a way to label the finite subgroups
of $\SL_2(\C)$ by Dynkin diagrams. Let $\Gamma_1$ be a finite subgroup of $\SL_2(\C)$ and let
$N_0,\ldots,N_r$ be the irreducible representations of $\Gamma_1$, where $N_0$ is the trivial representation.
Let us define the McKay graph of $\Gamma_1$: its vertices are $0,1,\ldots,r$ and the number of
edges (we consider a non-oriented graph)
between $i$ and $j$ is $\dim \operatorname{Hom}_{\Gamma}(\C^2\otimes N_i,N_j)$, note that this is
well-defined because $\C^2$ is a self-dual representation of $\Gamma$ and so the number of edges
between $i$ and $j$ is the same as between $j$ and $i$. McKay proved the following facts:
\begin{itemize}
\item[(i)] The resulting graph is an extended Dynkin graph of types $A,D,E$ and $0$
is the extending vertex.
\item[(ii)] The vector $(\dim N_i)_{i=0}^r$ is the indecomposable imaginary root $\delta$
of the corresponding affine Kac-Moody algebra.
\end{itemize}

\subsubsection{$\C^2/\Gamma_1$ as a quiver variety}
Let $Q$ be the McKay graph of $\Gamma_1$ with an arbitrary orientation. Then there is an isomorphism
$\M_0^0(\delta,0)\cong \C^2/\Gamma_1$.

Let us explain how this is established following \cite[Section 8]{CBH}. For this, we will need the representation
varieties.

Let $A$ be a finitely generated associative algebra and $V$ be a vector space. Then
the set $X:=\operatorname{Hom}(A, \End(V))$ of algebra homomorphisms is an algebraic variety.
More precisely, if $A$  is  the quotient of $\C\langle x_1,\ldots,x_n\rangle$ by relations
$F_\alpha(x_1,\ldots,x_n)$, where $\alpha$ runs over an indexing set $\mathfrak{I}$,
then $X=\{(A_1,\ldots,A_n)\in \operatorname{End}(V)| F_\alpha(A_1,\ldots,A_n)=0, \alpha\in \mathfrak{I}\}$.
The group $G:=\GL(V)$ naturally acts on $X$ and so we can form the quotient $X\quo G$ (called the
representation variety). Recall that, in general, the points of $X\quo G$ correspond to the closed
$G$-orbits on $X$, in our case an orbit is closed if its element is a semisimple representation.

This construction has various ramifications. For example, we can consider a semisimple
finite dimensional subalgebra $A_0\subset A$ and an $A_0$-module $V$. This leads
to the variety $X$ of $A_0$-linear homomomorphisms $A\rightarrow \operatorname{End}(V)$
acted on by the group $G$ of $A_0$-linear automorphisms of $V$. In this situation
we still can speak about representation varieties. We will realize $\M_0^0(\delta,0),\C^2/\Gamma_1$
as the representation varieties of this kind and then show that the algebras involved are Morita
equivalent, this will yield an isomorphism of interest.

Let us start with $\C^2/\Gamma_1$. Set $A:=\C[x,y]\#\Gamma_1, A_0:=\C\Gamma_1\subset A$
and $V:=\C\Gamma_1$, a regular representation. Then one can show that $\C^2/\Gamma_1$
is the representation variety for this triple.

Let us proceed to $\M^0_0(\delta,0)$. Let $\bar{Q}$ be the {\it double quiver} of $Q$. It is obtained from $Q$ by adding the inverse
arrow to each arrow in $Q$. Formally, $\bar{Q}_0=Q_0$, $\bar{Q}_1=Q_1\sqcup Q_1^*$,
where $Q_1^*$ is in bijection with $Q_1$, $a\mapsto a^*$, in such a way that
$t(a^*)=h(a), h(a^*)=t(a)$. Then form the {\it path algebra} $\C\bar{Q}$ of $\bar{Q}$,
it has a basis consisting of the paths in $\bar{Q}$, the multiplication is given by concatenation
(if two paths cannot be concatenated, the product is zero). This algebra is graded by the length
of  a path, where, by convention, the degree $0$ paths are just vertices so the corresponding
graded component $\C\bar{Q}_0$ is $\C^{Q_0}$.

Let us consider the quotient $\Pi^0(Q)$ of $\C\bar{Q}$ called the  preprojective algebra.
It is given by the following relation:
$$\sum_{a\in Q_1}[a,a^*]=0.$$
Note that $\C^{Q_0}$ naturally embeds into $\Pi^0(Q)$. It is easy to see that $\M_0^0(\delta,0)$
is the representation variety for the triple $(\Pi^0(Q),\C^{Q_0},\bigoplus_{i\in Q_0}\C^{\delta_i})$.

It turns out that there is an idempotent $f\in \C \Gamma_1$ such that $f(\C[x,y]\#\Gamma_1)f\cong
\Pi^0(Q)$. Namely, take primitive idempotents $f_i, i=0,\ldots,r,$ in the matrix summands of
$\C\Gamma_1$. Set $f:=\sum_{i\in Q_0}f_i$. Obviously, $f(\C\Gamma_1)f\cong \C^{Q_0}$.
Further, the construction of $Q$ implies that $f(\operatorname{Span}(x,y)\otimes\C\Gamma_1)f\cong \C \bar{Q}_1$.
These identifications induce an isomorphism $f(\C\langle x,y\rangle\#\Gamma_1)f\cong \C\bar{Q}$.
Under this isomorphism, the ideal $f(xy-yx)f$ becomes $(\sum_{a\in Q_1}[a,a^*])$, see \cite[Section 2]{CBH}.
Also note that the $\C^{Q_0}$-module $\bigoplus_{i\in Q_0}\C^{\delta_i}$ is nothing else
but $f\C\Gamma_1$. Finally, note that $f$ defines a Morita equivalence between $\C[x,y]\#\Gamma_1,
\Pi^0(Q)$. An isomorphism $\C^2/\Gamma_1\cong\M^0_0(\delta,0)$ now follows from the next lemma,
whose proof is left to the reader.

\begin{Lem}\label{Lem:rep_var_iso}
Let $A_0\subset A$ and $V$ be as above and let $f\in A_0$ be an idempotent giving a Morita
equivalence. Then the representation varieties for $(A,A_0,V)$ and $(fAf, fA_0f,fV)$
are naturally isomorphic.
\end{Lem}

Note that the algebras $\C[x,y]\#\Gamma_1$ and $\Pi^0(Q)$ are graded and an isomorphism $\Pi^0(Q)\cong
\C[x,y]\#\Gamma_1$ preserves the grading. From here one easily deduces that the isomorphism
$\C^2/\Gamma_1\cong \M_0^0(\delta,0)$ is equivariant with respect to the dilation $\C^\times$-actions.

\subsubsection{$V_n/\Gamma_n$ as a quiver variety}
Let us proceed now to the case of an arbitrary $n$. Let $\epsilon_0\in \C^{Q_0}$
be the coordinate vector at the extending vertex.

\begin{Prop}\label{Prop:quot_sing_quiv}
We have a $\C^\times$-equivariant isomorphism $\M^0_0(n\delta,\epsilon_0)\cong V_n/\Gamma_n$ (of Poisson schemes).
\end{Prop}
\begin{proof}
We have a diagonal embedding $T^*R(Q,\delta,0)^{\oplus n}\rightarrow T^*R(Q,n\delta,\epsilon_0)$,
compare to \ref{SSS_Hilb_quiver}, that restricts to $\mu_1^{-1}(0)^n\hookrightarrow
\mu^{-1}(0)$, where $\mu_1$ stands for the moment map $T^*R(Q,\delta,0)\rightarrow \mathfrak{gl}(\delta)^*$.
This gives rise to a $\mathfrak{S}_n$-invariant morphism
$\M^0_0(\delta,0)^n\rightarrow \M_0^0(n\delta,\epsilon_0)$ and hence to a morphism
$\iota:\C^{2n}/\Gamma_n=(\C^2/\Gamma_1)^n/\mathfrak{S}_n\rightarrow \M_0^0(n\delta,\epsilon_0)$.
One can show that this morphism is bijective. Also it is $\C^\times$-equivariant, where
the $\C^\times$-actions on $\C^{2n}/\Gamma_n, \M^0_0(n\delta,\epsilon_0)$
are induced from the dilation actions on $\C^{2n}, T^*R(Q,n\delta,\epsilon_0)$.
It follows that $\iota$ is finite. By Theorem \ref{Thm_quiv_norm}, $\M_0^0(n\delta,\epsilon_0)$
is normal and this implies that $\iota$ is an isomorphism.

We can make the isomorphism $\iota$ Poisson if we rescale it using the $\C^\times$-actions. This is
a consequence of the following lemma.
\end{proof}

\begin{Lem}{\cite[Lemma 2.23]{EG}}
Let $V$ be a symplectic vector space and $\Gamma\subset \operatorname{Sp}(V)$ be a finite subgroup
such that $V$ is symplectically irreducible, i.e., there are no proper symplectic $\Gamma$-stable
subspace in $V$. Then there are no nonzero brackets (=skew-symmetric bi-derivations) of degree $<-2$ on
$\C[V]^{\Gamma}$. Further, the space of brackets of degree $-2$ is one-dimensional.
\end{Lem}

One can ask why we use $\M_0^0(n\delta,\epsilon_0)$ instead of $\M_0^0(n\delta,0)$ in the proposition.
The reason is that the moment map for $T^*R(n\delta,\epsilon_0)$ is flat, this can be checked
using Theorem \ref{Thm_flat}.

\subsubsection{Symplectic resolutions of $V_n/\Gamma_n$}\label{SSS_sympl_res_quot}
Here we will study symplectic resolutions of $V_n/\Gamma_n$ constructed as non-affine Nakajima quiver
varieties for generic stability conditions $\theta$.

Let us consider the case $n=1$ first. Let $\bar{G}$ denote the quotient of $G=\GL(\delta)$
modulo the one-dimensional torus $T_{const}:=\{(x\operatorname{id}_{\C^{\delta_i}})_{i=0}^r, x\in \C^\times\}$.
Note that the $G$-action on $R:=R(Q,\delta,0)$ factors through $\bar{G}$.
Analogously to Nakajima's result explained in \ref{SSS_quiver_struct},
the group $\bar{G}$ acts freely on $\mu^{-1}(0)^{\theta-ss}$
if and only if $\theta\cdot \alpha\neq 0$ for every Dynkin root of $Q$ (these are the roots $\alpha\in \C^{Q_0}$
with $\alpha_0=0$). For such $\theta$, we get
a conical symplectic resolution $\M^\theta_0(\delta,0)\rightarrow \M_0^0(\delta,0)$, this can be deduced,
for example, from Theorem \ref{Thm:irred}. Of course, all these resolutions are isomorphic to the minimal resolution $\widetilde{\C^2/\Gamma_1}$: there are just no other symplectic resolutions.

Let us proceed to the case $n>1$. We get a projective morphism $\M_0^\theta(n\delta,\epsilon_0)
\rightarrow \M_0^0(n\delta,\epsilon_0)$. Theorem \ref{Thm:irred} no longer holds, in fact,
$\mu^{-1}(0)$ has $n+1$ irreducible components by \cite[Section 3.2]{GG}. Still, $\M^\theta_0(n\delta,\epsilon_0)
\rightarrow \M^0_0(n\delta,\epsilon_0)$ is a resolution of singularities. One just needs to
check that the fiber over a generic point in $\M^0_0(n\delta,\epsilon_0)$ consists of a single point.
A generic closed $G$-orbit in $\mu^{-1}(0)$  has a point of the form $r^1\oplus\ldots\oplus r^n$,
where $r^1,\ldots,r^n$ are pair-wise non-isomorphic simple representations of $\Pi^0(Q)$ of dimension $\delta$.
Then one can analyze the structure of the $G$-action near that orbit using a symplectic slice theorem,
see, for example,  \cite[Section 4]{CB_norm} or \ref{SSS_completions}
below.
This analysis shows that there is a unique semistable $G$-orbit
containing $Gr$ in its closure. So we see that $\M^\theta(n\delta,\epsilon_0)
\rightarrow \M^0_0(n\delta,\epsilon_0)$ is a conical symplectic resolution.

\subsubsection{Isomorphic resolutions}\label{SSS_resol_iso}
Now let us discuss how many resolutions we get. The stability condition $\theta$
is generic if $\theta\cdot\delta\neq 0$ and
$\theta\cdot v\neq 0$ for $v$ of the form $v=\alpha+m\delta$, where
$\alpha$ is a Dynkin root and $|m|<n$. So we get resolutions labeled by
the open cones in the complement to these hyperplanes in $\mathbb{R}^n$. However, some of these
resolutions are isomorphic: there is an  action of $W\times \Z/2\Z$ on $\Z^{Q_0}$
such that, for $\theta,\theta'$ lying in one orbit, the resolutions
$\M^\theta_0(n\delta,\epsilon_0)\rightarrow \M_0^0(n\delta,\epsilon_0),\M^{\theta'}_0(n\delta,\epsilon_0)
\rightarrow \M_0^0(n\delta,\epsilon_0)$ are isomorphic (here $W$ denotes the Weyl group
of the Dynkin diagram obtained from $Q$ by removing the vertex $0$). This is a special case
of a construction due to Namikawa, \cite{Namikawa1}, that we are going to explain now.

Let $X\rightarrow X_0$ be an arbitrary symplectic resolution. The variety $X_0$ has finitely
many symplectic leaves, \cite{Kaledin}. Let $\mathcal{L}_1,\ldots,\mathcal{L}_k$ be the leaves
of codimension $2$. Take formal slices $\mathsf{S}^{\wedge}_1,\ldots,\mathsf{S}^{\wedge}_k$ through $\mathcal{L}_1,
\ldots,\mathcal{L}_k$. The slices are formal neighborhoods of $0$ in Kleinian singularities $\mathsf{S}_1,
\ldots,\mathsf{S}_k$. From these Kleinian singularities one produces Weyl groups $\tilde{W}_1,\ldots,
\tilde{W}_k$ (of the same types as the singularities) acting on the spaces
$H^2(\mathsf{S}_k,\C)$ identified with their reflection representations
$\tilde{\mathfrak{h}}_i$. The fundamental group $\pi_1(\mathcal{L}_i)$ acts on the irreducible components
of the exceptional divisor in $\mathsf{S}_i$. Hence it also acts on $\tilde{W}_i$ (by diagram automorphisms)
and on $\tilde{\mathfrak{h}}_i$. Set $W_i:=\tilde{W}_i^{\pi_1(\mathcal{L}_i)}, \h_i:=\tilde{\h}_i^{\mathcal{L}_i}$
so that $W_i$ is a crystallographic reflection group and $\mathfrak{h}_i$ is its reflection
representation. There is a natural restriction map $H^2(X)\rightarrow \h:=\bigoplus_i \h_i$.
Namikawa proved that this map is surjective. Furthermore, he has constructed a $W:=\prod_{i}W_i$-action
on $H_{DR}^2(X)$ that makes the map equivariant  and is trivial on the kernel.

Let us return to our situation. The symplectic leaves in $V/\Gamma$ are in one-to-one correspondence
with conjugacy classes of stabilizers of points in $V$. The leaf corresponding to $\Gamma'\subset \Gamma$
is the image of $V^{\Gamma',reg}:=\{v\in V| \Gamma_v=\Gamma'\}$ under the quotient morphism
$\pi:V\rightarrow V/\Gamma$. The leaf is identified with $V^{\Gamma',reg}/N_\Gamma(\Gamma')$.
So, in the case when $V=V_n$ and $\Gamma=\Gamma_n$, we get two leaves of codimension 2
(provided $\Gamma_1\neq \{1\}$, in that case we get just one leaf of codimension $2$).
One of them, say $\mathcal{L}_1$, corresponds to $\Gamma_1\subset \Gamma_n$ (the stabilizer
of a point of the form $(0,p_1,\ldots,p_{n-1})$, where $p_1,\ldots,p_{n-1}$ are pairwise different
points of $\C^2$). The other, say $\mathcal{L}_2$, corresponds to $\mathfrak{S}_2$ (the stabilizer
of $(p_1,p_1,p_2,\ldots,p_{n-1})$). The fundamental group actions from the previous
paragraph are easily seen to be trivial. So we get $W_1=W, W_2=\Z/2\Z$. Further, $H^2(X)=\C^{Q_0}$
and $\h_1=\{(x_i)_{i\in Q_0}| x\cdot \delta=0\}, \h_2=\C\delta$. The group $W_2$ acts on $\C\delta$
by $\pm 1$, while $\h_1$ is identified with the Cartan space for $W_1$ via $(x_i)_{i\in Q_0}\mapsto
\sum_{i=1}^r x_i\omega_i^\vee$, where we write $\omega_i^\vee$ for the fundamental coweights.

Let us remark that the $W$-action can be recovered by using the quiver variety setting as well,
see \cite{Maffei} and \cite{quant} for more detail.

\subsection{Quantum Hamiltonian reduction}\label{SS_Quant_Ham_red}
Here we will explain a quantum counterpart of the constructions of the previous section.

\subsubsection{Quantum Hamiltonian reduction: algebra level}
Let $\A$ be an associative algebra, $\g$ a Lie algebra and $\Phi: \g\rightarrow \A$
be a  Lie algebra homomorphism. Then, for a character $\lambda$ of $\g$, set
$\mathcal{I}_\lambda:=\A\{x-\langle\lambda,x\rangle, x\in \g\}$, this is
a left ideal in $\A$ that is stable under the adjoint action of $\g$. We set
$\A\red_\lambda\g:=(\A/\mathcal{I}_\lambda)^{\g}$. This space has a natural
associative product given by $(a+\mathcal{I}_\lambda)(b+\mathcal{I}_\lambda):=
ab+\mathcal{I}_\lambda$. With this product, $\A\red_\lambda\g$ becomes naturally
isomorphic to $\operatorname{End}_{\A}(\A/\mathcal{I}_\lambda)^{opp}$, an element
$a+\mathcal{I}_\lambda$ gets mapped to the unique endomorphism sending $1+\mathcal{I}_\lambda$
to $a+\mathcal{I}_\lambda$. We also have a universal variant of quantum Hamiltonian reduction:
$\A\red\g:=(\A/\A\Phi([\g,\g]))^\g$.

Now suppose $\A$ is a filtered quantization of $\C[X]$, where $X$ is an an affine Poisson variety
(we assume that the bracket on $\C[X]$ has degree $-1$). Suppose that $G$ acts on $X$
in a Hamiltonian way and the functions $\mu^*(\xi)$ have degree $1$ for all $\xi\in \g$. By a quantization of
the Hamiltonian $G$-action on $\C[X]$ we mean  a rational $G$-action on $\A$ together with a $G$-equivariant map $\Phi:\g\rightarrow \A$
  such that
\begin{enumerate}
\item[(i)] the filtration on $\A$ is $G$-stable and  the isomorphism
$\gr\A\cong \C[X]$ is $G$-equivariant,
\item[(ii)] $\Phi(\xi)$ lies in $\A_{\leqslant 1}$ and coincides with $\mu^*(\xi)$ modulo $\A_{\leqslant 0}$,
\item[(iii)] and   $[\Phi(\xi),\cdot]=\xi_{\A}$, where $\xi_{\A}$ is the derivation of $\A$ coming from the $G$-action.
\end{enumerate}
Note that $\gr\mathcal{I}_\lambda\supset I:=\C[X]\mu^*(\g)$ and so we have a surjective homomorphism
$\C[X\red_0 G]\twoheadrightarrow \gr\A\red_\lambda G$. We want to get a sufficient condition
for $\gr\mathcal{I}_\lambda= I$ for all $\lambda$.

\begin{Lem}\label{Lem:reg_seq}
Let $\xi_1,\ldots,\xi_n$ be a basis in $\g$. Suppose $\mu^*(\xi_1),\ldots,\mu^*(\xi_n)$ form a regular sequence.
Then $\gr \mathcal{I}_\lambda=I$ for any $\lambda$.
\end{Lem}
\begin{proof}
The proof is based on the observation that the 1st homology in the Koszul complex associated to
$\mu^*(\xi_1),\ldots,\mu^*(\xi_n)$ is zero. In other words, if $f_1,\ldots,f_n\in \C[X]$ are such that
$\sum_{i=1}^n f_i\mu^*(\xi_i)$, then there are $f_{ij}\in \C[X]$ with $f_{ij}=-f_{ji}$
and $f_i=\sum_{j=1}^n f_{ij}\mu^*(\xi_j)$. Details of the proof are left to the reader.
\end{proof}

So if $G$ is reductive and the assumptions of Lemma \ref{Lem:reg_seq} hold, then $\A\red_\lambda \g$
is a filtered quantization of $\C[X\red_0 G]$.

We can also give the definition of a quantization of a Hamiltonian action in the setting of formal quantizations.
One should modify (i)-(iii) as follows. In  (i) one requires the $G$-action to be $\C[[\hbar]]$-linear
and the isomorphism $\A_\hbar/\hbar\A_\hbar\cong \C[X]$ has to be $G$-equivariant. In (ii), one requires
that $\Phi(\xi)$ coincides with $\mu^*(\xi)$ modulo $\hbar$. In  (iii) one requires
$\frac{1}{\hbar}[\Phi(\xi),\cdot]=\xi_{\A_\hbar}$. We then can consider reductions
of the form $\A_\hbar\red_{\lambda(\hbar)}G$, where $\lambda(\hbar)$ is an element in $(\g^{*G})[[\hbar]]$.
If $G$ is reductive, and the elements $\mu^*(\xi_i)-\langle\lambda(0),\xi_i\rangle, i=1,\ldots,n$,
form a regular sequence in $\C[X]$, then $\A_\hbar\red_{\lambda(\hbar)}G$ is a formal quantization
of $\C[X\red_{\lambda(0)} G]$.

\subsubsection{Quantum Hamiltonian reduction: sheaf level}\label{SSS_quant_red_sheaf}
Let $X$ be a smooth affine symplectic algebraic variety equipped with a Hamiltonian action of $G$
and let $\theta$ be a character of $G$. Assume that, for a basis $\xi_1,\ldots,\xi_n$
of $\g$, the elements $\mu^*(\xi_1),\ldots,\mu^*(\xi_n)$ form a regular sequence at
points of $\mu^{-1}(0)^{\theta-ss}$. Let $\mathcal{D}_\hbar$ be a formal quantization
of $\mathcal{O}_X$.  Our goal is to define a (formal)
quantization $\mathcal{D}_\hbar\red^\theta_{\lambda(\hbar)} G$ of $X\red^\theta_0 G$
(so $\lambda(0)=0$).

Recall that it is enough to define the following data:
\begin{enumerate}
\item For an open affine covering  $X\red^\theta_\lambda G:=\bigcup_i Y_i$, the algebras
of sections $\Gamma(Y_i,\mathcal{D}_\hbar\red^\theta_{\lambda(\hbar)} G)$ that quantize $Y_i$,
\item and identifications  $\Gamma(Y_i,\mathcal{D}_\hbar\red^\theta_{\lambda(\hbar)} G)_{Y_i\cap Y_j}\cong
\Gamma(Y_j,\mathcal{D}_\hbar\red^\theta_{\lambda(\hbar)} G)_{Y_i\cap Y_j}$ satisfying cocycle conditions.
\end{enumerate}
Recall that we can choose an open covering by setting $Y_i:=X_{f_i}\red_0 G$, where
polynomials $f_i\in \C[X]^{G,n_i\theta}$ are such that $X^{\theta-ss}=\bigcup_i X_{f_i}$.
Then we set $\Gamma(Y_i,\mathcal{D}_\hbar\red^\theta_{\lambda(\hbar)} G):=\Gamma(X_{f_i},\mathcal{D}_\hbar)\red_{\lambda(\hbar)}G$. The sections of the corresponding
sheaf on $Y_i\cap Y_j$ are easily seen to be  $\Gamma(X_{f_i}\cap X_{f_j},\mathcal{D}_\hbar)\red_{\lambda(\hbar)}G$
and this yields the gluing maps.

Now let us discuss the period map mentioned in \ref{SSS_quant_classif}.
Suppose that the  $G$-action on $\mu^{-1}(0)^{\theta-ss}$
is free so that $X\red_0^\theta G$ is smooth and symplectic. In this case we have a period map associated
to the quantization of $\mathcal{D}_\hbar\red_{\lambda(\hbar)}^\theta G$. Assume, for simplicity,
that $\lambda(\hbar):=\lambda\hbar$ for $\lambda\in \g^{*G}$ -- this is the most interesting case,
for example, it is the only case that appears when we work with the filtered setting. Further,
assume that $\mathcal{D}_\hbar$ is canonical, i.e., has period $0$. Recall that this means
the existence of a parity anti-automorphism, let us denote it by $\varrho$. Finally, assume that
$\Phi$ is {\it symmeterized}, meaning that $\varrho\circ \Phi=\Phi$, this can be achieved by
modifying $\Phi$. Then the period of $\mathcal{D}_\hbar\red_{\lambda\hbar}^\theta G$ equals
to the Chern class associated to $\lambda$ (if $\lambda$ integrates to a character of $G$, then it defines
the line bundle on $X\red_0^\theta G$, in general, we extend the notion of a Chern class by linearity).
This was essentially checked in \cite[Sections 3.2,5.4]{quant}.

\subsubsection{Algebra vs sheaf level}
We need to relate the sheaf $\mathcal{D}_\hbar\red_{\lambda(\hbar)}^\theta G$ to the algebra
$\mathcal{D}_\hbar\red_{\lambda(\hbar)} G$. What one could expect is that the algebra is the global
sections (or even better, the derived global sections) of the sheaf. Let us provide some sufficient
conditions for this to hold.

\begin{Prop}\label{Prop:quant_alg_vs_sheaf}
Assume, for simplicity, that $\lambda(0)=0$. Further, suppose that the following holds.
\begin{enumerate}
\item The moment map $\mu$ is flat.
\item $X\red_0 G$ is a normal reduced scheme.
\item $X\red^\theta_0 G\rightarrow X\red_0 G$ is a resolution of singularities.
\end{enumerate}
Then $R\Gamma(\mathcal{D}_\hbar\red^\theta_{\lambda(\hbar)}G)\cong \mathcal{D}_\hbar\red_{\lambda(\hbar)}G$.
\end{Prop}
\begin{proof}
By (3) and the Grauert-Riemmenschneider theorem, the higher cohomology of $\mathcal{O}_{X\red^\theta_0 G}$ vanish. This implies that the higher cohomology of $\mathcal{D}_\hbar\red_{\lambda(\hbar)}^\theta G$ vanish. Moreover,
$\Gamma(\mathcal{D}_\hbar\red_{\lambda(\hbar)}^\theta G)/(\hbar)\cong \C[X\red^\theta_0 G]$. By (2) and (3), the right hand side is naturally identified with $\C[X\red_0 G]$. By (1), $(\mathcal{D}\red_{\lambda(\hbar)}G)/(\hbar)=\C[X\red_0 G]$.
Besides, we have a natural homomorphism $\mathcal{D}\red_{\lambda(\hbar)}G\rightarrow
\Gamma(\mathcal{D}_\hbar\red_{\lambda(\hbar)}^\theta G)$. Modulo $\hbar$, this homomorphism
is the identity. The source algebra is complete and separated in the $\hbar$-adic topology,
and the target algebra is flat over $\C[[\hbar]]$. It follows that the homomorphism
$\mathcal{D}\red_{\lambda(\hbar)}G\rightarrow
\Gamma(\mathcal{D}_\hbar\red_{\lambda(\hbar)}^\theta G)$ is an isomorphism.
\end{proof}

\subsubsection{Isomorphism theorem}
Recall a $\C^\times$-equivariant
isomorphism $\C^{2n}/\Gamma_n\cong \M^\theta(n\delta,\epsilon_0)$ of  Poisson varieties.
The left hand side admits a family of quantizations, $eH_{1,c}e$, and so does the right hand side,
there quantizations are the quantum Hamiltonian reductions $D(R)\red_\lambda G$, where we use the symmetrized
quantum comoment map $\Phi(\xi)=\frac{1}{2}(\xi_R+\xi_{R^*})$. In fact, these two families
are the same. Let us state a precise result to be proved in Section \ref{SS_iso_proof} (using Procesi bundles).
We write ${\bf c}$ for $$\frac{1}{|\Gamma_1|}\left(1+\sum_{\gamma\in \Gamma_1\setminus \{1\}}c(\gamma)\gamma\right),$$
where $c(\gamma):=c_i$ for $\gamma\in S_i$ (recall that $S_0$ is the conjugacy class of a reflection
in $\mathfrak{S}_n\subset \Gamma_n$ and $S_1,\ldots,S_r$ are conjugacy classes of elements of $\Gamma_1\subset \Gamma_n$).

\begin{Thm}\label{Thm:iso}
We have a filtered algebra isomorphism  $eH_{1,c}e\cong D(R)\red_\lambda G$ that is the identity
on the level of associated graded algebras (we consider the filtration on $D(R)\red_\lambda G$
induced from the Bernstein filtration on $D(R)$, where $\deg R=\deg R^*=1$). Here $\lambda:=\sum_{i=0}^r
\lambda_i \operatorname{tr}_i$
is recovered from $c$ by the following formulas:
\begin{equation}\label{eq:param_corresp}
\lambda_i:=\operatorname{tr}_{N_i}{\bf c},\, i=1,\ldots,r, \quad
\lambda_0:=\operatorname{tr}_{N_0}{\bf c}-\frac{1}{2}(c_0+1),
\end{equation}
where in the $n=1$ case one needs to put $c_0=1$.
\end{Thm}
For $n=1$, this theorem was proved by Holland in \cite{Holland}. The case of $\Gamma_1=\{1\}$ was handled
in \cite{EG,GG} (\cite{EG} proved a weaker statement and then in \cite{GG} the proof was completed).
The case of cyclic $\Gamma_1$ was done in \cite{Oblomkov,Gordon_cyclic}. In \cite{EGGO} the proof
was completed: they considered the case when $Q$ is a bi-partive graph. Let us note that in these
papers formulas  look different from (\ref{eq:param_corresp}):
they use the quantum  comoment map $\Phi(\xi)=\xi_R$. A uniform and more
conceptual proof was given in \cite{quant} using Procesi bundles.

Theorem \ref{Thm:iso} is of crucial importance for the representation theory of the algebras $H_{1,c}$.
It turns out that the representation theory of the algebras $D(R)\red_\lambda G$ (actually, of sheaves
$D_R\red_\lambda^\theta G$) is easier to study. The main ingredient here is the geometry of the quiver varieties
$\M^\theta(v,\epsilon_0)$. Using this, in \cite{BL}, the author and Bezrukavnikov have proved a conjecture
of Etingof, \cite{Etingof_affine}, on the number of the finite dimensional irreducible representations
of $H_{1,c}$.

\subsubsection{Automorphisms}\label{SSS_quant_auto}
Here we are going to explain a quantum version of Namikawa's construction recalled in \ref{SSS_resol_iso}.
In the complete generality this construction was given in \cite[Section 3.3]{BPW}.

Let $X$ be a conical symplectic resolution of $X_0$. Let $\tilde{X}$ be its universal deformation over
$H^2_{DR}(X)$ and let $\tilde{\mathcal{D}}_\hbar$ be the canonical quantization of $\tilde{X}$.
Let $\tilde{\A}_\hbar$ denote the $\C^\times$-finite part of  $\Gamma(\tilde{\mathcal{D}}_\hbar)$.
Then Namikawa's Weyl group $W$ acts on $\tilde{\A}_\hbar$ by graded $\C[\hbar]$-algebra automorphisms
preserving $H^2_{DR}(X)^*$. Moreover, the action on $H^2_{DR}(X)^*$ is as explained in \ref{SSS_resol_iso}.

\subsection{Quantum Hamiltonian reduction for Frobenius constant quantizations}\label{SS_quant_Ham_Frob}
In this section, we will consider the situation in characteristic $p$. Our main result is
that a quantum GIT Hamitlonian reduction under a free Hamiltonian action is again Frobenius
constant.
\subsubsection{GIT in characteristic $p$}
The definition of a reductive group (one with trivial unipotent radical) makes sense in all characteristics.
A crucial difficulty of dealing with reductive groups in positive characteristic is that
their rational representations are no longer completely reducible, in general. The groups for which
the complete reducibility holds are called {\it linearly reductive}. Tori are
still linearly reductive independently of the characteristic. We need to deal with GIT for reductive groups
(such as products of $\GL$'s) and so we need to explain how this works in positive characteristic.

It turns out  that reductive groups satisfy a weaker condition than being linearly reductive,
they are {\it geometrically reductive}. This was conjectured by Mumford and proved by Haboush, \cite{Haboush}.
To state the condition of being geometrically reductive, let us reformulate the linear reductivity first:
a group $G$ is called linearly reductive, if, for any linear $G$-action on a vector space $V$
and any fixed point $v\in V$, there is $f\in (V^*)^G$ with $f(v)\neq 0$. A group
$G$ is called {\it geometrically reductive} if instead of  $f\in (V^*)^G$, one can find $f\in S^r(V^*)^G$
(for some $r>0$) with $f(v)\neq 0$.

This condition is enough for many applications. For example, if $X$ is an affine algebraic variety acted on
by a reductive (and hence geometrically reductive) group $G$, then $\Fi[X]^G$ is finitely generated.
So we can consider the quotient morphism  $X\rightarrow X\quo G$. This morphism is surjective and separates
the closed orbits. Moreover, if $X'\subset X$ is a $G$-stable subvariety, then the natural morphism
$X'\quo G\rightarrow X\quo G$ is injective with closed image.

The claim about the properties of the quotient morphism in the previous paragraph can be deduced from the
following lemma, \cite[Lemma A.1.2]{MFK}.

\begin{Lem}\label{Lem:char_p_inv}
Let $G$ be a geometrically reductive group acting on a finitely generated commutative $\Fi$-algebra $R$
rationally and by algebra automorphisms. Let $I\subset R$ be a $G$-stable ideal and $f\in (R/I)^G$. Then
there is $n$ such that $f^{p^n}$ lies in the image of $R^G$ in $(R/I)^G$.
\end{Lem}

In characteristic $p$, we can still speak about unstable and semistable points for reductive group actions
on vector spaces, about GIT quotients, etc.

Another very useful and powerful result of Invariant theory in characteristic $0$ is Luna's \'{e}tale
slice theorem, see, e.g., \cite[6.3]{PV}.
There is a version of this theorem in characteristic $p$ due to Bardsley and Richardson,
see \cite{BR}. We will need a consequence of this theorem dealing with free actions.

Recall that, in characteristic $0$, an action of an algebraic group $G$ on a variety $X$ is called  free
if the stabilizers of all points are trivial. In characteristic $p$ one should give this definition more carefully:
the stabilizer may be a nontrivial finite group scheme with a single point. An example is provided by the left
action of $G$ on $G^{(1)}$, we will discuss a closely related question in the next subsection. We have the following
three equivalent definitions of a free action.

\begin{itemize}
\item For every $x\in X$, the stabilizer $G_x$ equals $\{1\}$ as a group scheme.
\item For every $x\in X$, the orbit map $G\rightarrow X$ corresponding to $x$
is an isomorphism of algebraic varieties.
\item For every $x\in X$, $G_x$ coincides with $\{1\}$ as a set and the stabilizer of $x$
in $\g$ is trivial.
\end{itemize}

The following is a weak version of the slice theorem that we need.

\begin{Lem}\label{Lem:slice}
Let $X$ be a smooth affine variety equipped with a free action of a reductive algebraic group $G$.
Then the quotient morphism $X\rightarrow X/G$ is a principal $G$-bundle in \'{e}tale topology.
\end{Lem}

\subsubsection{Quiver varieties}\label{SSS_quiv_var_char_p}
Let us now discuss Nakajima quiver varieties in characteristic $p\gg 0$. We have a finite localization
$\mathfrak{R}$ of $\Z$ with the following properties:
\begin{enumerate}
\item $R$ together with the $G$-action and $\mu$ are defined over $\mathfrak{R}$.
\item $\mu^{-1}(0)^{\theta-ss}$ and the $G$-bundle $\mu^{-1}(0)^{\theta-ss}\rightarrow
\mu^{-1}(0)^{\theta-ss}/G$ are defined over $\mathfrak{R}$.
\end{enumerate}
For an $\mathfrak{R}$-algebra $\mathfrak{R}'$, let $R_{\mathfrak{R}'},G_{\mathfrak{R}'},
\mu_{\mathfrak{R}'}$ etc. denote the $\mathfrak{R}'$-forms of the corresponding
objects. Let us write $X_{\mathfrak{R}}$ for an $\mathfrak{R}$-form of
$\mu^{-1}(0)^{\theta-ss}/G$. After a finite localization
of $\mathfrak{R}$, we can achieve that $X_{\mathfrak{R}}$ is a symplectic scheme over
$\operatorname{Spec}(\mathfrak{R})$ with $\C\otimes_{\mathfrak{R}}\Gamma(X_{\mathfrak{R}}, \mathcal{O}_{X_{\mathfrak{R}}})\xrightarrow{\sim} \C[X_{\C}]$ and
$H^i(X_{\mathfrak{R}}, \mathcal{O}_{X_{\mathfrak{R}}})=0$ for $i>0$.

For $\mathfrak{R}'$, we can take $\Fi:=\overline{\Fi}_p$ when $p$ is large enough. So we get a symplectic
$\Fi$-variety $\M^\theta_0(n\delta,1)_{\Fi}$ that is naturally identified with
$T^*R_{\Fi}\red^\theta_0 G_{\Fi}$ as well as with  $\Spec(\Fi)\times_{\Spec(\mathfrak{R})}X_{\mathfrak{R}}$.
For $p\gg 0$, we get $\Fi[X_{\Fi}]=\Fi\otimes_{\mathfrak{R}}\Gamma(X_{\mathfrak{R}}, \mathcal{O}_{X_{\mathfrak{R}}})$
and $H^i(X_{\Fi}, \mathcal{O}_{X_{\Fi}})=0$.

We can take a finite algebraic extension of $\mathfrak{R}$ and assume that the $\Gamma_n$-module
$\C^{2n}$ is defined over $\mathfrak{R}$. Now we claim that (again for $p\gg 0$)
$\M^\theta_0(n\delta,1)_{\Fi}$ is a symplectic resolution of $\Fi^{2n}/\Gamma_n$.
This  follows  from the claim that both $\Gamma(X_{\mathfrak{R}},\mathcal{O}_{X_{\mathfrak{R}}}),
\mathfrak{R}[\underline{x},\underline{y}]^{\Gamma_n}$ are $\mathfrak{R}$-forms of $\C[\underline{x},\underline{y}]^{\Gamma_n}$
so they coincide after some finite localization of $\mathfrak{R}$.

\subsubsection{Quantum Hamiltonian reduction}\label{SSS_qHam_char_p}
Now suppose that $R$ is a symplectic vector space over $\Fi$, $G$ is a reductive group over $\Fi$
acting on $R$ and $\theta$ is a character of $G$. We suppose that $G$ acts freely on $\mu^{-1}(0)^{\theta-ss}$.
We are going to define a Frobenius constant quantization $D_R\red^\theta_\lambda G$ of $T^*R\red^\theta_0 G$,
where $\lambda\in \operatorname{Hom}(G,\Fi^\times)\otimes_{\Z}\Fi_p\hookrightarrow \g^{*G}$. The associated filtered
quantization of $T^*R\red_0^\theta G$ will be a quantization obtained by quantum Hamiltonian reduction,
see \ref{SSS_quant_red_sheaf}. We note that for $\lambda\not\in \operatorname{Hom}(G,\Fi^\times)$ we do not get a {\it Frobenius constant} quantization of $T^*R\red_0^\theta G$.

Consider the Frobenius twist $G^{(1)}$. It is a group and the morphism $\operatorname{Fr}:G\rightarrow G^{(1)}$
is a group epimorphism. Its kernel (a.k.a. the Frobenius kernel) $G_1$ is a finite group scheme whose Lie
algebra coincides with $\g$.

The action of $G$ on $R$ induces an action of $G^{(1)}$ on $R^{(1)}$. The $G^{(1)}$-action
on $T^*R^{(1)}$ is Hamiltonian with moment map $\mu^{(1)}:T^*R^{(1)}\rightarrow \g^{(1)*}$ induced by $\mu$.
Consider the sheaf $D_R\red_\lambda^\theta G_1$ (a subquotient of $D_R$) on $T^*R^{(1)\theta-ss}$.
One can show, see \cite[Section 3.6]{BFG}, that it is supported on  $\left(\mu^{(1)}\right)^{-1}(0)$,
here we use that $\lambda\in \operatorname{Hom}(G,\Fi^\times)\otimes_{\Z}\Fi_p$. Moreover, it is a $G^{(1)}$-equivariant
Azumaya  algebra on $\left(\mu^{(1)}\right)^{-1}(0)$. The descent of this algebra to
$(T^*R\red_0^\theta G)^{(1)}=T^*R^{(1)}\red_0^\theta G^{(1)}$ is an Azumaya algebra with a filtration
induced from that on $D_R$. We have a natural homomorphism
$\gr (D_R\red_\lambda^\theta G_1)\rightarrow \operatorname{Fr}_* \mathcal{O}_{T^*R\red_0^\theta G_1}$.
To show that it is an isomorphism one uses that the action of $G_1$ is free (that yields the required
cohomology vanishing). This isomorphism implies
$\gr (D_R\red_\lambda^\theta G)\xrightarrow{\sim} \operatorname{Fr}_* \mathcal{O}_{T^*R\red_0^\theta G}$.
So $D_R\red_\lambda^\theta G$ is indeed a Frobenius constant quantization.

Note that if $\lambda\not\in \operatorname{Hom}(G,\Fi^\times)\otimes_{\Z}\Fi_p$, then
$D_R\red_\lambda^\theta G_1$ is supported on a nonzero fiber of $\mu^{(1)}$, see \cite[Section 3.6]{BFG}
for details, and so  $D_R\red_\lambda G$ is no longer a Frobenius constant quantization of
$X\red_0^\theta G$.

\section{Existence and classification of Procesi bundles}\label{S_Procesi_SRA}
In this section we construct and classify Procesi bundles on $X=\M^\theta(n\delta,\epsilon_0)$
and also prove Theorem \ref{Thm:iso}.

In Section \ref{SS_Procesi} we construct a Procesi bundle on $X$.
The case $n=1$ is relatively easy,
it was done in \cite{KaVa}. For $n>1$, we follow \cite{BK_Procesi}. A key step here is to
construct a special Frobenius constant quantization of $X_{\Fi}$, where $\Fi$
is an algebraically closed field of large enough positive characteristic.
This quantization provides a suitable version of derived McKay equivalence
and using this equivalence we can produce a Procesi bundle over $\Fi$. Then
we lift it to characteristic $0$.

In Section \ref{SS_SRA} we prove that Symplectic reflection algebras
satisfy PBW property and, in some sense, the family of SRA $H_{t,c}$
is universal with this property. The proof is based on computing
relevant graded components in the Hochschild cohomology of $SV\#\Gamma$.

Theorem \ref{Thm:iso} is proved in Section \ref{SS_iso_proof}. Using the Procesi bundle, we show that each
algebra $D(R)\red_\lambda G$ is isomorphic to some $eH_{1,c}e$. Then the task
is to show that the correspondence between the parameters $\lambda$ and the
parameters $c$ is as in Theorem \ref{Thm:iso}. We first do this for $n=1$.
Then we reduce the case of $n>1$ to $n=1$ by studying completions of the algebras
involved. This allows to show that the map between the parameters is conjugate
to that in Theorem \ref{Thm:iso} up to a conjugation under an action of the
group $W\times \Z/2\Z$, where $W$ is the Weyl group of the finite part of the
quiver $Q$. But from \ref{SSS_quant_auto} we know that this action lifts to
an action on the universal reduction $D(R)\red G$ by automorphisms. This completes
the proof of Theorem \ref{Thm:iso}.

Then, in Section \ref{SS_Procesi_classif}, we classify Procesi bundles. Namely, we show that, when $n>1$,
there are $2|W|$ different Procesi bundles on $X$. For this, we use
Theorem \ref{Thm:iso} to produce this number of bundles. And then we
use techniques used in the proof to show that the number cannot exceed
$2|W|$. Further, we show that each $X$ carries a distinguished
Procesi bundle.

\subsection{Construction of Procesi bundles}\label{SS_Procesi}
\subsubsection{Baby case: $n=1$}
In this case it is easy to construct a vector bundle of required rank on $X$.
Namely, for $i=0,\ldots,r$, let $U_i$ be the $G$-module $\C^{\delta_i}$
and let $\mathcal{U}_i$ be the corresponding vector bundle on $X$.
We set $\mathcal{P}:=\bigoplus_{i=0}^r \mathcal{U}_i^{\delta_i}$. It follows from
results of Kapranov and Vasserot, \cite{KaVa}, that this bundle satisfies the axioms
of a Procesi bundle.
\subsubsection{Procesi bundles and derived McKay equivalence}\label{SSS_derived_McKay}
Before we proceed to constructing Procesi bundles in general, let us explain their connection
to  derived Mckay equivalences, i.e., equivalences $D^b(\operatorname{Coh}X)\xrightarrow{\sim}
D^b(\mathbb{K}[V_n]\#\Gamma_n)$, here $\mathbb{K}$ stands for the base field.

\begin{Prop}\label{Prop:McKay}
Let $\mathcal{P}$ be a Procesi bundle on $X$. Then the functor $R\operatorname{Hom}_{\Str_X}(\mathcal{P},\bullet)
$ is a derived equivalence $D^b(\operatorname{Coh}X)\rightarrow D^b(\mathbb{K}[V_n]\#\Gamma\operatorname{-mod})$.
\end{Prop}

The proof is based on the following more general result (Calabi-Yau trick) of (in this form)
Bezrukavnikov and Kaledin.

\begin{Prop}{\cite[Proposition 2.2]{BK_Procesi}}\label{Prop:CY_trick}
Let $X$ be a smooth variety, projective over an affine variety, with trivial canonical class.
Furthermore, let $\A$ be an Azumaya algebra over $X$ such that $\Gamma(\A)$ has finite homological
dimension and $H^i(X,\mathcal{A})=0$ for $i>0$. Then the functor $R\Gamma: D^b(\operatorname{Coh}(X,\A))
\rightarrow D^b(\Gamma(\A)\operatorname{-mod})$ is an equivalence.
\end{Prop}

Proposition \ref{Prop:McKay} follows from Proposition \ref{Prop:CY_trick} with $\A=\mathcal{E}nd(\mathcal{P})$.

Now suppose that we have a derived equivalence $\iota:D^b(\operatorname{Coh}(X))\xrightarrow{\sim}
D^b(\mathbb{K}[V]\#\Gamma_n\operatorname{-mod})$. Assume $\mathcal{P}':=\iota^{-1}(\mathbb{K}[V]\#\Gamma_n)$
is a vector bundle. Then $\End_{\mathcal{O}_X}(\mathcal{P}')=\mathbb{K}[V]\#\Gamma$
and $\operatorname{Ext}^i(\mathcal{P}',\mathcal{P}')=0$ for $i>0$. So $\mathcal{P}'$
is, basically, a Procesi bundle (it also needs to be $\mathbb{K}^\times$-equivariant, but we will
see below that this  always can be achieved). In fact, this is roughly, how the construction
of  a Procesi bundle will work, although it is  more involved and technical.

\subsubsection{Quantization of $X$}\label{SSS_quant_X_char_p}
Here and in \ref{SSS_constr_Procesi_p} everything is going to be over an algebraically closed field
$\Fi$ of characteristic $p\gg 0$.  The first step in the construction of a Procesi bundle is to produce a
Frobenius constant quantization of $X$ with special properties.

\begin{Prop}\label{Prop:quant}
There is a Frobenius constant quantization $\mathcal{D}$ of $X$ such that $\Gamma(\mathcal{D})=\Weyl(V_n)^{\Gamma_n}$
(an isomorphism of filtered algebras over $\Fi[X^{(1)}]=\Fi[V_n^{(1)}]^{\Gamma_n}$).
\end{Prop}

Note that this proposition can be thought as a special case of the characteristic $p$ version of
Theorem \ref{Thm:iso}.  Here $\Gamma(\mathcal{D})$ is an analog of $D(R)\red_\lambda G$
(indeed, the latter is the algebra of global sections of some filtered quantization
of $X_{\C}$, see Proposition \ref{Prop:quant_alg_vs_sheaf}),
while $\Weyl(V_n)^{\Gamma_n}$ is the characteristic $p$ analog of
$eH_{1,0}e$.

In fact, the following is true.

\begin{Lem}\label{Lem:iso_char_p_red}
Theorem \ref{Thm:iso} (for $c=0$)  implies Proposition \ref{Prop:quant}.
\end{Lem}
\begin{proof}
First, let us see that we get an isomorphism $\Gamma(\mathcal{D})\cong \Weyl(V_n)^{\Gamma_n}$
of filtered algebras that is the identity on the associated graded algebra. Set
$\mathcal{D}:=D(R)\red_{\lambda}^\theta G$, where $\lambda$ is the parameter corresponding to $c=0$.

The algebra
$\Weyl(V_{n,\C})^{\Gamma_n}$ is finitely generated and so an isomorphism in Theorem
\ref{Thm:iso} is defined over some finitely generated subring $\mathfrak{R}$ of $\C$.
We can enlarge $\mathfrak{R}$ and assume that we are in the situation described in \ref{SSS_quiv_var_char_p}.
We can form filtered quantizations $\mathcal{D}'_{\C}, \mathcal{D}'_{\mathfrak{R}}, \mathcal{D}'_{\Fi}$
of $X_{\C}, X_{\mathfrak{R}}, X$. Both $\mathcal{D}'_{\C}, \mathcal{D}'_{\Fi}$ are obtained as suitable
completions of base changes of $\mathcal{D}'_{\mathfrak{R}}$ (completions are necessary because
of our condition on the filtration in the definition of a filtered quantization,
see \ref{SSS_filt_quant_sheaf}). In particular, $D(R_{\C})\red_\lambda G_{\C}=(\Gamma(\mathcal{D})=)
\C\otimes_{\mathfrak{R}}\Gamma(\mathcal{D}'_{\mathfrak{R}})$, while $\Gamma(\mathcal{D})=(\Gamma(\mathcal{D}'_{\Fi})=)\Fi\otimes_{\mathfrak{R}}\Gamma(\mathcal{D}'_{\mathfrak{R}})$.

So we can reduce an isomorphism from Theorem \ref{Thm:iso} (for $c=0$) mod $p\gg 0$ and get an isomorphism  $\Gamma(\mathcal{D})\cong \Weyl(V_n)^{\Gamma_n}$.
What remains to show is that this isomorphism is $\Fi[V_n^{(1)}]^{\Gamma_n}$-linear.
The first step here is to show that $\Fi[V_n^{(1)}]^{\Gamma_n}$ is the center
of $\Weyl(V_n)^{\Gamma_n}$. It is enough to check that $\Fi[V_n^{(1)}]^{\Gamma_n}$
coincides with the center of the Poisson algebra $\Fi[V_n]^{\Gamma_n}$.
Here we just note that the Poisson center of $\Fi[V_n]^{\Gamma_n}$
is finite and birational over $\Fi[V_n^{(1)}]^{\Gamma_n}$ and use that the latter algebra
is normal. So the isomorphism $\Gamma(\mathcal{D})\cong \Weyl(V_n)^{\Gamma_n}$
induces an automorphism of $\Fi[V_n^{(1)}]^{\Gamma_n}$. This isomorphism preserves the filtration
and is trivial on the level of associated graded algebras.
%
%

The second step is to show that the algebra $\Fi[V_n^{(1)}]^{\Gamma_n}$ has no nontrivial automorphisms
$\varphi$ with such properties. Let us define a derivation $\psi$ of $\Fi[V_n^{(1)}]^{\Gamma_n}$ that should be
thought as $\ln \varphi$. The degrees of generators of  $\Fi[V_n^{(1)}]^{\Gamma_n}$
are bounded from above for all $p\gg 0$ and so are degrees of relations between them.
Observe that it is only enough to define a derivation on generators and it will be well-defined
as long as it sends all relations to $0$. Now to construct $\psi$ we note that  $\varphi-1$
decreases degrees, and hence $\psi:=\ln\varphi$ makes sense as long as $p$ is sufficiently large.
The derivation $\psi$ lifts to $\Fi[V_n^{(1)}]$ because the quotient
morphism $V_n^{(1)}\rightarrow V_n^{(1)}/\Gamma_n$ is ramified in codimension bigger than $1$.
Since it decreases degrees, we see that $\psi$ has the form $\partial_v$ for some $v\in \Fi^{2n(1)}$. But, if $\Gamma_1\neq \{1\}$, the vector $v$ cannot be $\Gamma_n$-equivariant and so $\partial_v$ does not preserve $\Fi[V_n^{(1)}]^{\Gamma_n}$.
When $\Gamma_1=\{1\}$, there is a $\Gamma_n$-invariant vector. However, in this case we can modify our construction:
consider the reflection representation $\h$ of $\mathfrak{S}_n$ instead of the permutation representation $\C^n$.
We need to replace $R$ with $\mathfrak{sl}_n\oplus \C^n$. Theorem \ref{Thm:iso} gets modified accordingly.
\end{proof}

However, the easiest way to prove Theorem \ref{Thm:iso} is by using Procesi bundles
(at least for non-cyclic $\Gamma_1$ or general $c$, the case $c=0$ may be easier).
So we need some roundabout way to construct $\mathcal{D}$. In \cite{BK_Procesi} the question
of existence of $\mathcal{D}$ was reduced to $n=1$. More precisely, let $V^{sr}$ denote
the set of all $v\in V_n$ such that $\dim V^{\Gamma_v}>2$. Let us write
$X_1:=\rho^{-1}(V_n^{sr}/\Gamma_n)$. This is an open subset in $X$  with
$\operatorname{codim}_{X}X\setminus X_1>1$. First, Bezrukavnikov and Kaledin
produce a Frobenius constant quantization $\mathcal{D}_1$ of $X_1$ with
$\Gamma(\mathcal{D}_1)=\Weyl(V_n)^{\Gamma_n}$. This requires the existence
of such a quantization in the case when $n=1$. The latter case can be handled using
Theorem \ref{Thm:iso} proved in this case by Holland (that can be alternatively
proved using the existence of a Procesi bundle in the case $n=1$). When
$\mathcal{D}_1$ is constructed, Bezrukavnikov and Kaledin use the inequality
$\operatorname{codim}X\setminus X_1>1$ to show that $\mathcal{D}_1$ uniquely
extends to a Frobenius constant quantization $\mathcal{D}$ of $X$, automatically
with $\Gamma(\mathcal{D})=\Weyl(V_n)^{\Gamma_n}$.

\subsubsection{Construction of a Procesi bundle: characteristic $p$}\label{SSS_constr_Procesi_p}
Let $\mathcal{D}$ be as in the previous subsection. We will produce a Procesi bundle on
$X^{(1)}$ starting from $\mathcal{D}$. Since $X^{(1)}\cong X$ (an isomorphism of
$\Fi$-varieties), this will automatically establish a Procesi bundle on $X$. The isomorphism
$X^{(1)}\cong X$ follows from the observation that $X$ is defined over $\Fi_p$ and
$\operatorname{Fr}$ is an isomorphism of $\Fi$ fixing $\Fi_p$.

By Proposition \ref{Prop:CY_trick}, we have a derived equivalence $D^b(\operatorname{Coh}(X^{(1)},\mathcal{D}))
\xrightarrow{\sim} D^b(\Weyl(V_n)^{\Gamma_n}\operatorname{-mod})$. Also we have an abelian
equivalence $\Weyl(V_n)^{\Gamma_n}\operatorname{-mod}\xrightarrow{\sim}\Weyl(V_n)\#\Gamma_n\operatorname{-mod}=
\Weyl(V_n)\operatorname{-mod}^{\Gamma_n}$. Composing the two equivalences, we get
\begin{equation}\label{eq:equi1} D^b(\operatorname{Coh}(X,\mathcal{D}))\xrightarrow{\sim}
D^b(\Weyl(V_n)\operatorname{-mod}^{\Gamma_n}),\end{equation}
while what we need is a derived McKay equivalence
\begin{equation}\label{eq:equi2} D^b(\operatorname{Coh}X^{(1)})\xrightarrow{\sim} D^b(\Fi[V_n^{(1)}]\operatorname{-mod}^{\Gamma_n}).\end{equation}
Recall that $\mathcal{D}$ is an Azumaya algebra on $X$, while $\Weyl(V_n)$
is a $\Gamma_n$-equivariant Azumaya algebra on $V_n^{(1)}$. If we had a splitting
and a $\Gamma_n$-equivariant splitting, respectively, we would get (\ref{eq:equi2})
from (\ref{eq:equi1}). However, this is obviously not the case: $\Weyl(V_n)$
admits no splitting at all.

This can be fixed by passing to completions at $0$. Namely, let $X^{(1)\wedge_0}$ denote the formal
neighborhood of $(\rho^{(1)})^{-1}(0)$ in $X^{(1)}$. It was checked in \cite[Section 6.3]{BK_Procesi} that
the restriction of $\mathcal{D}$ to $X^{(1)\wedge_0}$ splits. Also it was checked that
the restriction of $\Weyl(V_n)$ to the formal neighborhood of $0$ in $\Fi^{2n(1)\wedge_0}$
admits a $\Gamma_n$-equivariant splitting. So, we get an equivalence
$$\iota:D^b(\operatorname{Coh}(X^{(1)\wedge_0}))\xrightarrow{\sim} D^b(\Fi[V_n^{(1)}]^{\wedge_0}\#\Gamma_n\operatorname{-mod})$$
that makes the following diagram commutative (all arrows are equivalences
of triangulated categories and all arrows but $R\Gamma$ come from abelian equivalences):

\begin{picture}(140,30)
\put(4,2){$D^b(\operatorname{Coh}(X^{(1)\wedge_0}))$}
\put(1,22){$D^b(\operatorname{Coh}(X^{(1)\wedge_0}, \mathcal{D}))$}
\put(56,22){$D^b(\Weyl(V_n)^{\wedge_0\Gamma_n}\operatorname{-mod})$}
\put(111,22){$D^b(\Weyl(V_n)^{\wedge_0}\#\Gamma_n\operatorname{-mod})$}
\put(110,2){$D^b(\Fi[V_n^{(1)}]^{\wedge_0}\#\Gamma_n\operatorname{-mod})$}
\put(12,8){\vector(0,1){12}}
\put(13,12){\tiny $\mathcal{B}^*\otimes\bullet$}
\put(37,23){\vector(1,0){18}}
\put(45,24){\tiny $R\Gamma$}
\put(95,23){\vector(1,0){14}}
\put(125,8){\vector(0,1){12}}
\put(37,3){\vector(1,0){70}}
\end{picture}

Here $\mathcal{B}$ denotes a splitting bundle for the restriction of $\mathcal{D}$ to $X^{(1)\wedge_0}$.

Set $\mathcal{P}':=\iota^{-1}(\Fi[V_n^{(1)}]^{\wedge_0}\#\Gamma_n)$. We claim that
$\mathcal{P}'$ is a vector bundle on $X^{(1)\wedge_0}$. Indeed, the image of $\Fi[V_n^{(1)}]^{\wedge_0}\#\Gamma_n$
in $\Weyl(V_n)^{\wedge_0\Gamma_n}\operatorname{-mod}$ is a projective generator and so
is a direct summand in the sum of several copies of $\Weyl(V_n)^{\wedge_0\Gamma_n}$.
But $R\Gamma^{-1}(\Weyl(V_n)^{\wedge_0\Gamma_n})=\mathcal{B}^*$. So $\mathcal{P}'$ is a direct summand in
a vector bundle (the sum of several copies of $\mathcal{B}^*$) and hence is a vector bundle itself.

So we get a vector bundle $\mathcal{P}'$ on $X^{(1)\wedge_0}$ that satisfies
$\operatorname{End}(\mathcal{P}')\cong\Fi[V_n^{(1)}]^{\wedge_0}\#\Gamma_n,
\operatorname{Ext}^i(\mathcal{P}',\mathcal{P}')=0$ for $i>0$.
The latter vanishing implies that $\mathcal{P}'$ is equivariant with respect
to the $\Fi^\times$-action on $X^{(1)\wedge_0}$, see \cite{Vologodsky}. From here it follows
that $\mathcal{P}'$ can be extended to $X^{(1)}$ (this is because  $\Fi^\times$
contracts $X^{(1)}$ to the zero fiber, see \cite[Section 2.3]{BK_Procesi}). Moreover, we can modify the equivariant
structure on $\mathcal{P}'$ and achieve that the isomorphism
$\operatorname{End}(\mathcal{P}')\cong\Fi[V_n^{(1)}]^{\wedge_0}\#\Gamma_n$
is $\Fi^\times$-equivariant, see \cite[Section 3.1]{Procesi}. It follows that $\mathcal{P}$ is a Procesi bundle.

\subsubsection{Construction of a Procesi bundle: lifting to characteristic $0$}\label{SSS_constr_Procesi_0}
Recall the $\mathfrak{R}$-scheme $X_{\mathfrak{R}}$ from \ref{SSS_quiv_var_char_p}.
We may assume $\mathfrak{R}$ is regular. Taking an algebraic extension of  $\mathfrak{R}$,
we get a maximal ideal $\mathfrak{m}$ such that there is a Procesi bundle $\mathcal{P}_{\Fi}$ on $X_{\Fi}$,
where $\Fi$ is an algebraic closure of $\Fi_0:=\mathfrak{R}/\mathfrak{m}$. We may assume that $\mathcal{P}_{\Fi}$
is defined over $\Fi_0$, let $\mathcal{P}_{\Fi_0}$ be the corresponding form.
Let $\mathfrak{R}^\wedge$ be the $\mathfrak{m}$-adic completion of $\mathfrak{R}$.
Since $\Ext^i(\mathcal{P}_{\Fi_0}, \mathcal{P}_{\Fi_0})=0$ for $i=1,2$, we see that $\mathcal{P}_{\Fi_0}$ uniquely
deforms to a $\mathbb{G}_m$-equivariant vector bundle on the formal neighborhood of
$X_{\Fi_0}$ in $X_{\mathfrak{R}^\wedge}$ (see \cite[Section 2.3]{BK_Procesi}).

Let us show that the $\mathbb{G}_m$-finite part of $\operatorname{End}(\mathcal{P}_{\mathfrak{R}^\wedge})$
is $\mathfrak{R}^\wedge[V_n]\#\Gamma_n$. Consider the formal neighborhood $Z$
of $X_{\Fi_0}^{reg}$ in $X_{\mathfrak{R}^\wedge}^{reg}$. Note that
$\operatorname{Ext}^1(\mathcal{P}_{\Fi_0}|_{X_0^{reg}},\mathcal{P}_{\Fi_0}|_{X_0^{reg}})=0$, see, for example, \cite[Appendix]{BL}. So the  restriction of
$\mathcal{P}_{\mathfrak{R}^\wedge}$ to $Z$ coincides with $\eta_* \mathcal{O}_{(\mathfrak{R}^{\wedge 2n})^{reg}}$,
where $\eta$ denotes the quotient morphism $\mathfrak{R}^{\wedge 2n}\rightarrow \mathfrak{R}^{\wedge 2n}/\Gamma_n$.
This implies the claim about endomorphisms.

Since $\mathcal{P}_{\mathfrak{R}^\wedge}$ is $\mathbb{G}_m$-equivariant and the $\mathbb{G}_m$-action
is contracting, it extends from a formal neighborhood of $X_{\Fi_0}$ in $X_{\mathfrak{R}}$ to
$X_{\mathfrak{R}^\wedge}$. So we get a Procesi bundle
on $X_K$, where $K=\operatorname{Frac}(\mathfrak{R}^\wedge)$. But being a finite extension of
the $p$-adic field, $K$ embeds into $\C$ and so we get a Procesi bundle on $X$.


\subsection{Symplectic reflection algebras}\label{SS_SRA}
\subsubsection{Flatness and universality}
Let $V$ be a symplectic vector space with form $\Omega$ and $\Gamma\subset \operatorname{Sp}(V)$ be a finite group
of symplectomorphisms. We write $S$
for the set of symplectic reflections in $\Gamma$, it is a union of
conjugacy classes: $S=S_0\sqcup S_1\sqcup\ldots\sqcup S_r$. We pick independent variables
$t,c_0,\ldots,c_r$.

Recall the universal Symplectic reflection algebra ${\bf H}$, the quotient
of $T(V)\#\Gamma[t,c_0,\ldots,c_r]$ by the relations (\ref{eq:SRA_reln}). Let us write
$\param_{univ}$ for the vector space with basis $t,c_0,\ldots,c_r$ so that
${\bf H}$ is a graded $S(\param_{univ})$-algebra.

\begin{Thm}\label{Thm:flatness}
The algebra ${\bf H}$ is a free graded $S(\param_{univ})$-module. Moreover, assume
that $\Gamma$ is symplectically irreducible.  Then ${\bf H}$ is universal
with this property in the following sense. Let $\param'$ be a vector space and ${\bf H}'$
be a graded $S(\param')$-algebra (with $\deg\param'=2$) that is a free graded $S(\param')$-module
and ${\bf H}'/(\param')=S(V)\#\Gamma$. Then there is a unique linear map $\nu:\param_{univ}\rightarrow
\param'$ and unique isomorphism $S(\param')\otimes_{S(\param_{univ})}{\bf H}\xrightarrow{\sim} {\bf H}'$
of graded $S(\param')$-algebras that induces the identity isomorphism of $S(V)\#\Gamma_n$.
\end{Thm}

When $\Gamma_1\neq \{1\}$, then the action of the group $\Gamma_n$ on $V_n=\C^{2n}$ is symplectically
irreducible. When $\Gamma_1=\{1\}$, the module $\C^{2n}$ over $\Gamma_n$ is not symplectically irreducible,
so we replace $\C^{2n}$ with $V_n=\h\oplus \h^*$, where $\h$ is the reflection representation of $\mathfrak{S}_n$.
Note that we did the same in \ref{SSS_quant_X_char_p}.

\subsubsection{Hochschild cohomology}
Before we prove this theorem we will need to get some information about Hochshild cohomology
of $S(V)\#\Gamma$. We need this because the Hochschild cohomology controls deformations of $S(V)\#\Gamma$.

Let $A$ be a graded algebra. We want to describe graded deformations of $A$.
The Hochschild cohomology group $\operatorname{HH}^i(A)$ inherits the grading
from $A$, let $\operatorname{HH}^i(A)^j$ denote the $j$th graded component. The
general deformation theory implies the following.

\begin{Lem}\label{Lem:free_deform}
Assume that $\dim \operatorname{HH}^2(A)^{-2}<\infty$ and $\operatorname{HH}^i(A)^j=0$ for $i+j<0$. Set $P_{univ}:=(HH^2(A)^{-2})^*$. Then there is a free graded $S(P_{univ})$-algebra $\A_{univ}$ (with $\deg P_{univ}=2$)
such that $\A_{univ}/(P_{univ})=A$ that is a universal graded deformation of $A$
in the same sense as in Theorem \ref{Thm:flatness}.
\end{Lem}

What we are going to do is to compute the relevant graded components of $\operatorname{HH}^\bullet(SV\#\Gamma)$.
The vanishing result is easy and the computation of $P_{univ}$ is more subtle.

First, we use the fact that $\operatorname{HH}^i(A,M)=\operatorname{Ext}^i_{A\otimes A^{opp}}(A,M)$
(where $M$ is an $A$-bimodule) to see that
\begin{equation}\label{eq:HH1}
\operatorname{HH}^i(S(V)\#\Gamma, S(V)\#\Gamma)=\operatorname{HH}^i(S(V),S(V)\#\Gamma)^{\Gamma}.
\end{equation}
We have a $\Gamma$-action on $\operatorname{HH}^i(S(V),S(V)\#\Gamma)$ because both $S(V)$-bimodules
$S(V), S(V)\#\Gamma$ are $\Gamma$-equivariant. We have $S(V)\#\Gamma=\bigoplus_{\gamma\in \Gamma}S(V)\gamma$
of $S(V)$-bimodules, where $S(V)\gamma$ is identified with $S(V)$ as a left $S(V)$-module and the right action
is given by $f\cdot a:=f\gamma(a)$.

Let us compute $\operatorname{HH}^i(S(V),S(V)\gamma)$ in degrees we are interested in:  $j<-i$
and also $j=-2$ for $i=2$. We have $\gamma=\operatorname{diag}(\gamma_1,\ldots,\gamma_n)$,
where we view $\gamma_i$ as elements of  cyclic groups acting on $\C$. Then we have an isomorphism
of bigraded spaces
\begin{equation}\label{eq:HH2}
\bigoplus_i \operatorname{HH}^i(S(V),S(V)\gamma)\cong\bigotimes_{\ell=1}^n \bigoplus_i \operatorname{HH}^i(\C[x], \C[x]\gamma_\ell).
\end{equation}
For an arbitrary $\gamma_\ell$, we have $\operatorname{HH}^i(\C[x], \C[x]\gamma_\ell)=0$ when $i>1$.
When $\gamma_\ell=1$, we have $\operatorname{HH}^0(\C[x],\C[x])=\C[x]$ and $\operatorname{HH}^1(\C[x],\C[x])=\C[x]\{1\}$,
where $\{1\}$ indicates the grading shift by $1$ so that $\operatorname{HH}^1(\C[x],\C[x])$ is a free module
generated in degree $-1$. When $\gamma_\ell\neq 1$, then $\operatorname{HH}^0(\C[x],\C[x]\gamma_\ell)=0$
and $\operatorname{HH}^1(\C[x],\C[x])=\C\{1\}$.

This computation easily implies that $\operatorname{HH}^i(S(V),S(V)\#\Gamma)^j=0$ when $i+j<0$.
Now let explain how to compute $\left(\operatorname{HH}^2(S(V),S(V)\#\Gamma)^{-2}\right)^\Gamma$.
If $\operatorname{HH}^2(S(V),S(V)\gamma)^{-2}\neq 0$, then either $\gamma=1$ or $\gamma$
is a symplectic reflection. When $\gamma=1$, then $\operatorname{HH}^2(S(V),S(V)\gamma)^{-2}=\bigwedge^2 V$.
When $\gamma$ is a symplectic reflection, then $\operatorname{HH}^2(S(V),S(V)\gamma)^{-2}=\C$.
An element $\gamma_1\in \Gamma$ maps $S(V)\gamma$ to $S(V)(\gamma_1\gamma\gamma_1^{-1})$.
The action of $\Gamma$ on $\operatorname{HH}^2(S(V),S(V)\gamma)^{-2}=\bigwedge^2 V$
is a natural one. When $\gamma$ is a symplectic reflection, then the action of $Z_\Gamma(\gamma)$
on $\operatorname{HH}^2(S(V),S(V)\gamma)^{-2}=\C$ is trivial. From here we deduce
that $$\dim \operatorname{HH}^2(S(V)\#\Gamma, S(V)\#\Gamma)^{-2}=r+2,$$ as claimed.

\subsubsection{Proof of Theorem \ref{Thm:flatness}}
Let us write ${\bf H}_{univ}$ for the universal deformation, we need to prove that
${\bf H}_{univ}\xrightarrow{\sim}{\bf H}$.

First of all, note that degree $0$ and $1$ components of ${\bf H}_{univ}$ are the same as in $S(V)\#\Gamma$.
So we have natural embeddings $\Gamma,V\hookrightarrow {\bf H}_{univ}$. It is easy to see that
$\param_{univ},V,\Gamma$ generate ${\bf H}_{univ}$. This gives rise to an epimorphism
$S(\param_{univ})\otimes T(V)\#\Gamma\twoheadrightarrow {\bf H}_{univ}$. Further, for
$u,v\in V\subset {\bf H}_{univ}$, we have $[u,v]\in (\param_{univ})$. The degree $2$
of $(\param_{univ})$ is $\param_{univ}\otimes \C\Gamma$. So we get $[u,v]=\kappa(u,v)$
in ${\bf H}_{univ}$, where $\kappa$ is a map $\bigwedge^2 V\rightarrow \param_{univ}\otimes \C\Gamma$.
A computation done in \cite[Section 2]{EG} shows that, since ${\bf H}_{univ}$ is free over $S(\param_{univ})$,
we get $$\kappa=t\Omega+\sum_{i=0}^r c_i \sum_{s\in S_i}\Omega_s(u,v)s.$$
This completes the proof of Theorem \ref{Thm:flatness}.

\subsection{Proof of the isomorphism theorem}\label{SS_iso_proof}
We will prove an isomorphism of $e{\bf H}e$ and the universal Hamiltonian reduction
${\bf A}:=\Weyl_\hbar(T^*R)\red G$, where $\Weyl_\hbar(T^*R)$ is the Rees algebra of $D(R)$
(with modified grading so that $\deg T^*R=1, \deg\hbar=2$). Here we take $R:=R(Q,n\delta,\epsilon_0)$
for $n>1$ and $R:=R(Q,\delta,0)$ for $n=1$. In the case when $n>1$, we take $G:=\operatorname{GL}(n\delta)$.
For $n=1$, for $G$, we take the quotient of $\operatorname{GL}(\delta)$ by the one-dimensional
central subgroup of constant elements.

Both $e{\bf H}e,{\bf A}$ are graded algebras.
The algebra $e{\bf H}e$ is over $S(\param_{univ})$ with $\deg\param_{univ}=2$.
The algebra ${\bf A}$ is over $S(\param_{red})$, where $\param_{red}:=\g/[\g,\g]\oplus \C\hbar$.
We will prove that there is a graded algebra isomorphism $e{\bf H}e\xrightarrow{\sim}{\bf A}$
that maps $\param_{univ}$ to $\param_{red}$ and induces the identity automorphism
$e{\bf H}e/(\param_{univ})= \C[V_n]^{\Gamma_n}={\bf A}/(\param_{red})$. Further, we will
explain why the corresponding isomorphism $\param_{univ}\cong \param_{red}$ maps
$\hbar$ to $t$ and gives (\ref{eq:param_corresp}) on the hyperplanes $t=1$ and $\hbar=1$.
In other words, the isomorphism $\nu:\param_{univ}\rightarrow \param_{red}$
is the inverse of the following map
\begin{equation}\label{eq:nu}
\begin{split}
&\hbar\mapsto t, \quad \epsilon_i\mapsto \frac{1}{|\Gamma_1|}\operatorname{tr}_{N_i}\tilde{\bf c}, i\neq 0, \quad \epsilon_0\mapsto
\frac{1}{|\Gamma_1|}\operatorname{tr}_{N_0}\tilde{\bf c}-\frac{1}{2}(c_0+t), \qquad (n>1)\\
&\hbar\mapsto t, \quad \epsilon_i\mapsto \frac{1}{|\Gamma_1|}\operatorname{tr}_{N_i}\tilde{\bf c}, i\neq 0, \quad \epsilon_0\mapsto
\frac{1}{|\Gamma_1|}\operatorname{tr}_{N_0}\tilde{\bf c}-t,\qquad (n=1)
\end{split}
\end{equation}
Here the notation is as follows. We write $\tilde{\bf c}:=t+\sum_{i=1}^r c_i \sum_{\gamma\in S_i^0}\gamma$
and $\epsilon_i\in \g/[\g,\g]$ is specified by $\operatorname{tr}_i \epsilon_j:=\delta_{ij}$.

\subsubsection{Application of a Procesi bundle}
An isomorphism $e{\bf H}e\cong {\bf A}$ is produced as follows. The algebra $e{\bf H}e$ does not
have good universality properties (although it is expected to be semi-universal), it is ${\bf H}$
that does. We will produce a graded $S(\param_{red})$-algebra $\tilde{\bf A}$ deforming
$\C[V_n]\#\Gamma_n$ with $e\tilde{\bf A}e={\bf A}$. This will give rise to
a linear map $\nu:\param_{univ}\rightarrow \param_{red}$ and to an isomorphism
$S(\param_{red})\otimes_{S(\param_{univ})}{\bf H}\cong \tilde{\bf A}$ and hence
also to an isomorphism
$S(\param_{red})\otimes_{S(\param_{univ})}e{\bf H}e\cong {\bf A}$. The algebra $\tilde{\bf A}$
will be constructed from a Procesi bundle $\mathcal{P}$ on $X=\M^\theta(n\delta,\epsilon_0)$.

First, let us produce a sheaf version of ${\bf A}$. Consider the variety $\widetilde{X}:=T^*R\red^\theta G$,
this is a deformation of $X$ over $\g^{*G}$. Then we consider its deformation quantization
obtained by Hamiltonian reduction, the sheaf $\widetilde{\mathcal{D}}_\hbar:=\Weyl_\hbar(T^*R)^{\wedge_\hbar}\red^\theta G$. The algebra ${\bf A}$ coincides with the $\C^\times$-finite part of $\Gamma(\widetilde{\mathcal{D}}_\hbar)$.
Now let us take a Procesi bundle $\mathcal{P}$ on $X$. Since $\operatorname{Ext}^i(\mathcal{P},\mathcal{P})=0$,
the bundle $\mathcal{P}$ deforms to a unique $\C^\times$-equivariant vector bundle on the formal neighborhood of $X$
in $\widetilde{X}$. But the $\C^\times$-action contracts $\widetilde{X}$ to $X$.
So $\mathcal{P}$ extends to a unique $\C^\times$-equivariant bundle $\widetilde{\mathcal{P}}$
on $\widetilde{X}$. The extension $\widetilde{\mathcal{P}}$ again satisfies the Ext-vanishing conditions and so
further extends to a unique $\C^\times$-equivariant right $\widetilde{\mathcal{D}}_\hbar$-module
$\widetilde{\mathcal{P}}_\hbar$.

Consider the endomorphism algebra $\operatorname{End}_{\widetilde{\mathcal{D}}_\hbar^{opp}}(\widetilde{\mathcal{P}}_\hbar)$. Modulo
$(\param_{red})$, this algebra coincides with $\operatorname{End}_{\Str_X}(\mathcal{P})=\C[V_n]\#\Gamma_n$.
Let $\tilde{\bf A}$ be the $\C^\times$-finite part of $\operatorname{End}_{\widetilde{\mathcal{D}}_\hbar^{opp}}(\widetilde{\mathcal{P}}_\hbar)$.
It is the endomorphism algebra of the right $\widetilde{\mathcal{D}}_{\hbar,fin}$-module
$\widetilde{\mathcal{P}}_{\hbar,fin}$.
The algebra  $\tilde{\bf A}$ is a graded $S(\param_{red})$-algebra with $\tilde{\bf A}/(\param_{red})=\C[V_n]\#\Gamma_n$,
where $\param_{red}$ lives in degree $2$. We conclude that
there is a unique map $\nu_{\mathcal{P}}:\param_{univ}\rightarrow \param_{red}$
with $\tilde{\bf A}\cong S(\param_{red})\otimes_{S(\param_{univ})}{\bf H}$. Then, automatically,
we have
\begin{equation}\label{eq:iso_spher} {\bf A}(=e\tilde{\bf A}e)\cong S(\param_{red})\otimes_{S(\param_{univ})}e{\bf H}e.
\end{equation}
We will study the linear maps $\nu:\param_{univ}\rightarrow \param_{red}$ such that
(\ref{eq:iso_spher}) holds. We will see that
\begin{itemize}
\item[(a)]
 any such $\nu$ is an isomorphism,
\item[(b)] that there are $|W|$
options for $\nu$ when $n=1$ and $2|W|$ options else,
\item[(c)] and that one can choose $\nu$ as in
(\ref{eq:nu}).
\end{itemize}
(c) will complete the proof of Theorem \ref{Thm:iso}, while (b) will be used to classify the Procesi
bundles.

First of all,  let us point out that $\nu(t)=\hbar$. Indeed, the Poisson bracket on $\C[\M_0^0(n\delta,\epsilon_0)]$
induced by the deformation ${\bf A}$ equals $\hbar\{\cdot,\cdot\}$, where $\{\cdot,\cdot\}$ is the standard
bracket given by the Hamiltonian reduction (more precisely, if we specialize to $(\hbar',\lambda)\in \C\oplus \g^{*G}$,
then the bracket induced by the corresponding filtered deformation is $\hbar'\{\cdot,\cdot\}$). Similarly,
the bracket on $\C[V_n]^{\Gamma_n}$ induced by $e{\bf H}e$ coincides with $t\{\cdot,\cdot\}$, see Example
\ref{Ex:SRA}. Since the isomorphism $\M_0^0(n\delta,\epsilon_0)\cong V_n/\Gamma_n$ is Poisson, the
equality $\nu(t)=\hbar$ follows.

\subsubsection{Case $n=1$}\label{SSS_iso_thm_n1}
We start by proving (a)-(c) for $n=1$.

Let us prove (c). First of all, recall that $X$ can be constructed as the moduli space
of the $\C[x,y]\#\Gamma_1$-modules isomorphic to $\C\Gamma_1$ as $\Gamma_1$-modules that admit a cyclic vector. The universal bundle on $X$ is a Procesi bundle. Moreover, from \cite[Section 8]{CBH}, it follows that $\widetilde{X}$
is the moduli space of the ${\bf H}/(t)$-modules isomorphic to $\C\Gamma_1$ and admitting a cyclic vector.
The corresponding isomorphism $\param_{red}/\C\hbar\cong \param_{univ}/\C t$ is induced from $\nu$.

To show that $\nu$ then is given by (\ref{eq:nu}) we consider the loci of parameters $\lambda$ and $c$
where the homological dimensions of $A_{1,\lambda}:=\Weyl(T^*R)\red_\lambda G, e H_{1,c}e$ are infinite.
Both are given by the union of hyperplanes of the form $\lambda\cdot \beta=0$, where $\beta$ runs over
the set of the  roots of $Q\setminus \{0\}$ (when we speak of the parameter $\lambda$ for the algebra
$eH_{1,c}e$ we mean the parameter computed in Theorem \ref{Thm:iso}). The claim for $eH_{1,c}e$ follows
from \cite[Theorem 0.4]{CBH}, and that on $A_{1,\lambda}$ then follows from \cite[Section 5]{quant} (from an isomorphism
of $A_{1,\lambda}$ with a central reduction of a suitable W-algebra) or from \cite{Boyarchenko}.

The same considerations as in the previous paragraph imply (a). To prove (b) one now needs to describe
the group $\mathfrak{A}$ of the automorphisms of ${\bf A}(\cong e{\bf H}e)$ satisfying the following:
\begin{itemize}
\item they preserve the grading,
\item they preserve $\param_{red}$ as a subset of ${\bf A}$,
\item they are the identity modulo $\param_{red}$.
\end{itemize}
We have a natural homomorphism $\mathfrak{A}\rightarrow \operatorname{GL}(\param_{red})$ that is easily seen
to be injective. From the isomorphism with a W-algebra mentioned above, one sees that $W\subset \mathfrak{A}$
(recall that the $W$-action on $\g^{*G}$ was described in \ref{SSS_resol_iso}).
With some more work, see \cite[Proposition 6.4.5]{quant}, one shows that actually $W=\mathfrak{A}$. This implies (b).

\subsubsection{Completions}\label{SSS_completions}
The case of a general $n$ is reduced to $n=1$ using suitable completions of the algebras ${\bf A}, {\bf H}$.
Let us explain what completions we use as well as general results on their structure.

First, let us describe completions of algebras of the form ${\bf A}:=\Weyl_\hbar(V)\red G$, where $V$
is a symplectic vector space and $G$ is a reductive group acting on $V$ by symplectomorphisms.
Let $b\in V\red_0 G$. The point $b$ defines a maximal ideal  $\mathfrak{m}\subset{\bf A}$. So we can form the $b$-adic completion ${\bf A}^{\wedge_b}:=\varprojlim_{n\rightarrow +\infty} {\bf A}/\mathfrak{m}^n$. Let $v\in V$ be a point
with closed $G$-orbit mapping to $b$. Let us write $\Weyl_\hbar(V)^{\wedge_{Gv}}$
for the completion of $\Weyl_\hbar(V)$ with respect to the ideal of $Gv$. Then
it is easy to see that ${\bf A}^{\wedge_b}\cong \Weyl_\hbar(V)^{\wedge_{Gv}}\red G$.
The algebra $\Weyl_\hbar(V)^{\wedge_{Gv}}$ can be described using a suitable version of
the slice theorem. More precisely, it follows, for example, from
\cite[Section 4]{CB_norm}
that the formal neighborhood $V^{\wedge_{Gv}}$ is equivariantly symplectomorphic
to the neighborhood of the base $G/K$ in $(T^*G\times U)\red_0 K$, where $K:=G_v, U:=(T_v Gv)^\perp/T_v Gv$.
This statement quantizes: $\Weyl_\hbar(V)^{\wedge_{Gv}}\cong (D_\hbar(G)\otimes_{\C[\hbar]}\Weyl_\hbar(U))\red_0 K$,
this can be proved similarly to  \cite[Theorem 2.3.1]{HC}.
From here one deduces that
$${\bf A}^{\wedge_b}\cong \C[[\g^{*G}]]\widehat{\otimes}_{\C[[\kf^{*K}]]}\left(\Weyl_\hbar(U)^{\wedge_0}\red K\right),$$
where a homomorphism $\C[[\kf^{*K}]]\rightarrow \C[[\g^{*G}]]$ is induced from the restriction map
$\g^{*G}\rightarrow \kf^{*K}$.

On the other hand, take a symplectic vector space $V'$ and a finite subgroup $\Gamma\subset \operatorname{Sp}(V)$.
From these data we can form the symplectic reflection
algebra ${\bf H}$. Pick $b\in V'/\Gamma$.
We can produce the completion ${\bf H}^{\wedge_b}$: the point $b$ defines a natural
maximal ideal in $\C[V']\#\Gamma$, we take its preimage in ${\bf H}$ and complete with respect to that
preimage. The algebra ${\bf H}^{\wedge_b}$ can also be described in terms of a ``smaller''
algebra of the same type, \cite[Theorem 1.2.1]{sraco}. More precisely, let $\underline{\Gamma}$ be the stabilizer corresponding to
$b$ and let $\underline{\bf H}$ stand for the SRA corresponding to the pair $(\underline{\Gamma},V')$,
an algebra over $S(\param_{univ})$.
Then ${\bf H}^{\wedge_b}\cong Z(\Gamma,\underline{\Gamma}, \underline{\bf H}^{\wedge_0})$, where $Z(\Gamma,\underline{\Gamma},\bullet)$ is the centralizer algebra from \cite[3.2]{BE}, it is isomorphic
to $\operatorname{Mat}_{|\Gamma/\underline{\Gamma}|}(\bullet)$. A consequence we need is that
$e{\bf H}^{\wedge_b}e\cong \underline{e}\underline{\bf H}^{\wedge_0}\underline{e}$.
The algebra $\underline{\bf H}$ can be described as follows. Let us write
$V^+$ for a unique $\underline{\Gamma}$-stable complement to $V'^{\underline{\Gamma}}$
in $V'$. Consider the SRA  $\underline{\bf H}^+$ over $S(\underline{\param}_{univ})$,
where $\underline{\param}_{univ}$ is the parameter space for $\underline{\Gamma}$.
The inclusion $\underline{\Gamma}\hookrightarrow \Gamma$ gives rise to a natural
map $\underline{\param}_{univ}\rightarrow \param_{univ}$. Then
$\underline{\bf H}=\Weyl_t(V'^\Gamma)\otimes_{\C[t]}(S(\param_{univ})\otimes_{S(\underline{\param}_{univ})}
\underline{\bf H})$.

\subsubsection{Completions at leaves of codimension 2}
We are going to use the completions of ${\bf A}$ and $e{\bf H}e$
at points lying in the codimension 2 symplectic leaves.
Recall from \ref{SSS_resol_iso} that when $n>1$ and $\Gamma_1\neq \{1\}$, we have two such leaves.
One corresponds to $\underline{\Gamma}=\Gamma_1\subset \Gamma_n$, the other to
$\mathfrak{S}_2\subset \Gamma_n$. Let $\underline{\bf H}^{1+},\underline{\bf H}^{2+}$
be the corresponding SRA's. The corresponding parameter spaces are $\underline{\param}_{univ}^1=\operatorname{Span}(c_1,\ldots,c_r,t)$ and
$\underline{\param}_{univ}^2=\operatorname{Span}(c_0,t)$. When $\Gamma_1=\{1\}$, we have just one leaf
of codimension $2$, it corresponds  to $\mathfrak{S}_2$.

Now let us describe the completions on the Hamiltonian reduction side. Let $v^1,v^2$
be elements from closed $G$-orbits in $\mu^{-1}(0)\in T^*R$ whose images $b^1,b^2$ in $\M^0_0(n\delta,\epsilon_0),
V_n/\Gamma_n$ lie in the two leaves. We can take the points $v^1,v^2$ as follows.
We have a natural embedding $\mu_1^{-1}(0)^n\hookrightarrow \mu^{-1}(0)$ from the proof of
Proposition \ref{Prop:quot_sing_quiv}.
Take pairwise different elements $v_1,\ldots,v_{n}\in \mu^{-1}(0)$ with closed
$\operatorname{GL}(\delta)$-orbits.  Then we can take $v^1=(v_1,\ldots,v_{n-1},0)\in T^*R(n\delta,0)\subset T^*R$
and $v^2=(v_1,\ldots, v_{n-2},v_{n-1},v_{n-1})$.

Let us describe the completion  ${\bf A}^{\wedge_{b^1}}$.
We have $K_1(=G_{v^1})=(\C^\times)^{n-1}\times \GL(\delta)$. So the space $\kf_{1}^{*K_1}$ coincides $ \C^{n-1}\oplus \C^{Q_0}$. The restriction map $\C^{Q_0}=\g^{*G}\rightarrow \kf_1^{*K_1}=\C^{n-1}\oplus \C^{Q_0}$
sends $\lambda$ to $(\lambda\cdot\delta,\ldots,\lambda\cdot\delta, \lambda)$. The symplectic part
$U$ of the normal space $T^*R/T_{v^1}Gv^1$ splits into the direct sum of the trivial module $\C^{2(n-1)}$,
of the $(\C^\times)^{n-1}$-module $(T^*\C)^{\oplus n-1}$, and of the $\GL(\delta)$-module
$T^*R(\delta,\epsilon_0)$. So
$$\Weyl_\hbar(U)\red K_1\cong \C[z_1,\ldots,z_{n-1}]\otimes \Weyl_\hbar(\C^{2(n-1)})
\otimes_{\C[\hbar]}\Weyl_\hbar(T^*R(\delta,\epsilon_0))\red \operatorname{GL}(\delta),$$
where $z_1,\ldots,z_{n-1}$ are homogeneous elements of degree $2$, the images of the natural
basis in $\operatorname{Lie}(\C^{\times (n-1)})$ under the comoment map.

Let us write $\overline{\GL}(\delta)$ for the quotient of $\GL(\delta)$ by the one-dimensional
torus of constant elements. Set $\g^{*G}_0:=\g^{*G}/\C\delta$, clearly, $\g^{*G}_0=\overline{\mathfrak{gl}}(\delta)^{*\overline{\GL}(\delta)}$.
Set ${\bf A}^1:=\Weyl_\hbar(T^*R(\delta,0))\red \overline{\GL}(\delta)$.
It is easy to see that $\Weyl_\hbar(T^*R(\delta,\epsilon_0))\red \GL(\delta)
=\C[\g^{*G}]\otimes_{\C[\g^{*G}_0]}{\bf A}^1$. From here and the description of
the map $\kf_1^{*K_1}\rightarrow \g^{*G}$ given above, we deduce
that
$$\C[\g^{*G}]\otimes_{\C[\kf_1^{*K_1}]}\Weyl_\hbar(U)\red K_1\cong \Weyl_\hbar(\C^{2n-2})\otimes_{\C[\hbar]}(\C[\g^{*G}]\otimes_{\C[\g^{*G}_0]}{\bf A}^1).$$
It follows that
\begin{equation}\label{eq:red_compl1}
{\bf A}^{\wedge_{b^1}}\cong \Weyl_\hbar^{\wedge_0}(\C^{2n-2})\widehat{\otimes}_{\C[[\hbar]]}
(\C[[\param_{red}^*]]\widehat{\otimes}_{\C[[\param^{1*}_{red}]]}{\bf A}^{1\wedge_0}),
\end{equation}
where we write $\param^1_{red}$ for $\{\lambda\in \C^{Q_0}| \lambda\cdot\delta=0\}\oplus \C\hbar$.

Let us now deal with ${\bf A}^{\wedge_{b^2}}$. We have $K_2(=G_{v^2})=(\C^{\times})^{n-2}\times \GL(2)$.
The map $\g^{*G}\rightarrow \kf_2^{*K_2}$ sends $\lambda$ to the $n-1$-tuple with equal
coordinates $\lambda\cdot \delta$. The symplectic part $U^2$ of the normal space
$T^*R/T_{v^2}Gv^2$ is the sum of the trivial module $\C^{2(n-1)}$, the $(\C^{\times})^{n-2}$-module
$(T^*\C)^{\oplus 2}$ and the $\GL(2)$-module
$T^*(\mathfrak{sl}_2\oplus \C^2)$. Let $\param^2_{red}$ denote the span of $\sum_{i\in Q_0}\delta_i\epsilon_i$
and $\hbar$. Set ${\bf A}^2:=\Weyl_\hbar(T^*(\mathfrak{sl}_2\oplus \C^2))\red \GL(2)$, we can view
it as an algebra over $S(\param_{red}^2)$ (where a natural generator of $\mathfrak{gl}_2/[\mathfrak{gl}_2,\mathfrak{gl}_2]$ corresponds to $\sum_{i\in Q_0}\delta_i\epsilon_i$). As above, we have
$$\C[\g^{*G}]\otimes_{\C[\kf_2^{*K_2}]}\Weyl_\hbar(U^2)\red K_2\cong S(\param_{red})\otimes_{S(\param_{red}^2)}(\Weyl_\hbar(\C^{2n-2})\otimes_{\C[\hbar]}{\bf A}^2)$$
and we get the following description of ${\bf A}^{\wedge_{b^2}}$:
\begin{equation}\label{eq:red_compl2}
{\bf A}^{\wedge_{b^2}}\cong \Weyl_\hbar^{\wedge_0}(\C^{2n-2})\widehat{\otimes}_{\C[[\hbar]]}
(\C[[\param_{red}^*]]\widehat{\otimes}_{\C[[\param_{red}^{2*}]]}{\bf A}^{2\wedge_0}).
\end{equation}

\subsubsection{Reduction to $n=1$}
Using (\ref{eq:red_compl1}) we see that (\ref{eq:iso_spher}) yields an isomorphism of completions
${\bf A}^{\wedge_{b^1}}\cong e^1{\bf H}^{\wedge_{b^1}}e^1$ and hence an isomorphism
\begin{align*}&\Weyl_\hbar^{\wedge_0}(\C^{2(n-1)})\widehat{\otimes}_{\C[[\hbar]]}
(\C[[\param_{red}^*]]\widehat{\otimes}_{\C[[\param^{1*}_{red}]]}{\bf A}^{1\wedge_0})\cong\\
&\Weyl_\hbar^{\wedge_0}(\C^{2(n-1)})\widehat{\otimes}_{\C[[\hbar]]}
(\C[[\param_{red}^*]]\otimes_{\C[[\param_{univ}^*]]}\C[[\param^*_{univ}]]\otimes_{\C[[\param^{1*}_{univ}]]}
e^1{\bf H}^{1\wedge_0}e^1).\end{align*}
It was checked in \cite[Section 6.5]{quant} that this isomorphism restricts to
$$S(\param_{red})\otimes_{S(\param_{red}^1)}{\bf A}^1\cong S(\param_{red})
\otimes_{S(\param_{univ}^1)}e^1{\bf H}^1 e^1$$
that preserves the grading and is the identity modulo $(\param_{red})$.
From here it is easy to deduce that $\nu$ maps $\param^1_{univ}$ to $\param^1_{red}$
and restricts to one of $W$-conjugates of the map in (\ref{eq:nu})
for $n=1$.

Let us proceed to the second leaf. Similarly to \ref{SSS_iso_thm_n1}, one can show
that $\Weyl_\hbar(T^*(\mathfrak{sl}_2\oplus \C^2))\red \operatorname{GL}(2)\cong e^2 {\bf H}^2 e^2$,
where the isomorphism sends the element $\sum_{i\in Q_0}\delta_i\epsilon_i$
to $\pm (c_0+t)/2$.  It follows that $\nu$
maps $\param^2_{univ}$ to $\param^2_{red}$ and induces one of two maps
in the previous sentence. It follows that $\nu$ is an isomorphism that is $W\times \Z/2\Z$-conjugate to
the map given by (\ref{eq:nu}) for $n>1$. Since $W\times \Z/2\Z$-action comes from automorphisms,
that preserve the grading, map $\param_{red}$ to $\param_{red}$, and
are the identity modulo $(\param_{red})$ (see  \ref{SSS_quant_auto}), claims (b) and
(c) follow. This completes the proof of Theorem \ref{Thm:iso}.

\subsection{Classification of Procesi bundles}\label{SS_Procesi_classif}
Here we are going to prove that the number of different Procesi bundles on $X$
equals $2|W|$ for $n>1$ and $|W|$ for $n=1$.

\subsubsection{Upper bound}
Recall that
a Procesi bundle $\mathcal{P}$ on $X$ defines a linear isomorphism
$\nu_{\mathcal{P}}: \param_{univ}\rightarrow \param_{red}$. We claim that if $\nu_{\mathcal{P}^1}=\nu_{\mathcal{P}^2}$,
then $\mathcal{P}^1\cong \mathcal{P}^2$. Indeed, we have
\begin{equation}\label{eq:eq_glob_sec}\Gamma(\tilde{\mathcal{P}}^1_{\hbar,fin})=
\End(\tilde{\mathcal{P}}^1_{\hbar,fin})e\cong \End(\tilde{\mathcal{P}}^2_{\hbar,fin})e=\Gamma(\tilde{\mathcal{P}^2}_{\hbar,fin})\end{equation}
(an isomorphism of graded right ${\bf H}$-modules). Note that
$H^1(\tilde{X}, \tilde{\mathcal{P}}^i)=0$ because $\tilde{\mathcal{P}}^i$ is a direct
summand of $\mathcal{E}nd(\mathcal{P}^i)$ and the latter sheaf has no higher cohomology.
It follows that $\Gamma(\tilde{\mathcal{P}}^i_\hbar)/\hbar \Gamma(\tilde{\mathcal{P}}^i_\hbar)\xrightarrow{\sim}\Gamma(\tilde{\mathcal{P}}^i)$.
Taking the quotient of (\ref{eq:eq_glob_sec})
by $\hbar$, we get an isomorphism $\Gamma(\tilde{\mathcal{P}}^1)\cong \Gamma(\tilde{\mathcal{P}}^2)$ of graded $\C[\tilde{X}]$-modules. We claim that this implies that the vector bundles
$\tilde{\mathcal{P}}^1,\tilde{\mathcal{P}}^2$ are $\C^\times$-equivariantly isomorphic.
Indeed, consider the resolution of singularities morphism $\tilde{\rho}:\tilde{X}\rightarrow\tilde{X}_0$.
This morphism is birational over any $p\in \param_{red}^*$. Moreover, for a Zariski generic $p$,
the morphism $\rho_p$ is an isomorphism, indeed, $\mu^{-1}(p)^{\theta-ss}=\mu^{-1}(p)$. It follows that
the restrictions of bundles $\tilde{\mathcal{P}}^1,\tilde{\mathcal{P}}^2$ to some Zariski open subset
in $\tilde{X}$ with codimension of complement bigger than $1$ are isomorphic. It follows that $\tilde{\mathcal{P}}^1\cong
\tilde{\mathcal{P}}^2$ and hence $\mathcal{P}^1\cong \mathcal{P}^2$.

We have seen above that $\nu_{\mathcal{P}}$ can only be one of $2|W|$ (for $n>1$)
or $|W|$ (for $n=1$) maps. This implies the upper bound on the number of Procesi
bundles.

\subsubsection{Lower bound}\label{SSS_constr_Procesi_all}
Let us show that there are $2|W|$ different Procesi bundles in the case of $n>1$.
Recall that one can construct a Procesi bundle $\mathcal{P}_{\mathcal{D}}$ once one has a
Frobenius constant quantization $\mathcal{D}$ of $X_{\Fi}$ with
$\Gamma(\mathcal{D})=\Weyl(V_{n,\Fi})^{\Gamma_n}$. Note that
the action of $W\times \Z/2\Z$ on ${\bf A}$ is defined over some algebraic extension of
$\Z$. So, as before, it can be reduced modulo $q$ for  $q=p^\ell, p\gg 0$.
Let $\mathcal{D}_\lambda$
be the Frobenius constant quantization obtained by Hamiltonian reduction with parameter
$\lambda\in \Fi_p^{Q_0}$. The parameter $\lambda$ constructed from $c=0$ belongs to
$\Fi_p^{Q_0}$. Above, we have remarked  that $\Gamma(\mathcal{D}_\lambda)\cong \Weyl(V_{n,\Fi})^{\Gamma_n}$.
Moreover, for $q\gg 0$, the stabilizer of this parameter in $W\times \Z/2\Z$ is trivial.
So we get $2|W|$ different Frobenius constant quantizations with required global sections.
Procesi bundles produced by them are different as well, as was checked in \cite[Section 3.3]{Procesi}.

\subsubsection{Canonical Procesi bundle}\label{SSS_Procesi_canon}
By a canonical Procesi bundle we mean $\mathcal{P}$ such that $\nu_{\mathcal{P}}$ is as in
(\ref{eq:nu}). According to \cite[Section 4.2]{Procesi}, this bundle has the following property:
the subbundle $\mathcal{P}^{\Gamma_{n-1}}$ coincides with the rank $n|\Gamma_1|$
bundle $\mathcal{T}$ on $X=\M^\theta(n\delta,\epsilon_0)$ induced by the $G$-module
$\bigoplus_{i\in Q_0}(\C^{n\delta_i})^{\oplus \delta_i}$. We will write $\mathcal{P}^\theta$
for this bundle. Recall that for $w\in W\times \Z/2\Z$ we get an isomorphism
$\M^\theta(n\delta,\epsilon_0)\cong \M^{w\theta}(n\delta,\epsilon_0)$ that yields
the map $\param_{red}=H^2(\M^\theta(n\delta,\epsilon_0))\oplus \C\rightarrow H^2(\M^{w\theta}(n\delta,\epsilon_0))\oplus \C=\param_{red}$ equal to $w$. It follows that $\nu_{w_*\mathcal{P}^\theta}=w\nu$. So every other Procesi
bundle on $\M^\theta(n\delta,\epsilon_0)$ is obtained as a push-forward of the canonical
Procesi bundle $\mathcal{P}^{w\theta}$ on $\M^{w\theta}(n\delta,\epsilon_0)$.

Note that when $\mathcal{P}$ is a Procesi bundle, then so is $\mathcal{P}^*$. Indeed,
$\End_{\mathcal{O}_X}(\mathcal{P}^*,\mathcal{P}^*)\cong \End_{\mathcal{O}_X}(\mathcal{P},\mathcal{P})^{opp}$.
The algebra $\C[V_n]\#\Gamma_n$ is identified with its opposite via $v\mapsto v, \gamma\mapsto \gamma^{-1},
v\in V_n^*, \gamma\in \Gamma_n$ and this gives a Procesi bundle structure on $\mathcal{P}^*$.
We have $\nu_{\mathcal{P}^{\theta*}}=w_0\sigma\nu_{\mathcal{P}^*}$, where $w_0$ is the longest element
in $W$ and $\sigma$ is the image of $1$ in $\Z/2\Z$, see \cite[Remark 4.4]{Procesi}.

\section{Macdonald positivity and categories $\mathcal{O}$}\label{S_Macdonald}
In this section we provide some applications of results of Section \ref{S_Procesi_SRA}.

In Section \ref{SS_der_equiv}, we will produce equivalences between categories $D^b(H_{1,c})$ and
$D^b(\operatorname{Coh}(\mathcal{D}_\lambda))$.

Starting from Section \ref{SS_Cat_O}, we will only consider
the groups $\Gamma_n$ with cyclic $\Gamma_1$. Here $\Gamma_n$ is a complex
reflection group and the corresponding algebra $H_{t,c}$ (called a Rational Cherednik
algebra) in this case admits a triangular decomposition. This decomposition
allows to define Verma modules and, for $t=1$, category $\mathcal{O}$ for $H_{1,c}$
that has  a so called highest weight structure. We can also define the category
$\mathcal{O}$ for $\mathcal{D}_\lambda$, this will be a subcategory in $\operatorname{Coh}(\mathcal{D}_\lambda)$.
We will show that the derived equivalence $D^b(H_{1,c}\operatorname{-mod})\cong D^b(\operatorname{Coh}(\mathcal{D}_\lambda))$
restricts to categories $\mathcal{O}$. This was used in \cite{GL} to establish
\cite[Conjecture 5.6]{rouqqsch} for the groups $\Gamma_n$.

In Section \ref{SS_Macpos} we prove Theorem \ref{Thm:Mac_pos} and also
its generalization to the groups $\Gamma_n$ due to Bezrukavnikov and Finkelberg.
The proof is based on studying the algebras $H_{0,c}$ and their Verma modules.

Finally, in Section \ref{SS_Thm_ab_loc} we prove an analog of the Beilinson-Bernstein
localization theorem, \cite{BB_loc}, for the Rational Cherednik algebras
associated to the groups $\Gamma_n$. More precisely, we answer the question when
the derived equivalence $D^b(\operatorname{Coh}(\mathcal{D}_\lambda)\rightarrow
D^b(H_{1,c}\operatorname{-mod})$ restricts to an equivalence
$\operatorname{Coh}(\mathcal{D}_\lambda)\rightarrow H_{1,c}\operatorname{-mod}$.

\subsection{Derived equivalence}\label{SS_der_equiv}
\subsubsection{Deformed derived McKay correspondence} Similarly to \ref{SSS_derived_McKay}, the  functor $R\Gamma(\mathcal{P}\otimes_{\mathcal{O}_X}\bullet)$ defines an equivalence
$D^b(\operatorname{Coh}X)\xrightarrow{\sim} D^b(\C[V_n]\#\Gamma_n\operatorname{-mod})$
with quasi-inverse $\mathcal{P}^*\otimes^L_{\C[V_n]\#\Gamma_n}\bullet$. These equivalence
automatically upgrade to the categories of $\C^\times$-equivariant objects:
$D^b(\operatorname{Coh}^{\C^\times}X)\cong D^b(\C[V_n]\#\Gamma_n\operatorname{-mod}^{\C^\times})$ defined in the
same way.

Now let us consider the deformation  $\tilde{\mathcal{P}}_\hbar$ of $\mathcal{P}$ to a right
$\C^\times$-equivariant $\widetilde{\mathcal{D}}_\hbar$-module. It gives a functor
$\tilde{\mathcal{F}}:=R\Gamma(\tilde{\mathcal{P}}_{\hbar,fin}\otimes_{\widetilde{\mathcal{D}}_{\hbar,fin}}\bullet):
D^b(\operatorname{Coh}^{\C^\times}(\tilde{\mathcal{D}}_{\hbar,fin}))\rightarrow
D^b({\bf H}\operatorname{-mod}^{\C^\times})$.
This functor has left adjoint and right inverse $\tilde{\mathcal{G}}=\tilde{\mathcal{P}}_{\hbar,fin}^*\otimes^L_{\bf H}\bullet$. So we get the adjunction morphism $\tilde{\mathcal{G}}\circ \tilde{\mathcal{F}}\rightarrow \operatorname{id}$.
One can show (see \cite[Section 5]{GL} for details) that since this morphism is an isomorphism modulo $\param_{univ}$,
it is an isomorphism itself.

\subsubsection{Specialization}
The equivalence $\tilde{\mathcal{F}}$ can be specialized to a numerical parameter. In particular, we get  equivalences
$D^b(\operatorname{Coh}(\mathcal{D}_\lambda))\rightarrow D^b(H_{1,c}\operatorname{-mod})$, where $\lambda$
is recovered from $c$ as in Theorem \ref{Thm:iso}. This is done in two steps. First, one gets a derived equivalence
between $\operatorname{Coh}^{\C^\times}(R_{\hbar^{1/2}}(\mathcal{D}_\lambda))$ and
$R_{\hbar^{1/2}}(H_{1,c})\operatorname{-mod}^{\C^\times}$, the corresponding sheaf and algebra
are obtained from $\widetilde{\mathcal{D}}_{\hbar,fin}, {\bf H}$ by base change (and the
equivalence we need comes from the corresponding base change of $\tilde{\mathcal{P}}_{\hbar,fin}$).
To do the second step  we recall that $H_{1,c}\operatorname{-mod}$ is the quotient
$R_{\hbar^{1/2}}(H_{1,c})\operatorname{-mod}^{\C^\times}$ by the full subcategory of the $\C[\hbar]$-torsion
modules and the similar claim holds for $\operatorname{Coh}(\mathcal{D}_\lambda)$, see Lemma \ref{Lem:cat_equi}.
It follows that $D^b(H_{1,c}\operatorname{-mod})$ is the quotient of $D^b(R_{\hbar^{1/2}}(H_{1,c})\operatorname{-mod}^{\C^\times})$ by the category of all complexes whose homology are $\C[\hbar]$-torsion and a similar claim holds for $\mathcal{D}_\lambda$. Since
the equivalence $D^b(R_{\hbar^{1/2}}(H_{1,c})\operatorname{-mod}^{\C^\times})\cong
D^b(\operatorname{Coh}^{\C^\times}(R_{\hbar^{1/2}}(\mathcal{D}_\lambda)))$ is $\C[\hbar]$-linear
by the construction, they induce
\begin{equation}\label{eq:equi_spec} D^b(H_{1,c}\operatorname{-mod})\cong
D^b(\operatorname{Coh}(\mathcal{D}_\lambda)).\end{equation}

\subsubsection{Application: shift equivalences}
The  equivalences (\ref{eq:equi_spec}) can be applied to producing a result that only concerns the symplectic reflection
algebras. Namely, we say that parameters $c,c'$ for $H_?$ have integral difference if $\lambda-\lambda'\in \Z^{Q_0}$
for the corresponding parameters $\lambda$. Recall that we can view $\chi\in \Z^{Q_0}$ as a character of
$G$. So $\chi$ defines a line bundle on $X$, explicitly, $\mathcal{O}_\chi=\pi_*(\Str_{\mu^{-1}(0)^{\theta-ss}})^{G,\chi}$.
This line bundle can be quantized to a $\mathcal{D}_{\lambda+\chi}$-$\mathcal{D}_{\lambda}$-bimodule to be denoted
by $\mathcal{D}_{\lambda,\chi}$. Explicitly,
$$\mathcal{D}_{\lambda,\chi}:=\pi_*(\mathcal{D}^{ss}/\mathcal{D}^{ss}\{\Phi(x)-\langle \lambda,x\rangle\})^{G,\chi}.$$
This bundle carries a natural filtration and an isomorphism $\gr \mathcal{D}_{\lambda,\chi}\cong \mathcal{O}_\chi$
follows from the flatness of the moment map.

Note that there is a natural (multiplication) homomorphism $\Dcal_{\lambda+\chi,\chi'}\otimes_{\Dcal_{\lambda+\chi}}\Dcal_{\lambda,\chi}
\rightarrow \Dcal_{\lambda,\chi+\chi'}$ that becomes the isomorphism $\mathcal{O}_{\chi'}\otimes
\mathcal{O}_\chi\rightarrow \mathcal{O}_{\chi+\chi'}$ after passing to the associated graded.
So the multiplication homomorphism itself is an isomorphism. It follows that a functor
$\Dcal_{\lambda,\chi}\otimes_{\Dcal_\lambda}\bullet:\operatorname{Coh}(\Dcal_\lambda)\rightarrow
\operatorname{Coh}(\Dcal_{\lambda+\chi})$
is a category equivalence. We conclude that categories $D^b(H_{1,c}\operatorname{-mod})$ and $D^b(H_{1,c'}\operatorname{-mod})$ are equivalent provided $c,c'$ have integral difference.


\subsection{Category $\mathcal{O}$}\label{SS_Cat_O}
Starting from now on, we assume that $\Gamma_1$ is a cyclic group $\Z/\ell\Z$.
Recall that in this case
the space $V_n$ (equal to $\C^{2n}$ when $\ell>1$ and $\C^{2(n-1)}$ when $\ell=1$) splits
as $\h\oplus\h^*$, where $\h$ is a standard reflection representation of the group $\Gamma_n$.
The embeddings $\h,\h^*\hookrightarrow {\bf H}$ extend to algebra embeddings $S(\h),S(\h^*)
\hookrightarrow {\bf H}$.  These embeddings give rise to the {\it triangular decomposition}
${\bf H}=S(\h^*)\otimes S(\param_{univ})\Gamma_n\otimes  S(\h)$. We can also consider the specialization
$H_{1,c}=S(\h^*)\otimes \C\Gamma_n\otimes S(\h)$ (here and below $c$ is a numerical
parameter) of this decomposition.

\subsubsection{Category $\mathcal{O}$ for $H_{1,c}$}\label{SSS_cat_O_H}
By definition, the category $\mathcal{O}$ for $H_{1,c}$ consists of all $H_{1,c}$-modules
$M$ such that
\begin{itemize}
\item[(i)] $\h$ acts locally nilpotently on $M$.
\item[(ii)] $M$ is finitely generated over $H_{1,c}$.
\end{itemize}
Note that, modulo (i), the condition (ii) is equivalent to
\begin{itemize}
\item[(ii$'$)] $M$ is finitely generated over $S(\h^*)$.
\end{itemize}
An example of an object in the category $\mathcal{O}$ is a Verma module constructed as follows.
Pick  an irreducible representation $\tau$ of $\Gamma_n$
and view it as a $S(\h)\#\Gamma_n$-module by making $\h$ act by $0$. Then set
$\Delta_{1,c}(\tau):=H_{1,c}\otimes_{S(\h)\#\Gamma_n}\tau$.
As a $S(\h^*)\#W$-module, $\Delta_{1,c}(\tau)$ is naturally identified
with $S(\h^*)\otimes \tau$ (the algebra $S(\h^*)$ acts by multiplications
from the left, and $W$ acts diagonally).

The algebra $H_{1,c}$ carries an {\it Euler grading} given by $\deg \h=-1, \deg \h^*=1, \deg W=0$.
This grading is internal: we have an element $h\in H_{1,c}$ with $[h,a]=da$ for $a\in H_{1,c}$
of degree $d$. Explicitly, the element $h$ is given by
$$\sum_{i=1}^m x_i y_i+\sum_{s\in S} \frac{c(s)}{1-\lambda_s}s.$$ Here the notation is as follows.
We write $y_1,\ldots,y_m$ for a basis in $\h$ (of course, $m=n$ for $\ell>1$ and $m=n-1$
for $\ell=1$) and $x_1,\ldots,x_m$ for the dual basis in $\h^*$. By $S$ we, as usual, denote the set
of reflections in $\Gamma_n$ and $c(s)$ stands for $c_i$ if $s\in S_i$
(note that the formula for $h$ is different from the usual formula for the Euler element,
see, e.g., \cite[Section 2.1]{BE}, because our $c(s)$ is rescaled).
Finally, $\lambda_s$ is the eigenvalue of $s$ in $\h^*$ different from $1$.

Using the element $h$, we can show that every Verma module $\Delta_{1,c}(\tau)$ has a unique
simple quotient. These quotients form a complete collection of the simple objects in
$\mathcal{O}$. Also one can show that every object in $\mathcal{O}$ has finite length.
These claims are left as exercises to the reader.

\subsubsection{Category $\mathcal{O}$ for $\Dcal_\lambda$}\label{SSS_loc_O}
We have a $\C^\times$-action on $D(R)$ induced by the $\C^\times$-action on $R$ given by
$t.r:=t^{-1}r$. This action is Hamiltonian, the corresponding quantum comoment map
$\Phi:\C\rightarrow D(R)$ sends $1$ to the Euler vector field. The action descends to
a Hamiltonian $\C^\times$-action on $\Dcal_\lambda$ for any $\lambda$.

Consider the corresponding Hamiltonian $\C^\times$-action on $X=\M^\theta_0(n\delta,\epsilon_0)$.
Recall that the resolution of singularities morphism $X\rightarrow (\h\oplus \h^*)/\Gamma_n$
becomes $\C^\times$-equivariant if we equip the target variety with the $\C^\times$-action
induced by $t.(a,b)=(t^{-1}a, tb), a\in \h, b\in \h^*$. This action has finitely many fixed
points that are in a natural bijection with the irreducible representations of $\Gamma_n$, see
\cite[5.1]{Gordon_O}.  Namely, $X^{\C^\times}$ is in a natural bijection with $\M^0_{p}(n\delta,\epsilon_0)^{\C^\times}$, where $p\in \g^{*G}$ is generic. Indeed, $\M^0_{p}(n\delta,\epsilon_0)=\M^\theta_{p}(n\delta,\epsilon_0)$
and the sets $\M^\theta_p(n\delta,\epsilon_0)^{\C^\times}$ are identified for all $p$
by continuity. Let $c$ be a parameter corresponding to $p$ (meaning that $\nu(0,c)=(0,p)$).
Then we can consider the Verma module $\Delta_{0,c}(\tau):=H_{0,c}\otimes_{S(\h)\#\Gamma_n}\tau$.
The subalgebra $S(\h^*)^{\Gamma_n}$ is easily seen to be central. Let us write $S(\h^*)^{\Gamma_n}_+$ for the augmentation
ideal in $S(\h^*)^{\Gamma_n}$. Following \cite{Gordon_baby}, consider the {\it baby Verma module}
$\underline{\Delta}_{0,c}(\tau):=\Delta_{0,c}(\tau)/S(\h^*)^{\Gamma_n}_+ \Delta_{0,c}(\tau)
\cong S(\h^*)/(S(\h^*)^{\Gamma_n}_+)\otimes \tau$ (the last isomorphism is that of
$S(\h^*)\#\Gamma_n$-modules). This module is easily seen to be indecomposable so it has a central
character that is a point of $\operatorname{Spec}(Z(H_{0,c}))=\M^0_{p}(n\delta,\epsilon_0)$. Clearly, this point is fixed by $\C^\times$ and this defines a map $\operatorname{Irr}(\Gamma_n)\rightarrow \M^0_{p}(n\delta,\epsilon_0)^{\C^\times},\tau\mapsto z_\tau,$ that was shown to be a bijection
in \cite{Gordon_O}.

Fix some $p\in \g^{*G}$. Consider the attracting locus $Y_p\subset \M^\theta_p(n\delta,\epsilon_0)$ for the $\C^\times$-action.
Since this action has finitely many fixed points, we see that $Y_p$ is a lagrangian subvariety
with  irreducible components indexed by $\operatorname{Irr}(\Gamma_n)$. Namely, to $\tau\in
\operatorname{Irr}(\Gamma_n)$ we assign the attracting locus $Y_p(\tau):=\{z\in \M^\theta_p(n\delta,\epsilon_0)|
\lim_{t\rightarrow 0}t.z=z_\tau\}$. The irreducible components of $Y_p$ are the closures $\overline{Y_p(\tau)}$.
When $p$ is Zariski generic, the subvarietes $Y_p(\tau)$ are already closed.

By the category
$\mathcal{O}^{loc}$ for $\Dcal_\lambda$ we mean the full category of coherent $\Dcal_\lambda$-modules
that are supported on $Y$ (see \ref{SSS_supports}) and admit a $\C^\times$-equivariant structure
compatible with the $\C^\times$-action on $\mathcal{D}_\lambda$.
Such categories were systematically studied in \cite{BLPW}. In particular, it was shown that
all modules in $\mathcal{O}^{loc}$ have finite length and are indexed by $\M^\theta_0(n\delta,\epsilon_0)^{\C^\times}$,
see \cite[Sections 3.3,5.3]{BLPW}.

\subsubsection{Choice of identification $X^{\C^\times}\cong \operatorname{Irr}(\Gamma_n)$}\label{SSS_alt_order}
We note that despite our identification of $X^{\C^\times}$ with $\operatorname{Irr}(\Gamma_n)$ is natural,
there are other natural choices as well. The choice we have made is good for working with
the category $\mathcal{O}$. We could also consider the category $\mathcal{O}^*$, where the modules
are locally nilpotent for $\h^*$, not for $\h$ (and are still finitely generated over $H_{1,c}$).
Consequently, we need to use the opposite
Hamiltonian $\C^\times$-action on $X,\M^0_p(n\delta,\epsilon_0)$ and Verma modules
$\Delta^*_{0,c}(\tau):=H_{0,c}\otimes_{S(\h^*)\#\Gamma_n}\tau$. Let us explain how the bijection
$X^{\C^\times}\cong \operatorname{Irr}(\Gamma_n)$ changes.

All simple constituents of $\underline{\Delta}_{0,c}(\tau)$ are isomorphic modules of dimension
$|\Gamma_n|$ (indeed, $H_{0,c}$ is the endomorphism algebra of the rank $|\Gamma_n|$ bundle $\tilde{\mathcal{P}}_p$
on $\M^0_p(n\delta,\epsilon_0)$). Let us denote this simple module by $L_{0,c}(\tau)$. This module
is graded, the highest graded component is $\tau$. Let us determine the lowest graded component
in $L_{0,c}(\tau)$.  This component coincides with the lowest graded component in $\Delta_{0,c}(\tau)$
that is the tensor product  of $\tau$ with the lowest degree component in $\C[\h]/(\C[\h]^{\Gamma_n})_+$.
It is easy to see that the latter is $\Lambda^{top}\h$. Abusing the notation,
we will denote $\tau\otimes \Lambda^{top}\h$ by $\tau^t$. When $\Gamma_1=\{1\}$
we can use the standard identification of $\operatorname{Irr}(\mathfrak{S}_n)$ with the set of Young
diagrams of $n$ boxes.  In this case, $\Lambda^{top}\h$ is the sign representation of $\mathfrak{S}_n$
and $\tau^t$ indeed corresponds to the transposed Young diagram of $\tau$.

The previous paragraph shows that there is an epimorphism $\Delta^*_{0,c}(\tau^t)\twoheadrightarrow L_p(\tau)$.
So our new bijection sends the point $z_\tau\in X^{\C^\times}$ to $\tau^t$.

We also note that the identification  $X^{\C^\times}\cong \operatorname{Irr}(\Gamma_n), \tau\mapsto z_\tau,$
depends on the choice of a Procesi bundle $\mathcal{P}$ but we are not going to use this.

\subsubsection{Highest weight structures}\label{SSS_HWS}
Let us recall the definition of a highest weight category. Let $\mathcal{C}$ be an abelian category
that is equivalent to the category of modules over a finite dimensional algebra, equivalently, the
category $\mathcal{C}$ has finitely many simples, enough projectives and finite dimensional Hom's
(and hence every object has finite length). Let $\mathcal{T}$ denote an indexing set of the simple
objects in $\mathcal{C}$, we write $L(\tau)$ for the simple object indexed by $\tau\in \mathcal{T}$
and $P(\tau)$ for its projective cover. The additional structure of a highest weight category
is a partial order $\leqslant$ on $\mathcal{T}$ and a collection of so called {\it standard} objects
$\Delta(\tau), \tau\in \mathcal{T}$, satisfying the following axioms:
\begin{enumerate}
\item $\operatorname{Hom}_{\mathcal{C}}(\Delta(\tau),\Delta(\tau'))\neq 0$ implies
$\tau\leqslant \tau'$,
\item $\operatorname{End}_{\mathcal{C}}(\Delta(\tau))=\C$.
\item $P(\tau)\twoheadrightarrow \Delta(\tau)$ and the kernel admits a filtration
with quotients $\Delta(\tau')$ for $\tau'>\tau$.
\end{enumerate}

\begin{Rem}\label{Rem:stand_from_order}
Let us point out that the standard objects are uniquely recovered from the partial order.
Namely, consider the category $\mathcal{C}_{\leqslant \tau}$ that is the Serre span
of the simples $L(\tau')$ with $\tau'\leqslant \tau$. Then $\Delta(\tau)$ is the projective
cover of $L(\tau)$ in $\mathcal{C}_{\leqslant \tau}$.
\end{Rem}

Both categories $\mathcal{O},\mathcal{O}^{loc}$ that were described above are highest weight,
see \cite[Sections 2.6,3.2]{GGOR} for $\mathcal{O}$ and \cite[Section 5.3]{BLPW} for $\mathcal{O}^{loc}$.
The standard objects $\Delta(\lambda)$ are the Verma modules. The order can be introduced
as follows. Recall the element $h\in H_{1,c}$ introduced in \ref{SSS_cat_O_H}.
It acts on $\tau\subset \Delta(\tau)$
by $\displaystyle \sum_{s\in S}\frac{c(s)}{1-\lambda_s}s.$ The latter element in $\C \Gamma_n$ is central
and so acts on $\tau$ by a scalar, denote that scalar by $c_\tau$. Then we set
$\tau\leqslant\tau'$ if $c_{\tau}-c_{\tau'}\in \Z_{\geqslant 0}$.

Let us provide a formula
for $c_\tau$. We start with $\ell=1$.
Then a classical computation shows that $c_\tau=c_0\operatorname{cont}(\tau)/2$, where the integer $\operatorname{cont}(\tau)$
is defined as follows. For the box $b\in \tau$ lying in $x$th column and $y$th row,
we set $\operatorname{cont}(b):=x-y$. Then $\operatorname{cont}(\tau):=\sum_{b\in \tau}
\operatorname{cont}(b)$. Now let us proceed to $\ell>1$. In this case, the irreducible
representations of $\Gamma_n$ are parameterized by the $\ell$-multipartitions $(\tau^{(1)},\ldots,\tau^{(\ell)})$
of $n$. Define elements $\lambda_1,\ldots,\lambda_{\ell}$ by requiring that $\lambda_i, i=1,\ldots,\ell-1,$
is recovered from $c$ as in Theorem \ref{Thm:iso} and $\sum_{i=1}^{\ell}\lambda_i=0$. For a box $b\in \tau^{(j)}$
set $d_c(b):=c_0\ell\operatorname{cont}(b)/2+\ell \lambda_j$. Then, up to a summand independent of $\tau$,
we have $c_\tau=\sum_{b\in \tau} d_c(b)$, see \cite[Proposition 6.2]{rouqqsch} or
\cite[2.3.5]{GL} (in both papers the notation is different from what we use).

In fact, one can take a weaker ordering on $\operatorname{Irr}(\Gamma_n)$ making $\mathcal{O}$
into a highest weight category. Namely, according to \cite{Griffeth}, for two boxes $b, b'$
in $j$th and $j'$th diagrams respectively we say that $b\leqslant b'$ if $d_c(b)-d_c(b')$
is congruent to $j-j'$ modulo $\ell$ and is in $\Z_{\geqslant 0}$. Then $\lambda\leqslant \lambda'$
if one can order boxes $b_1,\ldots,b_n$ of $\lambda$ and $b_1',\ldots,b_n'$ of $\lambda'$
in such a way that $b_i\leqslant b'_i$ for all $i$.

Let us proceed to the categories $\mathcal{O}^{loc}$. They are highest weight with respect to
the  order $\leqslant$ (we will often write $\leqslant^\theta$ to indicate the dependence
on $\theta$) defined as follows. We first define a pre-order $\leqslant'$ by setting $\tau\leqslant' \tau'$ if $z_{\tau}\in \overline{Y}_{\tau'}$ and then define $\leqslant$ as the transitive closure of $\leqslant'$.
\begin{Ex}\label{Ex:geom_order_Hilbert}
When $\ell=1$ and $\theta<0$, the bijection between the $\C^\times_h$-fixed points
and partitions is the standard one.   A combinatorial description of $\leqslant^\theta$
follows from \cite[Section 4]{Nakajima_Jack}: we have $\tau\leqslant^\theta \tau'$
if $\tau\leqslant \tau'$ as Young diagrams.
\end{Ex}

In the case when $\ell>1$ an a priori stronger order (that automatically also makes $\mathcal{O}^{loc}$
into a highest weight category) was described by Gordon in \cite[Section 7]{Gordon_O} in combinatorial terms.
The standard modules are recovered from $\leqslant^\theta$ as before. Below we will see that they can be described
using the deformations of the Procesi bundle.

\subsubsection{Derived equivalence}
Here we are going to produce a derived equivalence $D^b(\mathcal{O})\cong D^b(\mathcal{O}^{loc})$.

Inside $D^b(H_{1,c}\operatorname{-mod})$ we can consider the full subcategory $D^b_{\mathcal{O}}(H_{1,c}\operatorname{-mod})$ consisting of all complexes
whose homology lie in the category $\mathcal{O}$. We then have a natural
functor $D^b(\mathcal{O})\rightarrow D^b_{\mathcal{O}}(H_{1,c}\operatorname{-mod})$.
This functor is an equivalence by \cite[Proposition 4.4]{Etingof_affine}. We can also consider the
category $D^b_{\mathcal{O}}(\operatorname{Coh}(\mathcal{D}_\lambda))$, the functor
$D^b(\mathcal{O}^{loc})\rightarrow D^b_{\mathcal{O}}(\operatorname{Coh}(\mathcal{D}_\lambda))$
is an equivalence as well, this follows   from  \cite[Corollary 5.13]{BLPW} and \cite[Corollary 5.12]{BPW}.

The equivalence $D^b(H_{1,c}\operatorname{-mod})\xrightarrow{\sim} D^b(\operatorname{Coh}(\mathcal{D}_\lambda))$
is compatible with the supports in the following sense. Recall that we have two commuting $\C^\times$-actions.
The Hamiltonian torus will be denoted by $\C^{\times}_h$, while, for the contracting torus
(which is present even when $\Gamma_1$ is not cyclic), we will write $\C^\times_c$.
Pick a closed subvariety $Y_0\subset (\h\oplus \h^*)/\Gamma_n$
that is stable under the $\C^\times_c$-action. Consider the full subcategory $D^b_{Y_0}(H_{1,c}\operatorname{-mod})$
in $D^b(H_{1,c})$ of all complexes with homology supported on $Y_0$. Set $Y:=\rho^{-1}(Y_0)$, where, recall,
$\rho$ stands for the resolution of singularities morphism $\rho:X\rightarrow V_n/\Gamma_n$
and consider the subcategory $D^b_{Y}(\operatorname{Coh}(\mathcal{D}_\lambda))\subset D^b(\operatorname{Coh}(\mathcal{D}_\lambda))$.
Then the equivalence $D^b(\operatorname{Coh}(\mathcal{D}_\lambda))\cong D^b(H_{1,c}\operatorname{-mod})$
restricts to $D^b_Y(\operatorname{Coh}(\mathcal{D}_\lambda))\cong D^b_{Y_0}(H_{1,c}\operatorname{-mod})$.

Note that the bundle $\mathcal{P}$ on $X$ is $(\C^\times)^2$-equivariant.
Therefore the deformation $\tilde{\mathcal{P}}_\hbar$ is $(\C^\times)^2$-equivariant
as well. It follows that the equivalence $D^b(\operatorname{Coh}(\mathcal{D}_\lambda))\cong
D^b(H_{1,c}\operatorname{-mod})$ preserves complexes whose homology admit $\C^\times$-equivariant
liftings. Combined with the previous paragraph, this means that we get an equivalence
$D^b_{\mathcal{O}}(H_{1,c}\operatorname{-mod})\cong D^b_{\mathcal{O}}(\operatorname{Coh}(\mathcal{D}_\lambda))$
and hence an equivalence $D^b(\mathcal{O})\cong D^b(\mathcal{O}^{loc})$.

This was used in \cite[Section 5]{GL} to prove a conjecture of Rouquier, \cite[Conjecture 5.6]{rouqqsch}.
Namely, suppose that we have parameters $c,c'$ such that the corresponding
parameters $\lambda,\lambda'$ have integral difference. Then we have an abelian equivalence
$\operatorname{Coh}(\mathcal{D}_\lambda)\xrightarrow{\sim} \operatorname{Coh}(\mathcal{D}_{\lambda'})$,
given by tensoring with the bimodule $\mathcal{D}_{\lambda,\lambda'-\lambda}$. This bimodule is $\C^\times_h$-equivariant,
this follows from the construction. Also it is clear that tensoring with $\mathcal{D}_{\lambda,\lambda'-\lambda}$
preserves the supports. So we conclude that $\mathcal{O}^{loc}_{\lambda}\xrightarrow{\sim} \mathcal{O}^{loc}_{\lambda'}$.
It follows that the categories $\mathcal{O}_c$ and $\mathcal{O}_{c'}$ are derived equivalent
that was conjectured by Rouquier (in the generality of all Cherednik algebras).

\subsection{Macdonald positivity}\label{SS_Macpos}
Consider the ${\bf H}$-module ${\bf \Delta}(\lambda):={\bf H}\otimes_{S(\h)\#\Gamma_n}\lambda$. Recall
the derived equivalence $D^b(\operatorname{Coh}(\widetilde{\Dcal}_{\hbar,fin}))\xrightarrow{\sim}
D^b({\bf H}\operatorname{-mod})$ given by
$$\mathcal{F}:=\Gamma(\widetilde{\mathcal{P}}_{\hbar,fin}\otimes_{\widetilde{\Dcal}_{\hbar,fin}}\bullet)$$
and its inverse $\mathcal{G}$.  It turns out that the study of the objects $\mathcal{G}({\bf \Delta}(\lambda))$
leads to the proof of the Macdonald positivity. The proof that we provide below is morally similar
to but  different from the original proof in \cite{BF}.

\subsubsection{Flatness}
A key step in the proof is to establish the flatness over $\C[\h]$ of an arbitrary Procesi bundle
$\mathcal{P}$, where we view $\mathcal{P}$
 ($\C[\h]$ acts on $\mathcal{P}$ via the inclusion $\C[\h]\hookrightarrow S(\h\oplus \h^*)\#\Gamma_n=\operatorname{End}_{\mathcal{O}_X}(\mathcal{P})$). This will imply that
the Koszul complex
$$\mathcal{P}\leftarrow \h^*\otimes\mathcal{P}\leftarrow \Lambda^2\h^*\otimes\mathcal{P}\leftarrow
\ldots\leftarrow \Lambda^n\h^*\otimes\mathcal{P}$$
is a resolution of $\mathcal{P}/\h^* \mathcal{P}$. The proof of the flatness is taken
from the proof of \cite[Lemma 3.7]{BF}.

Note that, since $\Gamma_n$ is a complex
reflection group, $\C[\h]$ is free over $\C[\h]^{\Gamma_n}$. So it is enough to show that
$\mathcal{P}$ is flat over $\C[\h]^{\Gamma_n}$.

Let us recall how $\mathcal{P}$ was constructed, see \ref{SSS_constr_Procesi_p} (construction of one Procesi
bundle in characteristic $p\gg 0$), \ref{SSS_constr_Procesi_0} (construction of one Procesi bundle
in characteristic $0$), \ref{SSS_constr_Procesi_all} (construction of all Procesi bundles).

\begin{enumerate}
\item We start with a suitable Frobenius constant quantization $\mathcal{D}$ of $X_{\Fi}$, where $\Fi$
is an algebraically closed field of characteristic $0$.
\item Then we take a splitting bundle $\mathcal{B}$ of $\mathcal{D}|_{X_{\Fi}^{(1)\wedge_0}}$.
\item We form a bundle $\mathcal{P}'$ on $X_{\Fi}^{(1)\wedge_0}$ that is the sum of indecomposable summands
of $\mathcal{S}^*$ with suitable multiplicities. Then we extend this bundle to $X_{\Fi}^{(1)}$
and get a Procesi bundle $\mathcal{P}^{(1)}_{\Fi}$ on $X_{\Fi}^{(1)}$.
\item Since $X_{\Fi}^{(1)}\cong X_{\Fi}$ as  $\Fi$-varieties, we can view $\mathcal{P}^{(1)}_{\Fi}$ as a bundle
$\mathcal{P}_{\Fi}$ on $X$.
\item Then we lift $\mathcal{P}_{\Fi}$ to characteristic 0.
\end{enumerate}

The procedure in (5) implies that if $\mathcal{P}_{\Fi}$ is flat over $\Fi[\h]^{\Gamma_n}$,
then $\mathcal{P}$ is flat over $\C[\h]^{\Gamma_n}$ (the reader is welcome to verify the technical
details). Obviously, $\mathcal{P}_{\Fi}$ is flat over $\Fi[\h]^{\Gamma_n}$ if and only if
$\mathcal{P}_{\Fi}^{(1)}$ is flat over $\Fi[\h^{(1)}]^{\Gamma_n}$. The latter is equivalent to
$\mathcal{B}^*$ being flat over $\Fi[[\h^{(1)}]]^{\Gamma_n}$, which, in turn, is equivalent
to the claim that $\mathcal{D}$ is a flat $\Fi[\h^{(1)}]^{\Gamma_n}$-module. But $\operatorname{gr}\mathcal{D}
\cong \operatorname{Fr}^X_*\mathcal{O}_{X_\Fi}$. So it is enough to verify that
$\mathcal{O}_{X_\Fi}$ is flat over $\Fi[\h^{(1)}]^{\Gamma_n}$. Since $\Fi[\h]^{\Gamma_n}$
is flat over $\Fi[\h^{(1)}]^{\Gamma_n}$, we reduce to proving that $X_{\Fi}$ is flat over
$\h_{\Fi}/\Gamma_n$, equivalently, all fibers of $X_{\Fi}\rightarrow \h_{\Fi}/\Gamma_n$
have the same dimension, equivalently, the zero fiber has dimension $\dim\h$. But the zero fiber of
this map is precisely the contracting variety for the Hamiltonian $\Fi^\times$-action
and so is lagrangian. This completes the proof.

Similarly, $\mathcal{P}$ is flat over $\C[\h^*]$.  Also let us recall, see \ref{SSS_Procesi_canon},
that $\mathcal{P}^*$ can be equipped with a structure of the Procesi bundle, for which we need to
convert  the right $S(\h\oplus \h^*)\#\Gamma_n$-module into a left $S(\h\oplus\h^*)\#\Gamma_n$
using a natural anti-automorphism of $S(\h\oplus\h^*)\#\Gamma_n$. This shows that $\mathcal{P}^*$
is a flat {\it right} module over both $\C[\h]$ and $\C[\h^*]$. This is what we are going to
use below.

\subsubsection{Upper triangularity}
Let $\theta$ be a generic stability condition and take $X=X^\theta$. This gives rise to the partial order
$\leqslant^\theta$ on the set $\operatorname{Irr}(\Gamma_n)$ described in \ref{SSS_loc_O}. Recall that we
write $z_\tau$ for the $\C^\times_h$-fixed point in $X$ corresponding to $\tau$
as explained in \ref{SSS_loc_O}.
We write $Y_\tau$ for  the $\C^\times_h$-contracting component of $z_\tau$, a
lagrangian subvariety in $X^\theta$. Further, write $e_\tau$ for a primitive idempotent in
$\C\Gamma_n$ corresponding to $\tau$ so that $\tau \cong (\C\Gamma_n)e_\tau$.

\begin{Prop}\label{Prop:upper_triang}
Let $\mathcal{P}$ be the canonical Procesi bundle on $X^\theta$. Then the sheaf $(\mathcal{P}^*/\mathcal{P}^*\mathfrak{h})e_\tau$ is supported on
$\bigcup_{\tau'\leqslant^\theta \tau}Y_{\tau'}$.
\end{Prop}
\begin{proof}
Consider the deformation $\widetilde{\mathcal{P}}^*$ of $\mathcal{P}^*$ to $\widetilde{X}$.  It is flat over $\C[\g^{*G},\h^*]$. Therefore $\widetilde{\mathcal{P}}^*/\widetilde{\mathcal{P}}^*\h$ is flat over
$\C[\g^{*G}]$. It follows that $\operatorname{Supp}((\mathcal{P}^*/\mathcal{P}^*\h)e_\tau)\subset
\overline{\C_c^\times\operatorname{Supp}(\mathcal{P}_p^*/\mathcal{P}_p^*\h)e_\tau}$ for a generic
$p\in \g^{*G}$. But $(\mathcal{P}_p^*/\mathcal{P}_p^*\h)e_\tau$ is nothing else but
$e\Delta_{0,c}(\tau)$. We claim that $\operatorname{Supp}\Delta_{0,c}(\tau)\subset Y_{p,\tau}$.
Indeed, we have shown in \ref{SSS_loc_O} that $\Delta_{0,c}(\tau)/S(\h^*)^{\Gamma_n}_+\Delta_{0,c}(\tau)$
is supported in $z_{p,\tau}$, the point in $\M_p^0(n\delta,\epsilon_0)^{\C^\times_h}$ indexed by $\tau$. If $\operatorname{Supp}\Delta_{0,c}(\tau)\not\subset Y_{p,\tau}$,
then there is $\tau'\neq \tau$ with $z_{p,\tau'}\in \operatorname{Supp}\Delta_{0,c}(\tau)$
(because the latter is closed and contained in $Y_p$). The support of $\Delta_{0,c}(\tau)$
is disconnected and so the module $\Delta_{0,c}(\tau)$ is indecomposable.  From here one deduces that
$z_{p,\tau'}$ lies in the support of $\Delta_{0,c}(\tau)/S(\h^*)^{\Gamma_n}_+\Delta_{0,c}(\tau)$, contradiction.

Now the inclusion
$$\operatorname{Supp}\left((\mathcal{P}^*/\mathcal{P}^*\mathfrak{h})e_\tau\right)
\subset \bigcup_{\tau'\leqslant^\theta \tau}Y_{\tau'}$$
follows from
$$\overline{\C^\times Y_{p,\tau}}\cap X^\theta\subset \bigcup_{\tau'\leqslant^\theta\tau} Y_{\tau'},$$ see \cite[Lemma 3.8]{BF}.
\end{proof}

In fact, $e\Delta_{0,c}(\tau)=\C[Y_{p,\tau}]$ but we do not need this fact.

\subsubsection{Wreath-Macdonald positivity}
Now we are ready to prove the Macdonald positivity theorem, Theorem \ref{Thm:Mac_pos},
and its ``wreath-generalization'' due to Bezrukavnikov and Finkelberg.

First of all, Proposition \ref{Prop:upper_triang} implies that if the fiber
of $[\mathcal{P}^*/\mathcal{P}^*\h]e_\tau$ in $z_{\tau'}$ is nonzero, then
$\tau'\leqslant^\theta \tau$. It follows that if $\tau^*$ is a constituent of
the fiber $(\mathcal{P}^*/\mathcal{P}^*\h)_{z_{\tau'}}$, then $\tau\geqslant^\theta \tau'$.
But since $\mathcal{P}^*$ is a flat right $\C[\h]$-module, we see that
the class of $[\mathcal{P}^*/\mathcal{P}^*\h^*]_{z_{\tau'}}$ in the $K_0$
of bigraded $\Gamma_n$-modules coincides with
that of the Koszul resolution
$$\mathcal{P}^*_{z_{\tau'}}\leftarrow \mathcal{P}^*_{z_{\tau'}}\otimes \h\leftarrow\ldots$$
Taking the duals, we see that if $\tau$ occurs in the class
$$\mathcal{P}_{z_{\tau'}}\otimes \sum_{i=0}^{\dim \h}(-1)^i \Lambda^i\h^*,$$
then $\tau'\leqslant^\theta\tau$. When $\Gamma_1=\{1\}$, this yields (a) from
Definition \ref{defi_Macdonald}.

To get (b) in that definition (and its wreath-generalization), we consider
$[\mathcal{P}^*/\mathcal{P}^*\h^*]e_\tau$. This sheaf is supported
on the union of repelling components for $\C^{\times}_h$ and can have nonzero
fibers only in the fixed points $z_{\tau'}$ with $z_{\tau'}\geqslant^\theta z_{\tau^t}$
meaning $\tau^t\leqslant^\theta \tau'$. In other words, if $\tau$ appears in
$$\mathcal{P}_{z_{\tau'}}\otimes \sum_{i=0}^{\dim \h}(-1)^i \Lambda^i\h,$$
then $\tau^t\leqslant^\theta \tau'$.  When $\Gamma_1=\{1\}$, this yields
(b) in Definition \ref{defi_Macdonald}. (c) there follows because $\mathcal{P}$
is normalized.

\subsection{Localization theorem}\label{SS_Thm_ab_loc}
Let $\mathcal{P}_{1,\lambda}$ denote the the right $\mathcal{D}_\lambda$-module
obtained by specializing $\widetilde{\mathcal{P}}_\hbar$.
One can ask when (i.e., for which $\lambda$) the functor $\Gamma(\mathcal{P}_{1,\lambda}\otimes_{\Dcal_\lambda}\bullet):\mathcal{O}^{loc}_\lambda\rightarrow \mathcal{O}_c$
is a category equivalence. The following result answers this question.

\begin{Thm}\label{Thm:ab_loc}
Suppose that there is an order $\leqslant$ on $\operatorname{Irr}(\Gamma_n)$ refining $\leqslant^\theta$
 and making \underline{both}
$\mathcal{O}^{loc}_\lambda, \mathcal{O}_c$  into highest weight categories. Then $\Gamma:\operatorname{Coh}(\mathcal{D}_\lambda)\rightarrow H_{1,c}\operatorname{-mod}, \mathcal{O}^{loc}_\lambda\rightarrow \mathcal{O}_c$ are equivalences of categories.
\end{Thm}

This theorem can be viewed as an analog of the Beilinson-Bernstein localization theorem, \cite{BB_loc},
from the Lie representation theory.

\begin{proof}[Sketch of proof]
It is enough to prove that $\Gamma$ gives an equivalence between the categories $\mathcal{O}$,
see \cite[Section 3.3]{cher_ab_loc}. So below in the proof we only deal with the categories $\mathcal{O}$.

Set $\Delta^{loc}(\lambda):=[\mathcal{P}_{1,\lambda}^*/\mathcal{P}_{1,\lambda}^*\h]e_\lambda$.
Further, let $\mathcal{F}$ stand for $R\Gamma(\mathcal{P}_{1,\lambda}\otimes_{\Dcal_\lambda}\bullet)$.
The flatness of $\mathcal{P}$ over $S(\h)$ from the previous subsection implies that
\begin{equation}\label{eq:images}
\mathcal{F}\Delta_\lambda^{loc}(\tau)=\Delta_c(\tau).
\end{equation}
We have $\Delta_\lambda^{loc}(\tau)\in \mathcal{O}^{loc}_{\leqslant^\theta \lambda}$.
The condition on the orders implies that $\Delta_\lambda^{loc}(\tau)$ is the standard object in
$\mathcal{O}^{loc}_\lambda$.
Now the claim of Theorem \ref{Thm:ab_loc} follows from the next general claim.
\end{proof}

\begin{Lem}
Let $\mathcal{C}^1,\mathcal{C}^2$ be two abelian categories  with the same indexing poset $\mathcal{T}$.
Suppose that $\mathcal{F}:D^b(\mathcal{C}^1)\rightarrow D^b(\mathcal{C}^2)$ is a derived equivalence
mapping $\Delta^1(\tau)$ to $\Delta^2(\tau)$ for any $\tau\in \mathcal{T}$. Then $\mathcal{F}$
is induced from an abelian equivalence $\mathcal{C}^1\rightarrow \mathcal{C}^2$.
\end{Lem}

Theorem \ref{Thm:ab_loc} generalizes  results of \cite{GS,KR} for $\Gamma_1=\{1\}$
to the case of general cyclic $\Gamma_1$.


\end{document}